\documentclass[a4paper,12pt,reqno]{article}
\usepackage{amsmath,amsfonts,amssymb,amsthm,bbm}
\usepackage{stmaryrd}
\usepackage{graphics,epsfig} 
\usepackage{paralist,subfigure}
\usepackage{hyperref} 
\usepackage[active]{srcltx} 
\usepackage{color,soul} 
\usepackage[all]{xy}

\newtheorem{theorem}{Theorem}[section]
\newtheorem{corollary}[theorem]{Corollary}
\newtheorem{lemma}[theorem]{Lemma}
\newtheorem{proposition}[theorem]{Proposition}

\theoremstyle{definition}
\newtheorem{definition}[theorem]{Definition}
\newtheorem{remark}{Remark}

\newcommand{\baa}{\begin{array}}
\newcommand{\eaa}{\end{array}}
\newcommand{\ba}{\begin{eqnarray}}
\newcommand{\ea}{\end{eqnarray}}
\newcommand{\be}{\begin{equation}}
\newcommand{\ee}{\end{equation}}

\newcommand{\Bk}{\color{black}}

\def\be{\begin{equation}}
\def\ee{\end{equation}}
\def\bes{\begin{equation*}}
\def\ees{\end{equation*}}

\def\bt{\begin{theorem}}
\def\et{\end{theorem}}
\def\bpr{\begin{proposition}}
\def\epr{\end{proposition}}
\def\dsp{\displaystyle}
\def\N{\mathbb{N}}

\def\R{\mathbb{R}}
\def\C{\mathbb{C}}
\def\L{\Lambda T^{\ast}M}
\def\L1{\Lambda^1 T^{\ast}M}
\def\L0{\Lambda^0 T^{\ast}M}
\def\bp{\begin{proof}}
\def\ep{\end{proof}}
\def\calS{{\mathcal S}}
\def\calQ{{\mathcal Q}}

\def\epsilon{\varepsilon}

\let\.=\cdot
\let\0=\emptyset

\def\1{\mathbbm{1}}

\newenvironment{formula}[1]{\begin{equation}\label{#1}}{\end{equation}\noindent}

\def\Fi#1{\begin{formula}{#1}}
\def\Ff{\end{formula}\noindent}

\newcommand{\cqfd}{\begin{flushright}                  
			 $\Box$
                 \end{flushright}}

\begin{document}

\title{{\bf Hardy spaces on Riemannian manifolds with quadratic curvature decay}}\author{Baptiste Devyver$^{\,\hbox{\small{a,b}}}$ and Emmanuel Russ$^{\,\hbox{\small{b}}}$\\
\\
\footnotesize{$^{\,\hbox{\small{a}}}$ Department of Mathematics, Technion, Israel Institute of Technology, Haifa, 32000, Israel.}\\
\footnotesize{$^{\,\hbox{\small{b}}}$ Universit\'e Grenoble Alpes, Institut Fourier}\\
\footnotesize{100 rue des maths, BP 74, 38402 Saint-Martin d'H\`eres Cedex, France.}}
\date{}

\maketitle

\begin{abstract}
Let $(M,g)$ be a complete Riemannian manifold. Assume that the Ricci curvature of $M$ has quadratic decay and that the volume growth is strictly faster than quadratic. We establish that the Hardy spaces of exact $1$-differential forms on $M$, introduced in \cite{amr}, coincide with the closure in $L^p$ of ${\mathcal R}(d)\cap L^p(\Lambda^1T^{\ast}M)$ when $\frac{\nu}{\nu-1}<p<\nu$, where $\nu>2$ is related to the volume growth. Here and after, $\mathcal{R}(d)$ denotes the range of $d$ as an unbounded operator from $L^2$ to $L^2(\Lambda^1 T^\ast M)$. The range of $p$ is optimal. This result applies, in particular, when $M$ has a finite number of Euclidean ends. 
\end{abstract}
\tableofcontents

\section{Introduction}
\subsection{Motivation}

Let $n\ge 1$ be an integer. The notation $\llbracket 1,n\rrbracket$ denotes the set of integers $j$ satisfying $1\le j\le n$. It is a well-known fact that, for all $j\in \llbracket 1,n\rrbracket$, the Riesz transform $\partial_j(-\Delta)^{-1/2}$ is $L^p(\R^n)$-bounded for all $p\in (1,\infty)$ and $H^1(\R^n)$-bounded, where $H^1(\R^n)$ denotes the real Hardy space. If one seeks for a version of this result in a complete Riemannian manifold $(M,g)$ endowed with its Riemannian measure $\mu$, one has to take into account that the Riesz transform, given by $d\Delta^{-1/2}$ in this context, is $1$-form valued. Motivated by this observation and relying on the connection between Hardy spaces and tent spaces (\cite{cms}), Auscher, McIntosh and the second author introduced, in \cite{amr}, a family of Hardy spaces of exact (resp. co-exact) differential forms on $M$, namely $H^p_d(\Lambda^k T^{\ast}M)$ (resp. $H^p_{d^{\ast}}(\Lambda^k T^{\ast}M)$) for $1\le p\le \infty$ and $0\le k\le \mbox{dim }M$. Denote $\Delta_k=dd^*+d^*d$ the Hodge Laplacian acting on differential forms of degree $k$; in particular, $\Delta_0=\Delta$, the usual Laplacian acting on scalar functions on $M$. In \cite{amr}, it was proved that, under a doubling volume condition for geodesic balls of $M$, the Riesz transform $d\Delta_k^{-1/2}$ acting on exact differential $k$-forms is bounded from $H^p_{d^{\ast}}(\Lambda^k T^{\ast}M)$ to $H^p_d(\Lambda^{k+1} T^{\ast}M)$ for all $k\in \llbracket 0,\mbox{dim }M-1\rrbracket$ and all $1\le p\le \infty$.\par
 With the issue of $L^p$-boundedness of the Riesz transform in mind, one may wonder if $H^p_d(\Lambda^kT^{\ast}M)$ coincides with the 
closure in $L^p$ of ${\mathcal R}(d)\cap L^p(\Lambda^kT^{\ast}M)$ for $1<p<\infty$, as well as the corresponding statement for $H^p_{d^{\ast}}(\Lambda^kT^{\ast}M)$, as in the Euclidean case. In the case of $0$-forms (that is, for functions), it was proved in \cite[Theorem 8.5]{amr} that the answer is positive for $H^p_{d^{\ast}}(\Lambda^0T^{\ast}M)$ if the heat kernel associated with the Laplace-Beltrami operator satisfies Gaussian pointwise upper estimates. A similar statement holds for $H^p_d(\Lambda^1T^{\ast}M)$ if one assumes analogous Gaussian bounds for the heat kernel associated with $\Delta_1$, the Hodge Laplacian on $1$-forms; this is however a much stronger assumption. In particular, it implies the $L^p$-boundedness of $d\Delta^{-1/2}$ for all $1<p<+\infty$ (\cite{cdgeq2,s}). Gaussian bounds for the heat kernels associated with $\Delta_0$ and $\Delta_1$ hold, in particular, if $(M,g)$ has nonnegative Ricci curvature (\cite{LY,bakry1,bakry2}). In the present work we want to compare $H^p$ and $L^p$, avoiding the use of Gaussian bounds for the heat kernel on $1$-forms. A general fact proved in \cite{amr} is that if the measure $\mu$ is doubling, then for all $p\geq 2$ and all $0\leq k\leq \mathrm{dim}\, M$, the 
closure in $L^p$ of ${\mathcal R}(d)\cap L^p(\Lambda^kT^{\ast}M)$ is included in $H^p_{d}(\Lambda^kT^{\ast}M)$. But the inclusion may be strict, as the following example demonstrates: consider the manifold $M$ made of the connected sum of two copies of $\R^n$. It is well-known that the heat kernel of $\Delta_0$ on $M$ has Gaussian estimates, but that the heat kernel of $\Delta_1$ does not, despite $M$ having vanishing Riemannian curvature outside a compact set. If $n\geq 3$ (resp. $n=2$), it was proved in \cite{CCH} that $d\Delta^{-1/2}$ is $L^p$-bounded if and only if $1<p<n$ (resp. $1<p\leq 2$).  and it follows that, on $M$, $H^p_d(\Lambda^1T^{\ast}M)$ and the closure in $L^p$ of ${\mathcal R}(d)\cap L^p(\Lambda^1T^{\ast}M)$ never coincide if $p\ge n$ (resp. $p>2$). However, as a consequence of the main result in the present paper, we shall prove that for the connected sum of two copies of $\R^n$, $n\geq3$, $H^p_d(\Lambda^1T^{\ast}M)$ is \textit{equal} to the closure in $L^p$ of ${\mathcal R}(d)\cap L^p(\Lambda^1T^{\ast}M)$ for all $p\in \left(\frac{n}{n-1},n\right)$. Thus, in this particular example, $H^p_d(\Lambda^1T^{\ast}M)$ is equal to the closure in $L^p$ of ${\mathcal R}(d)\cap L^p(\Lambda^1T^{\ast}M)$, if and only if $p\in \left(\frac{n}{n-1},n\right)$. \par
 More generally, following \cite{carron}, we consider complete Riemannian manifolds $(M,g)$ with a quadratic decay of the Ricci curvature, and, under suitable assumptions on the volume growth of balls in $M$, we prove that $H^p_d(\Lambda^1T^{\ast}M)$ and the closure in $L^p$ of ${\mathcal R}(d)\cap L^p(\Lambda^1T^{\ast}M)$ coincide for $\frac{\nu}{\nu-1}<p<\nu$, where $\nu$ is an exponent related to the volume growth of balls in $M$. In particular, if $n:=\mbox{dim }M>2$ and $M$ has a finite number of Euclidean ends, the conclusion holds with $\nu=n$. Moreover, in the latter situation, we also prove that, for $p\ge n$, the closure in $L^p$ of ${\mathcal R}(d)\cap L^p(\Lambda^1T^{\ast}M)$ is a strict subspace of $H^p_d(\Lambda^1T^{\ast}M)$, unless $M$ has only one end, in which case the two spaces are equal.
 
\subsection{The geometric context}

Throughout the paper, if two quantities $A(f)$ and $B(f)$ depend on a function $f$ ranging over some space $L$, the notation $A(f)\lesssim B(f)$ means that there exists $C>0$ such that $A(f)\leq CB(f)$ for all $f\in L$. Moreover, $A(f)\simeq B(f)$ means that 
$A(f)\lesssim B(f)$ and $B(f)\lesssim A(f)$. \par
\noindent Let $(M,g)$ be a complete Riemannian manifold. Denote by $\mu$ the Riemannian measure and by $d$ the Riemannian distance. For all $x\in M$ and all $r>0$, $B(x,r)$ stands for the open geodesic ball with center $x$ and radius $r$, and $V(x,r):=\mu(B(x,r))$. We assume that the measure $\mu$ is doubling: for all $x\in M$ and all $r>0$,
\begin{equation} \label{DV} \tag{D}
V(x,2r)\lesssim V(x,r).
\end{equation}
By iteration, this condition implies at once that there exists $D>0$ such that for all $x\in M$ and all $0<r<R$,

\begin{equation}\label{VD}\tag{VD}
V(x,R)\lesssim \left(\frac{R}{r}\right)^DV(x,r).
\end{equation}
We also consider a reverse doubling volume condition: there exists $\nu>0$ such that, for all $x\in M$ and all $0<r<R$,
\begin{equation} \label{reverseDV}\tag{RD}
\left(\frac Rr\right)^{\nu}V(x,r) \lesssim V(x,R) .
\end{equation}
When $M$ is connected, \eqref{reverseDV} follows from \eqref{DV} (see \cite[Chapter 15, p. 412]{grigobook}); furthermore, the exponent $\nu$ is related to lower bounds for the $p$-capacity of geodesic balls, see \cite[Theorem 5.6]{Dev2}.  \par
\noindent Fix $o\in M$ and set $r(x):=d(o,x)$ for all $x\in M$. We make the following assumption on the Ricci curvature of $M$: denoting $\mathrm{Ric}_x$ the Ricci tensor at the point $x$ and $g_x$ the Riemannian metric at $x$, we assume that there is $\eta\geq0$ such that
\begin{equation} \label{eq:QD}\tag{QD}
\mbox{Ric}_x\geq -\frac{\eta^2}{r^2(x)}g_x,\quad\forall x\in M
\end{equation}
in the sense of quadratic forms. We say that a ball $B(x,r)$ is {\em remote} if $r\leq \frac{r(x)}{2}$. A ball $B(o,r)$ will be called {\em anchored}. The assumption \eqref{eq:QD} on the Ricci curvature implies by the Bishop-Gromov theorem and a simple scaling argument that if $B(x,2r)$ is remote, then $V(x,2r)\lesssim V(x,r)$; hence, by \cite[Prop. 4.7]{GSC}, \eqref{DV} holds if and only if $M$ satisfies the so-called {\em volume comparison} condition, which writes as follows: for every $x\in M$,

\begin{equation}\label{eq:VC}\tag{VC}
V(o,r(x))\lesssim V(x,\frac{r(x)}{2}). 
\end{equation}
We notice also (see \cite{carron}) that \eqref{DV} implies that $M$ has a finite number of ends. Moreover, according to \cite{Buser}, \eqref{eq:QD} implies that remote balls satisfy the scale invariant $L^1$ Poincar\'e inequality: if $B$ is remote and has radius $r$ then

\begin{equation}\label{P}\tag{$P_1$}
||f-f_B||_{L^1(B)}\lesssim r||\nabla f||_{L^1(B)},\quad f\in C^\infty(B),
\end{equation}
where $f_B$ denotes the average of $f$ on $B$, that is $f_B:=V(B)^{-1}\int_B f$.

\noindent For $0\le k\le\mbox{dim }M$, as already said, we denote by $\Delta_k=dd^{\ast}+d^{\ast}d$ the Hodge-Laplacian acting on $k$-forms (here $d$ stands for the exterior differential and 
$d^{\ast}$ for its adjoint \footnote{The notation $d$ stands both for the exterior differential and for the Riemannian distance, which will cause no confusion.}). Recall that $-\Delta_k$ generates a holomorphic semigroup on $L^2(\Lambda^kT^{\ast}M)$, and the associated heat kernel, namely the kernel of $e^{-t\Delta_k}$, is denoted by $p^k_t$. One denotes $p_t(x,y)$ the scalar heat kernel, i.e. the kernel of $e^{-t\Delta_0}$. We consider the Gaussian upper-bounds for the heat kernel:

\begin{equation}\label{eq:UE}\tag{UE}
p_t(x,y)\lesssim \frac{1}{V(x,\sqrt{t})}\exp\left(-\frac{d^2(x,y)}{ct}\right),\quad \forall t>0,\,\forall x,y\in M.
\end{equation}
Under \eqref{eq:QD} and \eqref{eq:VC}, according to \cite{carron} there is a simple geometric condition ensuring that \eqref{eq:UE} holds:  

\begin{definition}
{\em 

We say that $(M,g)$ with a finite number of ends satisfies the Relative Connectedness in the Ends {\em (RCE)} condition, if there is a constant $\theta\in (0,1)$ such that for any point $x$ with $r(x)\geq1$, there is a continuous path $c:[0,1]\to M$ satisfying

\begin{itemize}

\item $c(0)=x$.

\item the length of $c$ is bounded by $\frac{r(x)}{\theta}$.

\item $c([0,1])\subset B(o,\theta^{-1}r(x))\setminus B(o,\theta r(x))$.

\item there is a geodesic ray $\gamma:[0,+\infty)\to M\setminus B(o,r(x))$ with $\gamma(0)=c(1)$.

\end{itemize}

}
\end{definition}
In simple words, the condition (RCE) says that any point $x$ in $M$ can be connected to an end by a path staying at distance approximately $r(x)$ from the origin $o$. With this definition, \cite[Theorem 2.4]{carron} asserts that under \eqref{eq:QD}, \eqref{eq:VC} and (RCE), the Gaussian upper-estimate \eqref{eq:UE} for the scalar heat kernel holds.

\subsection{Tent and Hardy spaces} 
Let us briefly recall here the definitions of Hardy spaces of differential forms on $(M,g)$ introduced in \cite{amr}. These definitions rely on tent spaces, which we first present. For all $x\in M$ and $\alpha>0$, the cone 
of aperture $\alpha$ and vertex $x$ is the set
\[
\Gamma_{\alpha}(x)=\left\{(y,t)\in M\times \left(0,+\infty\right);\ 
y\in B(x,\alpha t)\right\}.
\]
When $\alpha=1$, we write $\Gamma(x)$ instead of $\Gamma_{1}(x)$. For any closed set $F\subset M$, let ${\mathcal R}(F)$ 
be the union of all cones with aperture $1$ and vertices in $F$. 
Finally, if $O\subset M$ is an open set and $F=M\setminus O$, the 
tent 
over $O$, denoted by $T(O)$, is the complement of ${\mathcal R}(F)$ 
in $M\times \left(0,+\infty\right)$.  \par
\noindent Let $F=(F_t)_{t>0}$ be a family of measurable functions on $M$. Write $F(y,t):=F_t(y)$ for all $y\in M$ and all $t>0$ and assume that $F$ is measurable on $M\times (0,+\infty)$. Define then, for all $x\in M$,
\[
 \calS F(x)=\left(\iint_{\Gamma(x)} \left\vert F(y,t)\right\vert^2 
\frac{d\mu(y)}{V(x,t)} \frac {dt}t\right)^{1/2}
\]
and, if $1\leq p<+\infty$, say that $F\in T^{p,2}(M)$ if
\[
\left\Vert F\right\Vert_{T^{p,2}(M)}:=\left\Vert \calS 
F\right\Vert_{L^p(M)}<+\infty.
\]
Recall that we denote by $d$ the exterior differentiation and by $d^{\ast}$ its adjoint. Define the following Hardy spaces of forms of degree one, for $p=2$:

$$H^2(\Lambda^1 T^\ast M)=\mathrm{Ker}_{L^2}(\Delta_1)^\perp,$$

$$
H^2_d(\Lambda^1T^{\ast}M):=\overline{\{du\in L^2(\Lambda^1T^{\ast}M);u\in L^2(M)\}}^{L^2}\Bk=\overline{dC_0^\infty(\Lambda^0 T^*M)}^{L^2}\Bk,
$$
and 

$$
H^2_{d^*}(\Lambda^1T^{\ast}M):=\overline{\{d^*u\in L^2(\Lambda^1T^{\ast}M);u\in L^2(\Lambda^2 T^\ast M)\}}^{L^2}=\overline{d^*C_0^\infty(\Lambda^2 T^*M)}^{L^2}.
$$
One has the following orthogonal Hodge decomposition for $p=2$:
\Bk
$$L^2(\Lambda^1 T^* M)=\mathrm{Ker}_{L^2}(\Delta_1)\oplus_\perp H^2(\Lambda^1 T^\ast M)=\mathrm{Ker}_{L^2}(\Delta_1)\oplus_\perp H^2_d(\Lambda^1T^{\ast}M)\oplus_\perp H^2_{d^*}(\Lambda^1T^{\ast}M),$$
and moreover if $M$ is non-parabolic, thus has a Green operator, the Hodge projector $\Pi$ onto exact forms, that is the orthogonal projector onto $H^2_d(\Lambda^1T^{\ast}M)$ in the above orthogonal decomposition, is given by the formula:
\begin{equation} \label{defpi}
\Pi=d\Delta_0^{-1}d^*.
\end{equation}
\begin{remark}
When $M$ is non-parabolic, the operator $\Delta_0^{-1}$ is well-defined, which makes the expression \eqref{defpi} valid. If $M$ is parabolic, it is still possible to write $\Pi=(d\Delta^{-1/2})(d\Delta^{-1/2})^{\ast}$ and $\Pi$ is bounded on $L^2$ (see \cite{Dev4} for details), but \eqref{defpi} is not well-defined. The parabolic situation will not be encountered in the sequel of the present work.
\end{remark}
In this work, we will mainly focus on the scale of Hardy spaces of exact forms, $H^p_d(\Lambda^1T^{\ast}M)$, which we present now in details. Its definition for $p\neq 2$ relies on two operators (\cite[Section 5.3]{amr}):
\begin{definition}
Let $N\ge 0$ be an integer. 
\begin{enumerate}
\item For all $F\in T^{2,2}(M)$, let  
\bes
\calS_{d}^N(F):=\int_0^{+\infty} td (t\Delta)^N e^{-t^2\Delta}F_t\frac{dt}t\in L^2(\Lambda^1T^{\ast}M).
\ees
\item For all $\omega\in L^2(\Lambda^1T^{\ast}M)$ and all $t>0$, let
\bes
(\calQ^N_{d^{\ast}} \omega)_t:=td^{\ast}(t^2\Delta_1)^Ne^{-t^2\Delta_1}\omega\in T^{2,2}(M).
\ees
\end{enumerate}
\end{definition}
\noindent The spectral theorem shows that, for all integers \Bk$N,N^{\prime}\ge 0$\Bk, on $H^2_d(\Lambda^1T^{\ast}M)$, 
\begin{equation} \label{identity}
\calS_d^N\calQ^{N^{\prime}}_{d^{\ast}}=c_{N,N^{\prime}}\mbox{Id}
\end{equation}
for some constant $c_{N,N^{\prime}}>0$.  \par
\noindent We now turn to the definitions of Hardy spaces, starting from Hardy spaces of exact forms:
\begin{definition} \label{Hardyspaces}
Let $N\ge 0$  be an integer and $p\in [1,\infty)$. \Bk If $p>2$, assume moreover that $N\geq \lfloor \frac{D}{2}\rfloor +1$.\Bk
\begin{enumerate}
\item Define
$$
E^p_d(\Lambda^1T^{\ast}M):=\{\omega\in H^2_d(\Lambda^1T^{\ast}M);\ td^{\ast}(t^2\Delta_1)^Ne^{-t^2\Delta_1}\omega\in T^{p,2}(M)\},
$$
equipped with the norm
$$
\left\Vert \omega\right\Vert_{H^p_d(\Lambda^1T^{\ast}M)}=\left\Vert td^{\ast}(t^2\Delta_1)^Ne^{-t^2\Delta_1}\omega\right\Vert_{T^{p,2}(M)}.
$$
\item Let $H^p_d(\Lambda^1T^{\ast}M)$ be the completion of $E^p_d(\Lambda^1T^{\ast}M)$ under the norm $\left\Vert \cdot \right\Vert_{H^p_d(\Lambda^1T^{\ast}M)}$.
\end{enumerate}
\end{definition}
\Bk It turns out that any two $||\cdot||_{H_d^p(\Lambda^1T^*M)}$ norms, defined for two different values of $N$ satisfying the above assumptions, are equivalent, see \cite[Section 5]{amr}; therefore, by a slight abuse of notation, we will just write $||\cdot||_{H_d^p(\Lambda^1T^*M)}$ without mentionning the parameter $N$. \Bk

In some sense, as we already mentioned $H^p_d(\Lambda^1T^{\ast}M)$ is a space of exact forms of degree one. More generally, in \cite{amr}, a scale of Hardy spaces $H^p(\Lambda T^\ast M)$ of forms of any degree, that are not necessarily exact, are defined in a similar fashion. In degree one, the construction is the same, except that the space $E_d^p(\Lambda^1 T^\ast M)$ has to be replaced by 

$$E^p(\Lambda^1 T^\ast M):=\{\omega\in H^2(\Lambda^1T^{\ast}M);\ (t^2\Delta_1)^Ne^{-t^2\Delta_1}\omega\in T^{p,2}(\Lambda^1 T^\ast M)\},$$
where $N\geq 1$ is large enough, endowed with the corresponding norm. Finally, we also mention that Hardy spaces of co-exact forms $H_{d^*}^p(\Lambda^1 T^\ast M)$ can be defined similarly, and one has the following Hodge decomposition in degree one:

$$H^p(\Lambda^1 T^\ast M)=H^p_d(\Lambda^1T^{\ast}M)\oplus H_{d^*}^p(\Lambda^1 T^\ast M),$$
and the sum is topological (see \cite[Theorem 5.14]{amr}). Two other useful facts are the duality of the Hardy spaces, as well as the fact that they interpolate by the complex method. These facts extend to the scale of Hardy spaces of exact (resp. co-exact) forms $H^p_d(\Lambda T^{\ast}M)$ (resp. $H_{d^*}^p(\Lambda T^\ast M)$). More precisely, the duality of Hardy spaces for general forms of degree one is as follows: if $p\in (1,+\infty)$ and $q=p'$ is the conjugate exponent, the pairing 

$$(\omega,\eta)\mapsto \int_M \langle \omega(x),\eta(x)\rangle_x\,d\mu(x),$$
(where $\langle \cdot,\cdot\rangle_x$ denotes the complex inner product in $T_x^{\ast}M$ induced by the Riemannian metric), initially defined on $E^p(\Lambda^1 T^\ast M)\times E^q(\Lambda^1 T^\ast M)\subset L^2(\Lambda^1 T^\ast M)\times L^2(\Lambda^1 T^\ast M)$, extends uniquely to a pairing on $H^p(\Lambda^1 T^\ast M)\times H^q(\Lambda^1 T^\ast M)$, which realizes $H^q(\Lambda^1 T^\ast M)$ as the dual of $H^p(\Lambda^1 T^\ast M)$ (see \cite[Theorem 5.8]{amr}). Since $d\circ d=0=d^*\circ d^*$, the pairing of any element of $E_d^p(\Lambda^1 T^\ast M)$ (resp. $E_{d^*}^p(\Lambda^1 T^\ast M)$) with an element of $E_{d^*}^q(\Lambda^1 T^\ast M)$ (resp. $E_d^q(\Lambda^1 T^\ast M)$) is equal to zero, one concludes that the dual of $H^p_d(\Lambda^1 T^\ast M)$ is $H^q_{d}(\Lambda^1 T^\ast M)$.

\begin{remark}\label{rem:Epd2}

According to \cite[Proposition 5.4]{amr}, the space 

$$\tilde{E}^p_d(\Lambda^1T^{\ast}M):=\{\omega\in \mathcal{R}(d);\ td^{\ast}(t^2\Delta_1)^Ne^{-t^2\Delta_1}\omega\in T^{p,2}(M)\}$$
is dense in $E_d^p(\Lambda^1 T^\ast M)$. Hence, the Hardy space $H^p_d(\Lambda^1T^{\ast}M)$ can equivalently be defined as the completion of $\tilde{E}^p_d(\Lambda^1T^{\ast}M)$.

\end{remark}
The Hardy space $H_d^p(\Lambda^1T^{\ast}M)$ thus defined is an abstract Banach space, however it is not at all clear from this definition that elements of $H_d^p(\Lambda^1T^{\ast}M)$ can be identified {\em bona fide} with elements of a function space, for instance it is not clear whether there is an embedding $H_d^p(\Lambda^1T^{\ast}M)\hookrightarrow L_{loc}^1(\Lambda^1T^{\ast}M)$. For the case $p\in [1,2)$, this issue has been addressed in \cite{AMM} (see Theorem 1.1 therein), filling a gap in \cite[Corollary 6.3]{amr}. Since this is a quite subtle point, let us elaborate a bit more on this, following the concepts introduced in \cite{AMM}. 

\begin{definition}\label{def:inclusion}
Let $X$ and $Y$ be normed spaces. We write $X\subseteq Y$ when $X$ is a subspace of $Y$, with the property that there exists $C>0$ such that $||x||_Y\leq C||x||_X$ for all $x\in X$, and we write $X=Y$ when $X\subseteq Y\subseteq X$. We also write $X\subsetneqq Y$ if $X\subseteq Y$ but $Y\nsubseteq X$.
\end{definition}
An (abstract) completion $(\mathcal{X},\imath)$ of a normed space $X$ consists of a Banach space $\mathcal{X}$ and an isometry $\imath:X\rightarrow\mathcal{X}$ such that $\imath(X)$ is dense in $\mathcal{X}$. Every normed space $X$ has an abstract completion $\mathcal{X}$, defined as the set of all Cauchy sequences in $X$, quotiented by the following equivalence relation: two Cauchy sequences $(x_n)_{n\in\mathbb{N}}$ and $(y_n)_{n\in\mathbb{N}}$ in $X$  are identified provided that $\left\Vert x_n-y_n \right\Vert_X\to 0$. However, if one wants to realize the completion $\mathcal{X}$ as a function space, this abstract construction is useless. It is convenient to formalise the following related notion.

\begin{definition}\label{def: completion}
Let $X$ be a normed space and suppose that $X\subseteq Y$ for some Banach space $Y$. A Banach space $\widetilde{X}$ is called \textit{the completion of $X$ in $Y$} when $X\subseteq\widetilde{X}\subseteq Y$, the set $X$ is dense in~$\widetilde{X}$, and  $\|x\|_X = \|x\|_{\widetilde{X}}$ for all $x\in X$. 
\end{definition}  
It is easily checked that the completion $\widetilde{X}$ of $X$ in $Y$ is unique whenever it exists. Moreover, the set $\widetilde{X}$ consists of all $x$ in $Y$ for which there is a Cauchy sequence $(x_n)_n$ in $X$ such that $(x_n)_n$ converges to $x$ in $Y$, and the norm $\|x\|_{\widetilde{X}} = \lim_{n\rightarrow\infty} \|x_n\|_{X}$. Let us recall the following characterization from \cite[Proposition 2.2]{AMM}:

\begin{proposition}\label{pro:completion}
Let $X$ be a normed space and suppose  that $X\subseteq Y$ for some Banach space $Y$, so the identity $I:X\rightarrow Y$ is bounded. The following are equivalent:
\begin{enumerate}
\item the completion of $X$ in $Y$ exists;
\item if $(\mathcal{X},\imath)$ is a completion of $X$, then the unique operator $\widetilde{I}$ in $\mathcal{L}(\mathcal{X},Y)$ defined by the commutative diagram below, is injective;
\[\label{fig1}
\xymatrix{
X \ar[d]_-{{\imath}} \ar[r]^-{I}
& Y\\
\mathcal{X}  \ar[ur]_-{\widetilde{I}}
}\]
\item for each Cauchy sequence $(x_n)_n$ in $X$ that converges to $0$ in $Y$, it follows that $(x_n)_n$ converges to $0$ in $X$.
\end{enumerate}
\end{proposition}
It has been established in the proof of \cite[Corollary 6.3]{amr} that under the assumption \eqref{DV}, for any $p\in [1,2)$, one has $E^p_d(\Lambda^1T^{\ast}M)\subseteq L^p(\Lambda^1T^\ast M)$. Then, in \cite[Theorem 1.1]{AMM}, it has been proved that under \eqref{DV} and for any $p\in [1,2)$, the completion $H_d^p(\Lambda^1 T^\ast M)$ of $E^p_d(\Lambda^1T^{\ast}M)$ (or equivalently, of $\tilde{E}_d^p(\Lambda^1 T^\ast M)$) in $L^p(\Lambda^1 T^\ast M)$ exists; \Bk in fact, since $\tilde{E}_d^p(\Lambda^1 T^\ast M)$ is included as a set in $\mathcal{R}(d)$, it follows that the completion of $\tilde{E}_d^p(\Lambda^1 T^\ast M)$ exists in $\overline{\mathcal{R}(d)\cap L^p(\Lambda^1 T^\ast M)}^{L^p}$, which one writes shortly as $H_d^p(\Lambda^1 T^\ast M)\subseteq \overline{\mathcal{R}(d)\cap L^p(\Lambda^1 T^\ast M)}^{L^p}$, $p<2$. In particular, the following inequality holds for $p<2$:

\begin{equation}\label{eq:p<2}
||\omega||_{L^p(\Lambda^1 T^\ast M)}\lesssim ||\omega||_{H_d^p(\Lambda^1 T^\ast M)}=||td^\ast e^{-t^2\Delta_1}\omega||_{T^{p,2}},\quad \forall \omega\in E_d^p(\Lambda^1 T^\ast M)
\end{equation}
The converse inequality is then equivalent to $H_d^p(\Lambda^1 T^\ast M)= \overline{\mathcal{R}(d)\cap L^p(\Lambda^1 T^\ast M)}^{L^p}$; to stress this latter fact, we single it out in a seperate lemma:

\begin{lemma}\label{lem:p<2}

 Assume that $M$ satisfies \eqref{DV} and has Ricci curvature bounded from below. Let $p<2$, then the inequality

\begin{equation}\label{eq:conv-p<2}
||td^\ast e^{-t^2\Delta_1}\omega||_{T^{p,2}}\lesssim ||\omega||_{L^p(\Lambda^1 T^\ast M)},\quad \forall \omega\in H_d^2(\Lambda^1 T^*M)\cap L^p(\Lambda^1 T^\ast M).
\end{equation}
is equivalent to $H_d^p(\Lambda^1 T^\ast M)=\overline{H_d^2(\Lambda^1 T^*M)\cap L^p(\Lambda^1 T^\ast M)}^{L^p}=\overline{\mathcal{R}(d)\cap L^p(\Lambda^1 T^\ast M)}^{L^p}$, in the sense of Definition \ref{def: completion}.

\end{lemma}

\begin{proof}

As mentioned above, one always has the inclusion 

$$H_d^p(\Lambda^1 T^\ast M)\subseteq\overline{\mathcal{R}(d)\cap L^p(\Lambda^1 T^\ast M)}^{L^p}.$$
Assume \eqref{eq:conv-p<2}, it remains to prove the converse inclusion. Notice also that by definition $\mathcal{R}(d)\subset H_d^2$, so

$$\overline{\mathcal{R}(d)\cap L^p(\Lambda^1 T^\ast M)}^{L^p}\subset \overline{ H_d^2(\Lambda^1 T^*M)\cap L^p(\Lambda^1 T^\ast M)}^{L^p}.$$
Let $\omega\in H_d^2(\Lambda^1 T^*M)\cap L^p(\Lambda^1 T^\ast M)$. By \eqref{eq:conv-p<2}, we get that $\omega\in H^p_d(\Lambda^1T^{\ast}M)$ with 
\begin{equation} \label{omegahpd}
\left\Vert \omega\right\Vert_{H^p_d(\Lambda^1T^{\ast}M)}\lesssim \left\Vert \omega\right\Vert_p.
\end{equation}
If $\omega\in \overline{H_d^2(\Lambda^1 T^*M)\cap L^p(\Lambda^1 T^\ast M)}^{L^p}$, there exists a sequence $(\omega_k)_{k\ge 1}\in  H_d^2(\Lambda^1 T^*M)\cap L^p(\Lambda^1 T^\ast M)$ converging to $\omega$ in $L^p(\Lambda^1T^{\ast}M)$. Hence, \eqref{eq:conv-p<2} entails that $(\omega_k)_{k\ge 1}$ is a Cauchy sequence in $H^p_d(\Lambda^1T^{\ast}M)$. As a consequence, this sequence converges in $H^p_d(\Lambda^1T^{\ast}M)$, therefore also in $L^p(\Lambda^1T^{\ast}M)$. Thus, $\omega\in H^p_d(\Lambda^1T^{\ast}M)$ and therefore $\overline{H_d^2(\Lambda^1 T^*M)\cap L^p(\Lambda^1 T^\ast M)}^{L^p}\subseteq H_d^p(\Lambda^1 T^\ast M)$ holds. Finally, we have proved that $H_d^p(\Lambda^1 T^\ast M)=\overline{H_d^2(\Lambda^1 T^*M)\cap L^p(\Lambda^1 T^\ast M))}^{L^p}$, and it follows that

$$H_d^p(\Lambda^1 T^\ast M)=\overline{H_d^2(\Lambda^1 T^*M)\cap L^p(\Lambda^1 T^\ast M)}^{L^p}=\overline{\mathcal{R}(d)\cap L^p(\Lambda^1 T^\ast M)}^{L^p}.$$
Conversely, assume that $H_d^p(\Lambda^1 T^\ast M)=\overline{H_d^2(\Lambda^1 T^*M)\cap L^p(\Lambda^1 T^\ast M)}^{L^p}$. Then,

$$||\omega||_{H_d^p(\Lambda^1T^*M)}=||td^\ast e^{-t^2\Delta_1}\omega||_{T^{p,2}}\simeq ||\omega||_{L^p(\Lambda^1 T^\ast M)},\quad \forall \omega\in E^p_d(\Lambda^1 T^\ast M).$$
Let $\omega\in H_d^2(\Lambda^1 T^*M)\cap L^p(\Lambda^1 T^\ast M)$, then the assumed equality $H_d^p(\Lambda^1 T^\ast M)=\overline{H_d^2(\Lambda^1 T^*M)\cap L^p(\Lambda^1 T^\ast M)}^{L^p}$ implies that $\omega\in H_d^p(\Lambda^1 T^\ast M)$, so that there is a sequence $(\omega_n)_{n\in \N}$ in $E^p_d(\Lambda^1 T^\ast M)$, such that $||\omega_n-\omega||_{H_d^p(\Lambda^1 T^*M)}\simeq ||\omega_n-\omega||_{L^p(\Lambda^1 T^*M)}\to 0$. In particular, this implies that $(td^\ast e^{-t^2\Delta_1}\omega_n)_{n\in \N}$ is a Cauchy sequence in $T^{p,2}$, hence converges to $F\in T^{p,2}$. We also have $||F||_{T^{p,2}}\simeq ||\omega||_{L^p(\Lambda^1 T^*M)}$, therefore, in order to prove \eqref{eq:conv-p<2}, it is enough to prove that $F=td^\ast e^{-t^2\Delta_1}\omega$. Denote $F_n=td^\ast e^{-t^2\Delta_1}\omega_n$. We use the duality between $T^{p,2}(\Lambda^0 T^*M)$ and $T^{q,2}(\Lambda^0 T^*M)$, $q=p^\prime$: let $G\in T^{q,2}(\Lambda^0T^*M)$, then

\begin{eqnarray*}
\int_{M\times (0,+\infty)} F(t,x)G(t,x) \,d\mu(x)\frac{dt}{t}&=& \lim_{n\to\infty} \int_{M\times (0,+\infty)} F_n(t,x)G(t,x) \,d\mu(x)\frac{dt}{t}\\
&=& \lim_{n\to\infty} \int_0^\infty \mathcal{Q}^0_{d^*}\omega_n G\, \frac{dt}{t}\\
&=& \lim_{n\to\infty} \int_M \langle \omega_n,\mathcal{S}^0_{d}G\rangle\,d\mu(x)
\end{eqnarray*}
Let us assume now that $G$ is a smooth, compactly supported function in $M\times (0,+\infty)$. We claim that $\mathcal{S}^0_{d}G\in L^q(\Lambda^1T^{\ast}M)$. Indeed, observe first that, since $M$ satisfies \eqref{DV} and has Ricci curvature bounded from below, there exist $C,c>0$ such that, for all $t\in (0,1)$ and all $x,y\in M$,
\begin{equation} \label{gradpt}
\left\vert \nabla _xp_t(x,y)\right\vert\le \frac{C}{\sqrt{t}V(x,\sqrt{t})} \exp\left(-c\frac{d^2(x,y)}t\right).
\end{equation}
Indeed, arguing as in \cite[p. 488]{Dev2}, one deduces from the lower bound for the Ricci curvature of $M$ that, for all $t\in (0,1)$ and all $x,y\in M$,
\begin{equation} \label{domination}
\left\vert \overrightarrow{p}_t(x,y)\right\vert \lesssim p_t(x,y),
\end{equation}
where $\overrightarrow{p}_t$ denotes the keat kernel on $1$-forms. Inequality \eqref{domination} and \cite[Theorem 16]{davies} entail the existence of $C,c>0$ such that, for all $t\in (0,1)$ and all $x,y\in M$,
\begin{equation} \label{Gaussform}
\left\vert \overrightarrow{p}_t(x,y)\right\vert\le\frac{C}{V(x,\sqrt{t})} \exp\left(-c\frac{d^2(x,y)}t\right).
\end{equation}
Finally, arguing as in the proof of \cite[(5.7)]{cdgeq2}, one derives from \eqref{DV}, \cite[Theorem 16]{davies} and \eqref{Gaussform} that \eqref{gradpt} holds for all $t\in (0,1)$.\par
Thus, \eqref{gradpt}, as well as the compactness of the support of $G$, show that, for all $x\in M$,
$$
\mathcal{S}^0_{d}G(x)\lesssim \mathcal{M}G(x),
$$
where $\mathcal M$ stands for the (uncentered) Hardy-Littlewood maximal function. In turn, this implies that $\mathcal{S}^0_{d}G\in L^q(M)$, as claimed. Using that $\omega_n\to\omega$ in $L^p$, we conclude that

$$\int_{M\times (0,+\infty)}\langle F(t,x),G(t,x)\rangle \,d\mu(x)\frac{dt}{t}=\int_M \langle \omega,\mathcal{S}^0_{d}G\rangle\,d\mu(x),$$
for all $G$ is a smooth, compactly supported function in $M\times (0,+\infty)$. Since the set of such $G$'s is dense in $T^{q,2}(\Lambda^1 T^*M)$, we obtain that $F=(\mathcal{S}^0_{d})^*\omega=\mathcal{Q}^0_{d^*}\omega =td^\ast e^{-t^2\Delta_1}\omega$.

\end{proof}
Let us now turn to the case $p>2$. In this case, as far as we know, it is still unknown whether under the assumption \eqref{DV} only, the completion of $E^p_d(\Lambda^1T^{\ast}M)$ in $L_{loc}^1(\Lambda^1T^{\ast}M)$ exists. Denote $F^p_d(\Lambda^1 T^\ast M):=H^2_d(\Lambda^1 T^\ast M)\cap L^p(\Lambda^1 T^\ast M)$, endowed with the $L^p$ norm. The proof of \cite[Corollary 6.3, (b)]{amr} shows that $F_d^p(\Lambda^1 T^\ast M)\subseteq E_d^p(\Lambda^1 T^\ast M)$. This means in particular that the following inequality holds:

\begin{equation}\label{eq:Lp_inclusion}
||\omega||_{H^p_d(\Lambda^1 T^\ast M)}\lesssim ||\omega||_{L^p(\Lambda^1 T^\ast M)},\quad \forall \omega \in H^2_d(\Lambda^1 T^\ast M)\cap L^p(\Lambda^1 T^\ast M).
\end{equation}
Clearly, by item 3 in Proposition \ref{pro:completion}, the completion of $F_d^p(\Lambda^1 T^\ast M)$ in $L^p(\Lambda^1 T^\ast M)$  therefore exists; in what follows, this completion will be denoted by 

$$\overline{H_d^2(\Lambda^1 T^\ast M)\cap L^p(\Lambda^1 T^\ast M)}^{L^p},$$
or $\overline{F_d^p(\Lambda^1 T^\ast M)}^{L^p}$, or even just $\overline{H_d^2(\Lambda^1 T^\ast M)\cap L^p(\Lambda^1 T^\ast M)}$ to keep notations reasonnably light, when there is no possible confusion. The inclusion $F^p_d(\Lambda^1 T^\ast M)\subseteq E_d^p(\Lambda^1 T^\ast M)$ therefore extends uniquely to an inclusion:
$$\overline{F^p_d(\Lambda^1 T^\ast M)}^{L^p}\subseteq H^p_d(\Lambda^1T^*M).$$
In light of Remark \ref{rem:Epd2}, one can alternatively define the Hardy space $H^p_d(\Lambda^1T^*M)$ as the completion of $\tilde{E}_d^p(\Lambda^1 T^\ast M)$, and the proof of \cite[Corollary 6.3, (b)]{amr} shows that 

$$\mathcal{R}(d)\cap L^p(\Lambda^1 T^\ast M)\subseteq \tilde{E}_d^p(\Lambda^1 T^\ast M),$$
and this inclusion also extends uniquely to an inclusion:

$$\overline{\mathcal{R}(d)\cap L^p(\Lambda^1 T^\ast M)}^{L^p}\subseteq H^p_d(\Lambda^1T^*M).$$
In the sequel, we will be interested in the converse inclusion.

\subsection{Statement of the results}
With these definitions settled, our main result states as follows:

\begin{theorem} \label{comparhardy}
Assume that $(M,g)$ satisfies \eqref{eq:QD}, \eqref{eq:VC}, {\em (RCE)}  and \eqref{reverseDV} with some $\nu>2$. Then, for all $p\in (\frac{\nu}{\nu-1},\nu)$, the completion $H_d^p(\Lambda^1T^{\ast}M)$ of $E^p_d(\Lambda^1T^{\ast}M)$ in $L^p(\Lambda^1T^{\ast}M)$ exists, and moreover this completion satisfies

$$H_d^p(\Lambda^1T^{\ast}M)=\overline{\mathcal{R}(d)\cap L^p(\Lambda^1 T^\ast M)}^{L^p},$$
in the sense of Definition \ref{def:inclusion}.

\end{theorem}

As a corollary, we are able to completely characterize the Hardy space $H^p_d(\Lambda^1 T^{\ast} M)$  for $p\in \left(\frac{n}{n-1},n\right)$, as well as for $p\ge n$ in some cases, where  $n=\mathrm{dim}(M)\geq 3$, for all manifolds having Euclidean or conical ends. Recall first that a manifold $M$ with topological dimension $n$ is said to have a finite number of Euclidean ends if there is a compact set $K\Subset M$, an integer $N\in \mathbb{N}$ and positive real numbers $R_1, \cdots,R_N$ such that $M\setminus K$ is isometric to the disjoint union $\bigsqcup_{i=1}^N \R^n\setminus \overline{B(0,R_i)}$. The sets $\R^n\setminus \overline{B(0,R_i)}$ can thus be identifed isometrically with open subsets of $M$, which are called the Euclidean ends of $M$ (of course, the positive numbers $R_i$ are not unique, so we slightly abuse notations here). More generally, in the same fashion one can define the notion of a manifold with  a finite number of {\em conical ends}, where a conical end $E$ is defined as follows: there exists a compact Riemannian manifold $(\Sigma,h)$ and $R>0$, such that $E=(R,+\infty)\times \Sigma$ endowed with the metric $dr^2+r^2h.$ Note that manifolds with conical ends satisfy assumptions \eqref{eq:QD}, \eqref{eq:VC}, (RCE). Moreover, there exists $C>0$ such that, for all $x\in M$ and all $R>0$, 
$$
C^{-1}R^n\le V(x,R)\le CR^n
$$
so that \eqref{reverseDV} holds with $\nu=n$. All these properties are stated in \cite[Section 7.1]{carron}). For manifolds with a finite number of Euclidean ends, our result writes as follows:

\begin{corollary}\label{cor:eucl_ends}

Let $M$ be a complete Riemannian manifold of dimension $n\geq 3$ with a finite number of Euclidean ends, and let $p\in \left(\frac{n}{n-1},+\infty\right)$;

\begin{itemize}
\item if $p\in \left(\frac{n}{n-1},n\right)$, then $H^p_d(\Lambda^1T^{\ast}M)= \overline{\mathcal{R}(d)\cap L^p(\Lambda^1T^{\ast}M)}^{L^p}$. 

\item if $p\in [n,+\infty)$, and $M$ has only one end, then 
$$H^p_d(\Lambda^1T^{\ast}M)=  \overline{\mathcal{R}(d)\cap L^p(\Lambda^1T^{\ast}M)}^{L^p}$$

\item  if $p\in [n,+\infty)$ and $M$ has two or more ends, the completion of $E_d^p(\Lambda^1T^\ast M)$ in $L^p(\Lambda^1 T^\ast M)$ does not exist. The inclusion $\mathcal{R}(d)\cap L^p(\Lambda^1T^{\ast}M)\subseteq E_d^p(\Lambda^1 T^\ast M)$ extends uniquely to an inclusion:
$$\overline{\mathcal{R}(d)\cap L^p(\Lambda^1T^{\ast}M)}^{L^p}\subseteq H^p_d(\Lambda^1T^*M),$$
but the latter inclusion is strict.
\end{itemize}

\end{corollary}

\begin{proof}

The statement for $\frac{n}{n-1}<p<n$ follows from Theorem \ref{comparhardy}.  

It thus remains to discuss the case $p\in [n,+\infty)$. For $p\geq n$, as already seen, the inclusion

$$\overline{\mathcal{R}(d)\cap L^p(\Lambda^1T^{\ast}M)}^{L^p}\subset H^p_d(\Lambda^1T^*M)$$
holds true. Moreover, by \cite[Theorem 5.16]{amr}, the Riesz transform is bounded from $H^p_{d^{\ast}}(\Lambda^0 T^*M)$ to $H^p_d(\Lambda^1T^*M)$. By the argument in \cite[p. 12-13]{Dev},  \eqref{DV} and \eqref{eq:UE} imply that $H^p_{d^{\ast}}(\Lambda^0 T^*M)\simeq L^p(M)$, hence the Riesz transform $d\Delta_0^{-1/2}$ is bounded from $L^p(M)$ to $H^p_d(\Lambda^1T^*M)$. Since it is known that $d\Delta_0^{-1/2}$ is {\em not} bounded on $L^p$, $p\geq n$, in the case $M$ has several Euclidean ends (see \cite[Corollary 7.5]{CCH}), one concludes that in this case, for $p\in [n,+\infty)$,

$$\overline{\mathcal{R}(d)\cap L^p(\Lambda^1T^{\ast}M)}^{L^p}\subsetneqq H^p_d(\Lambda^1T^*M).$$
If $M$ has only one end, (RCE) is the more familiar (RCA) condition (Relative Connectedness of Annuli) from \cite{GSC}, hence by \cite[Corollary 5.4]{GSC} $M$ satisfies the scaled $L^2$ Poincar\'e inequalities. According to \cite{CCH}, the Riesz transform on $M$ is bounded on $L^p$, for every $p\in (1,+\infty)$; hence, by \cite{Dev}, for every $p\in [n,+\infty)$,

$$H^p_d(\Lambda^1T^{\ast}M)\simeq  \overline{\mathcal{R}(d)\cap L^p(\Lambda^1T^{\ast}M)}^{L^p}.$$

\end{proof}
More generally, we can prove a similar result for manifolds with conical ends:

\begin{corollary} \label{cor:conical_ends}

Let $M$ be a complete Riemannian manifold of dimension $n\geq 3$ with a finite number of conical ends. Define a number $p_*$ as follows: $p_*$ is equal to $n$ if $M$ has two ends or more, whereas if $M$ has only one end which is isometric to $[R,+\infty)\times \Sigma$, one lets

$$p_*=\frac{n}{\frac{n}{2}-\sqrt{\lambda_1+\left(\frac{n-2}{2}\right)^2}}>n,$$
where $\lambda_1>0$ is the first non-zero eigenvalue of the Laplacian on $\Sigma$ (by convention, $p_*=+\infty$ if $\lambda_1\geq n-1$). Then, for all \Bk $p\in \left(\frac{n}{n-1},p_*\right)$, \Bk

$$H^p_d(\Lambda^1T^{\ast}M)\simeq  \overline{\mathcal{R}(d)\cap L^p(\Lambda^1T^{\ast}M)}^{L^p},$$
whereas for all $p\in [p_*,+\infty)$, 

$$\overline{\mathcal{R}(d)\cap L^p(\Lambda^1T^{\ast}M)}^{L^p}\subsetneqq H^p_d(\Lambda^1T^*M).$$

\end{corollary}

\begin{remark}
The same result holds for \textit{asymptotically conical} manifolds in the sense of \cite{GH}. 
\end{remark}

\begin{proof}
The proof is the same as for Corollary \ref{cor:eucl_ends}, taking into account that the Riesz transform on $M$ is bounded on $L^p$, if and only if $1<p<p_*$ (see \cite{GH}).
\end{proof}

\begin{remark}
As already seen, the conclusion of Theorem \ref{comparhardy} does not hold when $p\ge \nu$. The validity of this conclusion, as well as the corresponding assertions in Corollaries \ref{cor:eucl_ends} and \ref{cor:conical_ends}, when $1<p\le \frac{\nu}{\nu-1}$ is an open problem.
\end{remark}


\subsection{Strategy of the proof}

\noindent Our strategy is as follows. \Bk We distinguish the cases $p<2$, and $p>2$. For the former, we will use the characterization of Lemma \ref{lem:p<2}, while for $p>2$, one will make use of the following result, whose proof is based on the duality of the Hardy spaces:\Bk

\bpr \label{enough}
\Bk Let $p>2$, and denote $q=p'<2$ the conjugate exponent. Assume that the following inequality is satisfied:

\be \label{tq}
\left\Vert td^{\ast}e^{-t^2\Delta_1}\omega\right\Vert_{T^{q,2}(M)}\lesssim \left\Vert \omega\right\Vert_{q},\quad\forall \omega\in L^{q}(\Lambda^1T^{\ast}M),
\ee
Then, the Hardy space $H^p_d(\Lambda^1T^{\ast}M)$ can be realized as the completion of $E_d^p(\Lambda^1 T^\ast M)$ into $L^p(\Lambda^1 T^\ast M)$. Moreover, this completion satisfies

$$H_d^p(\Lambda^1T^{\ast}M)=\overline{\mathcal{R}(d)\cap L^p(\Lambda^1 T^\ast M)}^{L^p},$$
in the sense of Definition \ref{def:inclusion}.\Bk 
\epr

\begin{proof}

%
Let $\eta\in E^p_{d}(\Lambda^1T^{\ast}M)$ and $\omega\in L^{q}(\Lambda^1T^{\ast}M)\cap L^2(\Lambda^1T^*M)$. \Bk Let $N'\geq \lfloor \frac{D}{2}\rfloor +2$ be an integer\Bk. Then, using the duality pairing between $T^{p,2}$ and $T^{q,2}$ (\cite[Section 5, Theorem 2]{cms}) and \eqref{identity}, we get
$$
\begin{array}{lll}
\dsp \int_M \langle \eta(x),\omega(x)\rangle d\mu(x) & = &  \dsp c\int_M \langle \calS_{d}^0\calQ^{N^{\prime}}_{d^{\ast}}\eta(x),\omega(x)\rangle d\mu(x)  \\
& = & \displaystyle c\iint_{M\times (0,+\infty)} \langle \calQ^{N^{\prime}}_{d^{\ast}}\eta(x),  \calQ_{d^{\ast}}^0 \omega(x)\rangle d\mu(x)\frac{dt}t\\
& \leq & \displaystyle c\left\Vert td^{\ast}(t^2\Delta_1)^{N^{\prime}}e^{-t^2\Delta_1}\eta\right\Vert_{T^{p,2}} \left\Vert td^{\ast}e^{-t^2\Delta_1}\omega\right\Vert_{T^{q,2}}\\
& \lesssim & \displaystyle \left\Vert \eta\right\Vert_{H^p_d(\Lambda^1T^{\ast}M)} \left\Vert \omega\right\Vert_{L^{q}(\Lambda^1T^{\ast}M)},\\
\end{array}
$$
where we have used the hypothesis, as well as Definition \ref{Hardyspaces}.
%
By density of $L^{q}(\Lambda^1T^{\ast}M)\cap L^2(\Lambda^1T^*M)$ in $L^{q}(\Lambda^1T^{\ast}M)$, the above inequality extends to all $\omega\in L^{q}(\Lambda^1T^{\ast}M)$. Dividing both sides by $\left\Vert \omega\right\Vert_{L^{q}(\Lambda^1T^{\ast}M)}$ and taking the supremum in $\omega\neq 0$ belonging to $L^{q}(\Lambda^1T^{\ast}M)$, one obtains that for every $\eta\in E^p_{d}(\Lambda^1T^{\ast}M)$,

\begin{equation}\label{eq:Lp_inclusion2}
||\eta||_{L^p(\Lambda^1T^*M)}\lesssim \left\Vert \eta\right\Vert_{H^p_d(\Lambda^1T^{\ast}M)}.
\end{equation}
In particular, $E_d^p(\Lambda^1T^\ast M)\subseteq F_d^p(\Lambda^1 T^\ast M)$, and $\tilde{E}_d^p(\Lambda^1 T^\ast M)\subseteq \mathcal{R}(d)\cap L^p(\Lambda^1 T^\ast M)$. Since $p>2$, as already mentioned previously (cf the proof of \cite[Corollary 6.3, (b)]{amr}), the converse inclusions hold, and one concludes that $E_d^p(\Lambda^1T^\ast M)=F_d^p(\Lambda^1 T^\ast M)$ and $\tilde{E}_d^p(\Lambda^1T^\ast M)=\mathcal{R}(d)\cap L^p(\Lambda^1 T^\ast M)$ in the sense of Definition \ref{def:inclusion}. By item 3 of Proposition \ref{pro:completion}, the completion of $F_d^p(\Lambda^1 T^\ast M)$ in $L^p(\Lambda^1 T^\ast M)$ exists, hence the same is true for $E_d^p(\Lambda^1 T^\ast M)$, and we conclude that the Hardy space $H^p_d(\Lambda^1 T^\ast M)$ can be realized as the completion of $E_d^p(\Lambda^1T^\ast M)$ in $L^p(\Lambda^1 T^\ast M)$. Moreover, since 

$$\tilde{E}_d^p(\Lambda^1T^\ast M)=\mathcal{R}(d)\cap L^p(\Lambda^1 T^\ast M)$$
and $\tilde{E}_d^p(\Lambda^1T^\ast M)$ is dense in $H^p_d(\Lambda^1 T^\ast M)$, one obtains that

$$H^p_d(\Lambda^1 T^\ast M)=\overline{\mathcal{R}(d)\cap L^p(\Lambda^1 T^\ast M)}^{L^p}.$$

%
%
%
%
%
%
%

\end{proof}
According to Lemma \ref{lem:p<2} and Proposition \ref{enough}, in order to prove Theorem \ref{comparhardy}, it is enough to establish that for all \Bk $p\in \left(\frac{\nu}{\nu-1}, 2\right)$\Bk
\Bk
\be \label{tp12}
\left\Vert td^{\ast}e^{-t^2\Delta_1}\omega\right\Vert_{T^{p,2}(M)}\lesssim \left\Vert \omega\right\Vert_{p},\quad \omega\in L^p(\Lambda^1 T^\ast M).
\ee
\Bk
For $p\in (1,\infty)$ and $N\geq 0$ an integer, we will in fact consider the more general inequality:

\be \label{tp2}
\left\Vert td^{\ast}(t^2\Delta_1)^Ne^{-t^2\Delta_1}\omega\right\Vert_{T^{p,2}(M)}\lesssim \left\Vert \omega\right\Vert_{p},\quad \omega\in L^p(\Lambda^1 T^\ast M).
\ee
We now introduce the inequality \eqref{tp2}, in restriction to {\em exact} forms: \Bk more precisely, for every $ \omega$ in the $L^p$ closure of $H_d^2(\Lambda^1 T^* M)\cap L^p(\Lambda^1 T^\ast M)\cap C^\infty(\Lambda^1 T^*M)$,

\begin{equation}\label{tp3}
\left\Vert td^{\ast}(t^2\Delta_1)^Ne^{-t^2\Delta_1}\omega\right\Vert_{T^{p,2}(M)}\lesssim \left\Vert \omega\right\Vert_{p}.
\end{equation}
The reason for adding the condition that $\omega\in C^\infty(\Lambda^1 T^*M)$ is the following: let $\omega\in H_d^2(\Lambda^1 T^* M)\cap  C^\infty(\Lambda^1 T^*M)$, then by definition of $H_d^2(\Lambda^1 T^*M)$, there is a sequence $(f_n)_{n\in\N}$ of smooth, compactly supported function so that $\omega$ is the $L^2$ limit of $(df_n)_{n\in \N}$; then, since $\omega$ is smooth, according to \cite[Lemma 1.11]{carron2} there is $f\in C^\infty(M)$ such that $\omega=df$. Thus, $\omega$ really is an exact $1$-form. Observe also that, if $N$ is large enough and $p>2$, then \eqref{tp3} always holds. Indeed, if $\omega\in H^2_d\cap L^p$, \eqref{tp3} follows from the fact that $\omega\in H^p_d$. The general case follows by approximation.

The following lemma shows that \eqref{tp2} follows from \eqref{tp3} if the \Bk Hodge projector on exact forms $\Pi=d\Delta_0^{-1}d^\ast$ \Bk is bounded in appropriate Lebesgue spaces:

\begin{lemma}\label{lem:tp3to2}

Let \Bk$p\in (1,\infty)$\Bk, and assume that \Bk Hodge projector on exact forms $\Pi=d\Delta_0^{-1}d^\ast$\Bk is bounded on $L^p$. Then, \eqref{tp3} implies \eqref{tp2}.

\end{lemma}

\begin{proof}

Let $\omega\in C_0^\infty(\Lambda^1 T^*M)$. \Bk Since $\Pi : L^2(\Lambda^1 T^* M)\to H_d^2(\Lambda^1 T^* M)$ is always bounded, and since by assumption $\Pi$ is bounded on $L^p$, it follows that $\Pi\omega\in H_d^2(\Lambda^1 T^* M)\cap L^p(\Lambda^1 T^\ast M)\cap C^\infty(\Lambda^1T^*M)$ with $||\Pi\omega||_p\lesssim ||\omega||_p$ \Bk, therefore by \eqref{tp3} one obtains

$$\left\Vert td^{\ast}(t^2\Delta_1)^Ne^{-t^2\Delta_1}\Pi\omega\right\Vert_{T^{p,2}(M)}\lesssim \left\Vert \Pi\omega\right\Vert_{p}\lesssim ||\omega||_p.$$
Therefore, noticing that $td^{\ast}(t^2\Delta_1)^Ne^{-t^2\Delta_1}\omega= td^{\ast}(t^2\Delta_1)^Ne^{-t^2\Delta_1}\Pi\omega$, \eqref{tp2} holds for any $\omega\in _0^\infty(\Lambda^1 T^*M)$. Since $C_0^\infty(\Lambda^1 T^*M)$ is dense in $L^p(\Lambda^1T^*M)$,  we conclude that \eqref{tp2} holds for any $\omega\in L^p(\Lambda^1 T^\ast M)$. 

\end{proof}
We can now make explicit the connections between all these different inequalities, and the Hardy spaces:

\begin{proposition}\label{pro:equi-ineq-hardy}

Assume that $M$ is a complete Riemannian manifold satisfying \eqref{DV} and \eqref{eq:UE}. Let $r\in (2,\infty)$ and denote $s=r'<2$ the conjugate exponent. Recall that $\Pi$ denotes the Hodge projector onto exact $1$-forms. Then, the following are equivalent (in which the equalities are taken in the sense of Definition \ref{def: completion}):

\begin{enumerate}

\item[(i)] $H_d^s(\Lambda^1 T^* M)=\overline{\mathcal{R}(d)\cap L^s(\Lambda^1 T^*M)}^{L^s}$ and $\Pi$ extends to a bounded operator on $L^s$. 

\item[(ii)] for any $p\in [s,r]$, $H_d^p(\Lambda^1 T^* M)=\overline{\mathcal{R}(d)\cap L^p(\Lambda^1 T^*M)}^{L^p}$.

\item[(iii)] for any $p\in [s,2]$, \eqref{tp3} with $N=0$ is satisfied, and $\Pi$ extends to a bounded operator on $L^s$.

\item[(iv)] the inequality \eqref{tp2} with $N=0$ holds for $p=s$, and $\Pi$ extends to a bounded operator on $L^s$.

\end{enumerate}

\end{proposition}

\begin{proof}
(i) $\Rightarrow$ (ii): first, according to Lemma \ref{lem:p<2}, if $p<2$ then $H_d^p(\Lambda^1 T^* M)=\overline{\mathcal{R}(d)\cap L^p(\Lambda^1 T^*M)}^{L^p}$ is equivalent to \eqref{eq:conv-p<2}. Since $\Pi$ is $L^s$-bounded, Lemma \ref{lem:tp3to2} entails that \eqref{tp2} holds for all $\omega\in L^s$. Since \eqref{tp2} clearly holds for $p=2$, it also holds for all $p\in [s,2]$ by interpolation, which implies in turn that $H_d^p(\Lambda^1 T^* M)=\overline{\mathcal{R}(d)\cap L^p(\Lambda^1 T^*M)}^{L^p}$ for all $p\in [s,r]$ (the case $s\le p\le 2$ follows from Lemma \ref{lem:p<2} while the case $2\le p\le r$ is due to Proposition \ref{enough}). \par
\noindent 


\medskip

\Bk (ii) $\Rightarrow$ (iii): clearly, \eqref{tp3} follows from \eqref{eq:conv-p<2} for $p<2$, which holds according to Lemma \ref{lem:p<2}. Since $M$ satisfies \eqref{DV} and \eqref{eq:UE}, for every $p\in (1,+\infty)$, $H^p(\Lambda^0 T^*M)\simeq L^p(\Lambda^0 T^*M)$. Moreover, by \cite{amr}, the Riesz transform $d\Delta_0^{-1/2}$ is bounded from $H^p(\Lambda^0 T^*M)$ into $H^p_d(\Lambda^1 T^*M)$, hence the Riesz transform is bounded in particular on $L^s$ and $L^r$. Since $\Pi=d\Delta_0^{-1}d^*=(d\Delta_0^{-1/2})(d\Delta_0^{-1/2})^* $, we get that $\Pi$ is bounded on $L^s$. Hence, (iii) holds.

\medskip

(iii) $\Rightarrow$ (iv): follows directly from Lemma \ref{lem:tp3to2} since $\Pi$ is $L^s$-bounded.

\medskip

(iv) $\Rightarrow$ (i): clearly, \eqref{tp2} implies \eqref{eq:conv-p<2}, so (i) follows directly from Lemma \ref{lem:p<2}.

\end{proof}

\noindent The key technical result in this work is the following:

\begin{proposition}\label{pro:tp3}

Assume that $(M,g)$ satisfies \eqref{eq:QD}, \eqref{eq:VC}, {\em (RCE)} and \eqref{reverseDV} for some $\nu>2$. Let $p\in (\frac{\nu}{\nu-1},\nu)$, where $\nu$ is the reverse doubling exponent from \eqref{reverseDV}. Then, \eqref{tp3} \Bk with any $N\in\N$ \Bk holds.

\end{proposition}

%
%
%
%
The remaining of the article will be devoted to the proof of Proposition \ref{pro:tp3}. Assuming for the moment the result of Proposition \ref{pro:tp3}, let us give the proof of Theorem \ref{comparhardy}:

\medskip

\noindent{\em Proof of Theorem \ref{comparhardy}:}  recall (\cite[Theorem A]{carron}) that, under the assumptions of Theorem \ref{comparhardy}, the Riesz transform $d\Delta^{-1/2}$ is $L^q$-bounded for all $q\in (1,\nu)$. This implies, according to Lemma \ref{lem:tp3to2}, that \eqref{tp2} holds. The result then follows from Proposition \ref{enough}.
$\hfill\Box$

\medskip

\noindent In what follows, we establish \eqref{tp3}.\footnote{Our proof of \eqref{tp3} relies on the $L^p$ and the $L^{p^{\prime}}$ boundedness of $d\Delta_0^{-1/2}$ (or alternatively, of the boundedness on $L^{p^\prime}$ of Hodge projector onto exact forms $\Pi$), for $p\in (\frac{\nu}{\nu-1},2)$. A variation on our argument for \eqref{tp3} (with square vertical functionals instead of non-tangential ones) will show that for every $p\in (\frac{\nu}{\nu-1},2)$ and every function $u$, $||d^*\Delta_1^{-1/2}(du)||_p\lesssim ||du||_p$. This is equivalent to $||d^*\Delta_1^{-1/2}\Pi\omega||_p\lesssim ||\Pi\omega||_p$, where $\Pi=d\Delta_0^{-1}d^{\ast}$ is the Hodge projector. It is not clear how to get from this the boundedness of the Riesz transform on $L^{p'}$. It would be more satisfying to recover directly the boundedness of the Riesz transform from our result on $H^p_{d}(\Lambda^1T^*M)$. Note also that it is possible, for asymptotically euclidean manifolds, to establish directly the $L^p$-boundedness of $\Pi$, without going through the continuity of Riesz transforms, see \cite[Lemma 4.5]{KP}. }

Our strategy for \eqref{tp3} is as follows. \Bk First, by density, notice that it suffices to establish \eqref{tp3} for any $\omega\in H_d^2(\Lambda^1 T^* M)\cap L^p(\Lambda^1 T^*M)\cap C^\infty(\Lambda^1 T^*M)$. Since $\omega\in H_d^2(\Lambda^1 T^* M)\cap C^\infty(\Lambda^1 T^*M)$, according to \cite[Lemma 1.11]{carron2}, there existe $f\in C^\infty(M)$ such that $\omega=df$. \Bk Thus, the inequality \eqref{tp3} amounts to
\begin{equation} \label{adup}
\left\Vert {\mathcal A}(df)\right\Vert_p\lesssim \left\Vert  df\right\Vert_p,
\end{equation}
where
\begin{equation} \label{aomega}
{\mathcal A}\omega(x) := \left(\iint_{\Gamma(x)} \left\vert  td^{\ast}(t^2\Delta_1)^Ne^{-t^2\Delta_1}\omega(z)\right\vert^2\frac{d\mu(z)}{V(z,t)}\frac{dt}t\right)^{\frac 12}.
\end{equation}
The spectral theorem implies that

$$\Vert \mathcal{A}\omega \Vert_2\lesssim ||\omega||_2,\quad \omega\in L^2(\Lambda^1T^*M).$$
We express 
$$
td^{\ast}(t^2\Delta_1)^Ne^{-t^2\Delta_1}\omega(z)=\int_M k_t(z,y)\cdot \omega(y)d\mu(y),
$$ 
where $k_t$ is the kernel of $td^{\ast}(t^2\Delta_1)^Ne^{-t^2\Delta_1}$, and plug this expression into \eqref{aomega}. Following ideas of \cite{carron}, we then split the integration domain into three parts, involving different conditions on $t,y,z$. \par
\noindent The first one, called ``long-to-short'', is defined by the conditions $(z,t)\in \Gamma(x)$ and $r(y)\geq \kappa r(z)$.  We establish the part of inequality \eqref{adup} corresponding to this regime thanks to pointwise bounds on $\left\vert k_t\right\vert$, which in turn follow from pointwise Gaussian type bounds on the heat kernel on functions and its gradient. More precisely, we obtain in this way a weak type $(1,1)$ inequality, and the required $L^p$ bound is obtained by interpolation between this weak type $(1,1)$ inequality and a strong type $(2,2)$ inequality. \par
\noindent The second one, called ``short-to-long'', is defined by the conditions $(z,t)\in \Gamma(x)$, $\kappa r(z)>r(y)$ and $d(z,y)\geq \kappa^{-1}r(z)$, and the corresponding part of \eqref{adup} is proved by similar arguments. Note that the part of \eqref{tp3} corresponding to these two regimes holds even if the form $\omega$ is not exact.\par
\noindent The last part of the splitting is the so-called ``diagonal regime'', defined by $(z,t)\in \Gamma(x)$ and $d(z,y)<\kappa^{-1}r(z)$. The proof of the corresponding part in \eqref{adup} is more involved. We use a covering of $M$ by a suitable collection of balls $(B_{\alpha})_{\alpha\in A}$ which are either remote or anchored,  and localize in some sense the operator $\mathcal{A}$ in the balls $B_\alpha$. When $t\ge r_{\alpha}$, a pointwise bound for $\left\vert k_t\right\vert$ is still sufficient. When $t<r_{\alpha}$, we use the fact that $\omega$ is an exact form and, writing $\omega=df$, decompose
$$
\omega= \sum_{\alpha\in A} d(\chi_{\alpha}(f-f_{B_{\alpha}}))-\sum_{\alpha\in A} (f-f_{B_{\alpha}})d\chi_{\alpha}=\sum_{\alpha\in A} df_{\alpha}-\sum_{\alpha\in A}\eta_{\alpha},
$$
where $(\chi_{\alpha})_{\alpha\in A}$ is a special partition of unity associated with the covering $(B_{\alpha})_{\alpha\in A}$. The part corresponding to $df_{\alpha}$ is treated by arguments similar to those used in \cite{Auscher2007}, and relies on $L^1-L^2$ estimates for the heat semigroup of the Hodge-Laplacian acting on {\it exact} $1$-forms (see Lemma \ref{off-diag1} below). Roughly speaking, these estimates hold since
$$
e^{-s\Delta_1}du=de^{-s\Delta_0}u
$$
and pointwise estimates on the gradient of the heat kernel on functions can be used again (note that pointwise bounds on the heat kernel on $1$-forms do not hold in the context of the present paper). \par
\noindent Finally, to treat the terms arising from $\eta_{\alpha}$, we write
$$
d^{\ast}e^{-s\Delta_1}\eta_{\alpha}=e^{-s\Delta_0}d^{\ast}\eta_{\alpha},
$$
and we conclude using pointwise bounds for $e^{-s\Delta_0}$, the inequality $\left\vert d\chi_{\alpha}\right\vert\lesssim r_{\alpha}^{-1}$, and the fact that, due to $L^1$ Poincar\'e inequalities on remote balls,
$$
\left\Vert \eta_{\alpha}\right\Vert_{L^1(B_{\alpha})}\lesssim \frac 1{r_{\alpha}} \left\Vert f-f_{B_{\alpha}}\right\Vert_{L^1(B_{\alpha})}\lesssim \left\Vert df\right\Vert_{L^1(B_{\alpha})}.
$$
The paper is organized as follows. Section \ref{prelest} first presents the covering of $M$ by remote  and anchored balls, as well as the associated partition of unity. We also gather (and give proofs for) various pointwise or integrated estimates involving the heat semigroup on functions or $1$-forms. The proof of \eqref{tp3} is presented in Section \ref{prooftp3}, where the three regimes are successively considered. \par
\noindent  {\bf Acknowledgements: } this work was partly supported by the French ANR project RAGE ANR-18-CE40-0012. The authors thank the department of mathematics at the Technion - Israel Institute of Technology and the Institut Fourier at the Grenoble Alpes University for their hospitality. \Bk They also thank the referees for their careful reading and their very interesting remarks, which helped them to improve the paper. \Bk
\section{Preliminary estimates} \label{prelest}

\subsection{A good covering by admissible balls} \label{goodcover}
 For convenience, let us first gather definitions about balls of $M$ (the first two ones were already introduced before):
 \begin{definition} \label{balls}
Let $x\in M$ and $r>0$. 
\begin{enumerate}
\item The ball $B(x,r)$ is called \textit{remote} if $r\leq \frac{r(x)}{2}$,
\item The ball $B(x,r)$ is called \textit{anchored} if $x=o$,
\item The ball $B(x,r)$ is \textit{admissible} if and only if $B$ is remote or $B$ is anchored and $r(B)\le r(B_0)$, where the ball $B_0$ will be defined in the construction of the covering below.
\end{enumerate}
\end{definition}
We now explain how the assumption on the Ricci curvature allows one to construct a good covering of $M$ by remote and anchored balls, as well as a good partition of unity associated to it. Following \cite[Sections 4.3 and 5.3]{carron}, consider a special covering of $M$ by a countable collection of admissible balls $(B_{\alpha})_{\alpha\in \N}$, with the finite overlap property. Let us briefly recall the construction, for the sake of completeness:
\begin{enumerate}
\item define $B_{0,1}:=B(o,1)$,
\item for all integer $N\ge 0$, since
$$
B\left(o,2^{N+1}\right)\setminus B\left(o,2^N\right)\subset \bigcup_{2^N\le r(x)<2^{N+1}} B\left(x,2^{N-13}\right),
$$
the ``$5r$'' covering lemma (\cite[Theorem 1.2]{H}) provides a collection of points $\left(x_{N+1,i}\right)_{i\in I_N}\in B\left(o,2^{N+1}\right)\setminus B\left(o,2^N\right)$, where the set $I_N$ is at most countable, such that the balls $B\left(x_{N+1,i},2^{N-13}\right)$ are pairwise disjoint and 
$$
B\left(o,2^{N+1}\right)\setminus B\left(o,2^N\right)\subset \bigcup_{i\in I_N} B\left(x_{N+1,i},2^{N-10}\right).
$$
Since, for all $i\in I_N$, $B(x_{N+1,i},2^{N-13})\subset B\left(o,2^{N+2}\right)$ and the balls $B(x_{N+1,i},2^{N-13})$ are pairwise disjoint, the doubling property shows that, for all finite subset $J\subset I_N$,
$$
(\sharp J)V(o,2^{N+2})\le \sum_{i\in J} V\left(x_{N+1,i},2^{N+3}\right)\lesssim \sum_{i\in J} V\left(x_{N+1,i},2^{N-13}\right)\le V\left(o,2^{N+2}\right),
$$
hence the set $I_N$ is actually finite and $\sharp I_N\leq C$ with $C$ independent of $N$. \par
\noindent For all $N\ge 0$ and all $i\in I_N$, denoting $B_{N+1,i}=B\left(x_{N+1,i},2^{N-9}\right)$, the balls $B_{N+1,i}$ and $7B_{N+1,i}$ are remote and satisfy
$$
2^{9}r(B_{N+1,i}) \le r\left(x_{N+1,i}\right)\le 2^{10}r(B_{N+1,i}).
$$
\end{enumerate}
We have constructed a countable family $(B_\alpha)_{\alpha\geq 0}$ of balls covering $M$; actually the family of balls $(\frac{1}{2}B_\alpha)_{\alpha\geq 0}$ also covers $M$ and this will be relevant later. Up to increasing the radius of $B_0$ and deleting balls included in $B_0$, we assume that $B_0$ is the unique ball containing the origin $o$. Denoting the family of balls by $(B_{\alpha})_{\alpha\in \N}$, by $r_{\alpha}$ the radius of $B_{\alpha}$ and by $x_{\alpha}$ its center, then for $\alpha\neq 0$,
\begin{equation} \label{ralpha}
2^{-10}r(x_{\alpha})\leq r_{\alpha}\leq 2^{-9} r(x_{\alpha}).
\end{equation}
In particular, for $\alpha\neq 0$, the balls $B_\alpha$ and $7B_\alpha$ are remote. Also, note that by construction, if $\alpha\neq \beta$ such that $B_\alpha\cap B_\beta\neq \emptyset$, then
\begin{equation}\label{eq:r_alpha}
r_\alpha\simeq r_\beta.
\end{equation}
Another consequence of the construction is that there exists $C\ge 1$ such that, for all $x\in M$, 
$$
\sharp\left\{\alpha\in \N;\ x\in B_{\alpha}\right\}\le C.
$$
In the sequel, if $B\subset M$ is a ball with radius $r(B)$, say that $B$ is admissible if and only if $B$ is remote or $B$ is anchored and $r(B)\le r(B_0)$. We also state for future use (see \eqref{P}):
\begin{lemma} \label{poincadm}
For all admissible balls $B\subset M$ with radius $r(B)$ and all $C^{\infty}$ functions $u\in L^1(B)$:
\begin{enumerate}
\item if $B$ is remote,
\begin{equation} \label{P1rem}
\left\Vert u-u_B\right\Vert_{L^1(B)}\lesssim r(B)\left\Vert du\right\Vert_{L^1(B)},
\end{equation}
\item if $B$ is anchored, \eqref{P1rem} holds, as well as
$$
\left\Vert u-u_{2B}\right\Vert_{L^1(2B)}\lesssim r(B)\left\Vert du\right\Vert_{L^1(2B)}
$$
for all $C^{\infty}$ functions $u\in L^1(2B)$.
\end{enumerate}
\end{lemma}

\noindent Let us now construct a suitable partition of unity adapted to the covering $(B_{\alpha})_{\alpha\in \N}$. 
\begin{lemma}\label{lem:part}
There is a partition of unity $(\chi_{\alpha})_{\alpha\in A}$ subordinate to $(B_{\alpha})_{\alpha\in \N}$, satisfying, for $\alpha\in \N$,
\begin{equation} \label{conditionchi}
\left\vert d\chi_{\alpha}\right\vert\lesssim \frac 1{r_{\alpha}+1},\ \left\vert \Delta\chi_{\alpha}\right\vert\lesssim \frac 1{r_{\alpha}^2+1}.
\end{equation}

\end{lemma}

\begin{proof}

It is clearly enough to prove the estimates \eqref{conditionchi} only for $\alpha\neq 0$. By \cite[Theorem 6.33]{CC} and a scaling argument, for every $\alpha$, there exists a smooth function $\varphi_\alpha :M\to[0,1]$ such that:

\begin{itemize}

\item[(i)] $\varphi_\alpha|_{\frac{1}{2}B_\alpha}\equiv 1$,

\item[(ii)] The support of $\varphi_\alpha$ is included in the (remote) ball $B_\alpha$,

\item[(iii)] $|\nabla \varphi_\alpha|\lesssim \frac{1}{r_\alpha}$,

\item[(iv)] $|\Delta \varphi_\alpha|\lesssim \frac{1}{r_\alpha^2}$.

\end{itemize} 
Let 

$$\varphi:=\sum_{\alpha}\varphi_\alpha,$$
then $\varphi\geq1$ on $M$ since the family of balls $(B_\alpha)_{\alpha\geq0}$ covers $M$. As a consequence of \eqref{eq:r_alpha}, of the fact that the covering has the finite overlap property, and of (iii) and (iv) above,

\begin{equation}\label{eq:phi}
|\nabla \varphi|\lesssim r_\alpha^{-1},\quad |\Delta\varphi|\lesssim r_\alpha^{-2}\quad \mbox{on }B_\alpha.
\end{equation}
We let

$$\chi_\alpha:=\frac{\varphi_\alpha}{\varphi}.$$
Obviously, $\sum_\alpha \chi_\alpha\equiv 1$, and the support of $\chi_\alpha$ is included in $B_\alpha$. Hence,  $(\chi_{\alpha})_{\alpha\in A}$ is a partition of unity, subordinate to $(B_{\alpha})_{\alpha\in A}$. Let us check that $\chi_\alpha$ has the desired properties. One has

$$\nabla \chi_\alpha=-\frac{\varphi_\alpha\nabla\varphi}{\varphi^2}+\frac{\nabla\varphi_\alpha}{\varphi},$$
which implies that $|\nabla \chi_\alpha|\lesssim r_\alpha^{-1}$ by using \eqref{eq:phi} and $\varphi\geq1$, $0\leq \varphi_\alpha\leq 1$. Next,

$$\begin{array}{rcl}
\Delta\chi_\alpha&=&\displaystyle \frac{\Delta\varphi_\alpha}{\varphi}+\varphi_\alpha \Delta(\varphi^{-1})+2\frac{\nabla \varphi_\alpha\cdot \nabla \varphi}{\varphi^2}\\\\
&=&\displaystyle\frac{\Delta\varphi_\alpha}{\varphi}+\varphi_\alpha \left(-\frac{\Delta \varphi}{\varphi^2}+4\frac{|\nabla\varphi|^2}{\varphi^3}\right)+2\frac{\nabla \varphi_\alpha\cdot \nabla \varphi}{\varphi^2},
\end{array}$$
and it follows from \eqref{eq:phi} and $\varphi\geq1$, $0\leq \varphi_\alpha\leq 1$ that $|\Delta \chi_\alpha|\lesssim r_\alpha^{-2}.$
\end{proof}

\subsection{Heat kernel estimates}

Recall that $p_t$ denotes the kernel of $e^{-t\Delta_0}$. One starts with the following gradient estimates for the heat kernel under the hypothesis \eqref{eq:QD} on the Ricci curvature:

\begin{lemma}\label{lem:grad}

Assume that \eqref{eq:QD} and \eqref{DV} holds. Then, for every $N\in \mathbb{N}$,
$$|\nabla_x \partial_t^N p_{t}(x,y)|\lesssim t^{-N}\left(\frac{1}{\sqrt{t}}+\frac{1}{r(x)+1}\right)\frac{1}{V(x,\sqrt{t})}e^{-c\frac{d^2(x,y)}{t}},\quad t>0,\,x,y\in M.$$
\end{lemma}

\begin{proof}

The case $N=0$ follows directly from the classical Li-Yau gradient estimate for positive solutions of the heat equation (see \cite[Section 3.2-3.3]{carron}). Note that for $r(x)\lesssim 1$, we use the fact that the Ricci curvature is bounded from below on $M$.

Let us now turn to the case $N\geq 1$. Denote $k_t(x,y)=t^N \nabla_x \partial_t^N p_{t}(x,y)$, then $k_t(x,y)$ is the kernel of the operator $\nabla (t\Delta)^Ne^{-t\Delta}=2^N \left(\nabla e^{-\frac{t}{2}\Delta} \right)\left((\frac{t}{2}\Delta)^Ne^{-\frac{t}{2}\Delta}\right)$. According to \cite[Theorem 4]{D}, $t^N \partial_t^N p_t(x,y)$ has Gaussian estimates, hence the kernel of the operator $(\frac{t}{2}\Delta)^Ne^{-\frac{t}{2}\Delta}$ has Gaussian estimates. By the estimate for $N=0$, the kernel of the operator $\left(\frac{1}{\sqrt{t}}+\frac{1}{r(x)+1}\right)^{-1}\nabla e^{-\frac{t}{2}\Delta}$ has Gaussian estimates. It is a well-known fact that under \eqref{DV}, the composition of two operators whose kernels have Gaussian estimates, also has a kernel with Gaussian estimates \Bk(\cite[Lemma A.5]{BB}) \Bk. Therefore, we obtain the claimed estimate for $k_t(x,y)$.


\end{proof}
By duality, Lemma \ref{lem:grad} has consequences for the heat kernel on $1$-forms; let $k_N(t,x,y)$ be the kernel of $td^{\ast}(t^2\Delta_1)^N \Bk e^{-t^2\Delta_1}$. Then,

\begin{lemma}\label{estimkernel}

 One has, for all $t>0$ and all $x,y\in M$,
$$
\left\vert k_N(t,x,y)\right\vert\lesssim \frac 1{V(y,t)}\left(1+\frac{t}{r(y)+1}\right)e^{-c\frac{d^2(x,y)}{t^2}}.
$$

\end{lemma}

\begin{proof}

For all $g\in C_0^\infty(\Lambda^1T^{\ast}M)$ and $h\in C_0^\infty(M)$,
$$
\begin{array}{lll}
\dsp \left\vert \int_M td^{\ast} (t^2\Delta_1)^N e^{-t^2\Delta_1}g(x) h(x) d\mu(x) \right\vert & =  & \dsp \left\vert \int_M  (t^2\Delta_1)^N e^{-t^2\Delta_1}g(x)\cdot tdh(x)d\mu(x)\right\vert \\
& = & \dsp \left\vert \int_M g(x)\cdot  (t^2\Delta_1)^N te^{-t^2\Delta_1}dh(x)d\mu(x)\right\vert\\
& = & \dsp\left\vert \int_M g(x)\cdot td (t^2\Delta_0)^N e^{-t^2\Delta_0}h(x)d\mu(x)\right\vert\\
& \lesssim & \dsp \iint_M \left\vert g(x)\right\vert \frac 1{V(x,t)}\left(1+\frac{t}{r(x)+1}\right)e^{-c\frac{d^2(x,y)}{t^2}}\\\\
&& \qquad \qquad \qquad \qquad \qquad \qquad \qquad \times \left\vert h(y)\right\vert d\mu(y)d\mu(x),
\end{array}
$$
where the last line follows from Lemma \ref{lem:grad}. 
\end{proof}
\noindent The following lemma deals with heat kernel estimates for complex time. Before stating the result, define, for all $\theta\in \left(0,\frac{\pi}2\right)$,
$$
\Sigma_{\theta}:=\left\{z\in \C;\ \left\vert \mbox{arg }z\right\vert<\theta\right\}.
$$
\begin{lemma}\label{lem:complex}
Let $\theta<\frac{\pi}{4}$ and $\delta>0$. The operator $V(\cdot,|z|)^{\delta}e^{-z^2\Delta_0}V(\cdot,|z|)^{-\delta}$ has $L^2\to L^2$ off-diagonal estimates for $z\in \Sigma_{\theta}$. More precisely, for every $x,\,y\in M$, and every $z\in \Sigma_\theta$,

$$\left|\left|\chi_{B(y,|z|)}V(\cdot,|z|)^{\delta}e^{-z^2\Delta_0}V(\cdot,|z|)^{-\delta}\chi_{B(x,|z|)}\right|\right|_{2\to 2}\lesssim e^{-C\frac{d^2(x,y)}{|z|^2}}.$$
\end{lemma}

\begin{proof}
For a fixed $z\in\Sigma_\theta$, let us consider a covering of $M$ by balls $B_i:=B(x_i,|z|)$, $i\in\mathbb{N}$ with the following property: there exists $N\ge 1$ independent of $z$ such that, for all $x\in M$, at most $N$ balls $B_i$ intersect $B(x,\left\vert z\right\vert)$.\footnote{Indeed, by the ``$5r$'' covering theorem, for any fixed $z$ there is a covering of $M$ with balls $B_i=B(x_i,\left\vert z\right\vert)$ such that the balls $\frac{1}{5}B_i$ are pairwise disjoint. Now, if $x\in M$, $I_x:=\left\{i\,;\, B_i\cap B(x,\left\vert z\right\vert)\neq \emptyset\right\}$ and $i\in I_x$, then by doubling $V(x,\left\vert z\right\vert)\simeq V(B_i)$. Thus, if we call $N_x:=\sharp I_x$, then
$$
N_xV(x,\left\vert z\right\vert)\leq C\sum_{i\in I_x} V(B_i)\leq C\sum_{i\in I_x} V\left(\frac 15B_i\right)\leq C'V(x,\left\vert z\right\vert),$$
where the constant $C>0$ only depends on the doubling constants. It follows that $N_x\leq C'$.} 

Denote $d_{ij}:=d(x_i,x_j)$, and $\chi_i:=\chi_{B_i}$. Then, by the properties of the covering, it is easy to see that it is enough to prove:

$$\left|\left|\chi_{i}V(\cdot,|z|)^{\delta}e^{-z^2\Delta_0}V(\cdot,|z|)^{-\delta}\chi_{j}\right|\right|_{2\to 2}\lesssim e^{-C\frac{d_{ij}^2}{|z|^2}}.$$
By doubling and Davies-Gaffney estimates for complex times (see \cite[Prop 2.1]{Auscher2007}, the proof of which only relies on uniform ellipticity of the operator under consideration),

$$\begin{array}{rcl}
\dsp\left|\left|\chi_{i}V(\cdot,|z|)^{\delta}e^{-z^2\Delta_0}V(\cdot,|z|)^{-\delta}\chi_{j}\right|\right|_{2\to 2}&\lesssim& \dsp\left(\frac{V(x_i,|z|)}{V(x_j,|z|) }\right)^{\delta}\left|\left|\chi_{i}e^{-z^2\Delta_0}\chi_{j}\right|\right|_{2\to 2} \\\\
&\lesssim & \dsp\left(\frac{V(x_j,|z|+d(x_i,x_j))}{V(x_j,|z|) }\right)^{\delta} e^{-c\frac{d_{ij}^2}{|z|^2}}\\\\
&\lesssim & \dsp\left(1+\frac{d_{ij}^2}{|z|^2}\right)^{\delta D}e^{-c\frac{d_{ij}^2}{|z|^2}}\\\\
&\lesssim& e^{-C\frac{d_{ij}^2}{|z|^2}}
\end{array}$$
\end{proof}
\noindent We now turn to a lemma concerning Gaussian kernels. Let

$$K_t(x,y):=\frac{1}{V(x,t)}e^{-c\frac{d^2(x,y)}{t^2}}$$
be a Gaussian kernel, and $K_t$ be the associated integral operator

$$K_tv(x):=\int_MK_t(x,y)v(y)\,d\mu(y),$$
defined for all measurable functions $v$ such that the integral converges.

\begin{lemma}\label{lem:Gauss}

Let $1\leq p\leq q\leq +\infty$, and denote $\gamma_{p,q}=\frac{1}{p}-\frac{1}{q}$. Let $E$ and $F$ be two measurable sets in $M$. Then, for some positive constants $c_1$ and $c_2$, independent of the sets $E$ and $F$, and for all $t>0$,

$$
e^{c_1\frac{d^2(E,F)}{t^2}}||V(\cdot,t)^{\gamma_{p,q}}K_t||_{L^p(E)\to L^q(F)}\leq c_2 
$$
as well as

$$
e^{c_1\frac{d^2(E,F)}{t^2}}||K_tV(\cdot,t)^{\gamma_{p,q}}||_{L^p(E)\to L^q(F)}\leq c_2.
$$
\end{lemma}

\begin{proof}
We first claim that $K_t$ satisfies

\begin{equation}\label{eq:K1}
\sup_{t>0}||V(\cdot,t)^{\gamma_{p,q}}K_t||_{p\to q}<+\infty,
\end{equation}
Indeed, let us denote $A(x,t,0)=B(x,t)$ and $A(x,t,k)=B(x,(k+1)t)\setminus B(x,kt)$, $k\geq1$. Let $x_0\in M$. Then, for all $k\geq 2$, all measurable functions $v$ supported in $A(x_0,t,k)$, and $x\in B(x_0,t)$, one has by doubling and H\"older

$$\begin{array}{rcl}
\dsp |K_tv(x)|&=&\dsp \int_{A(x_0,t,k)}\frac{e^{-c\frac{d^2(x,y)}{t^2}}}{V(x,t)}|v(y)|\,d\mu(y)\\\\
&\lesssim & \dsp\frac{1}{V(x_0,t)}e^{-ck^2}\mu(A(x_0,t,k))^{1-1/p}||v||_p\\\\
&\lesssim & \dsp\frac{1}{V(x_0,t)^{1/p}}e^{-Ck^2} \left\Vert v\right\Vert_p.
\end{array}$$
Therefore,

$$\begin{array}{rcl}
\dsp ||K_tv||_{L^q(B(x_0,t))}&\leq&\dsp V(x_0,t)^{1/q}||K_tv||_{L^\infty(B(x_0,t))}\\\\
&\lesssim& \dsp\frac{1}{V(x_0,t)^{\gamma_{p,q}}}e^{-Ck^2} ||v||_p.
\end{array}$$ 
Consequently,

$$||\chi_{B(x_0,t)}\,K_t\,\chi_{A(x_0,t,k)}||_{p\to q}\lesssim \frac{e^{-Ck^2}}{V(x_0,t)^{\gamma_{p,q}}}.$$
Hence, the proof of \cite[Prop. 2.9]{AO} applies, and gives \eqref{eq:K1}. This implies the result, in the case $d(E,F)=0$. If now $d(E,F)>0$, then for every $u$ with support in $E$ and every $x\in F$,

$$\begin{array}{rcl}
\dsp\left\vert K_tu(x)\right\vert &\leq&\dsp e^{-\frac{c}{2}\frac{d^2(F,E)}{t^2}}\int_E \frac{1}{V(x,t)}e^{-\frac{c}{2}\frac{d^2(x,y)}{t^2}}\, \left\vert u(y)\right\vert \,d\mu(y)\\\\
&=&\dsp e^{-\frac{c}{2}\frac{d^2(F,E)}{t^2}} \int_E\tilde{K}_t(x,y)\, \left\vert u(y)\right\vert \,d\mu(y),
\end{array}$$
where

$$\tilde{K}_t(x,y)=\frac{1}{V(x,t)}e^{-\frac{c}{2}\frac{d^2(x,y)}{t^2}}$$
is a Gaussian kernel. By the above argument, the associated operator $\tilde{K}_t$ satisfies \eqref{eq:K1}, hence with $C=c/2$,

$$
\sup_{t>0}e^{C\frac{d^2(E,F)}{t^2}}||V(\cdot,t)^{\gamma_{p,q}}K_t||_{L^p(E)\to L^q(F)}<+\infty.
$$
Finally, the inequality for $K_tV(\cdot,t)^{\gamma_{p,q}}$ can be proved by duality: indeed, it is equivalent to

$$
\sup_{t>0}e^{C\frac{d^2(E,F)}{t^2}}||V(\cdot,t)^{\gamma_{p,q}}K_t^*||_{L^{q'}(E)\to L^{p'}(F)}<+\infty,
$$
where $p'$ and $q'$ are the conjugate exponent to $p$ and $q$ respectively, and $K_t^*$ is the adjoint operator to $K_t$. 

%

The kernel of $K_t^*$ is

$$K_t^*(x,y)=K_t(y,x)=\frac{1}{V(y,t)}e^{-c\frac{d^2(x,y)}{t^2}},$$
and using the inequality

$$\frac{V(x,t)}{V(y,t)}\leq \frac{V(y,t+d(x,y))}{V(y,t)}\lesssim \left(1+\frac{d(x,y)}{t}\right)^D,$$
it is easily seen that

$$K_t^*(x,y)\lesssim \frac{1}{V(x,t)}e^{-C\frac{d^2(x,y)}{t^2}},$$
hence $K_t^*$ is bounded by a Gaussian kernel. Therefore, the first part of the argument yields the inequality

$$
\sup_{t>0}e^{C\frac{d^2(E,F)}{t^2}}||V(\cdot,t)^{\gamma_{p,q}}K_t^*||_{L^{q'}(E)\to L^{p'}(F)}<+\infty,
$$
which implies

$$
\sup_{t>0}e^{C\frac{d^2(E,F)}{t^2}}||K_tV(\cdot,t)^{\gamma_{p,q}}||_{L^p(E)\to L^q(F)}<+\infty.
$$

\end{proof}
The next two lemmata will be needed in order to control the heat kernel of the Hodge Laplacian acting on exact one-forms.

\begin{lemma}\label{off-diag1}

Let $B$ be a ball such that $2B$ is admissible, and $u$ be a function in $C_0^\infty(B)$. Let $F\subset M$ be such that

$$r(B)\lesssim r(x)+1,\quad \forall x\in F.$$
Then, for every $t>0$,

$$|| V(\cdot,t)^{1/2} e^{-t^2\Delta_1}(du)||_{L^2(F)}\lesssim \left(1+\frac{r(B)}{t}\right)e^{-\frac{cd(F,B)^2}{t^2}}||du||_1.$$

\end{lemma}

\begin{proof}

For every $x\in F$,

\begin{equation}\label{eq:domin}
\begin{array}{rcl}
\dsp \left\vert e^{-t^2\Delta_1}(du)\right\vert (x)&=&\dsp |\nabla e^{-t^2\Delta_0}u|(x)\\
&\leq &\dsp \int_B |\nabla_x  p_{t^2}(x,y)| |u(y)|\,d\mu(y)\\
&\lesssim& \dsp\left(\frac{1}{t}+\frac{1}{r(x)+1}\right)\int_B K_t(x,y)|u(y)|\,d\mu(y)\\
&\lesssim& \dsp\left(\frac{1}{t}+\frac{1}{r(B)}\right)K_t(|u|)(x),
\end{array}
\end{equation}
where $K_t$ is a Gaussian kernel and we have used the assumption on $F$ and Lemma \ref{lem:grad}. According to Lemma \ref{lem:Gauss}, one gets

$$\sup_{t>0}e^{C\frac{d^2(F,B)}{t^2}}||V^{1/2}(\cdot,t)e^{-t^2\Delta_1}(du)||_{L^2(F)}\lesssim  \left(\frac{1}{t}+\frac{1}{r(B)}\right) ||u||_{L^1(B)}.$$
Since $2B$ is admissible, it supports an $L^1$ Poincar\'e inequality with constant of order $r(B)$ by Lemma \ref{poincadm}. Since $u$ vanishes on $2B\setminus B$, one gets (see \cite[Lemma 4.2.3]{BGL})

\begin{equation}\label{eq:domin_P}
\int_B|u|\lesssim r(B)\int_{2B}|\nabla u|=r(B)\int_B|\nabla u|.
\end{equation}
Therefore, one arrives to

$$\sup_{t>0}e^{C\frac{d^2(F,B)}{t^2}}||V^{1/2}(\cdot,t)e^{-t^2\Delta_1}(du)||_{L^2(F)}\lesssim  \left(\frac{r(B)}{t}+1\right) ||du||_{L^1(B)}.$$
and the result follows.

\end{proof}

\begin{lemma}\label{off-diag2}

Let $B$ be an admissible ball, and $u$ be a function in $C_0^\infty(B)$. Let $0<\theta<\frac{\pi}{2}$, and let $\Sigma_\theta$ denotes the sector of angle $\theta$ in $\mathbb{C}$. Let $F$ be a measurable set in $M$. Let $N\in \mathbb{N}$. Then, for $z\in \Sigma_\theta$, there holds:

$$||V(\cdot,|z|)^{1/2}zd^* (t^2\Delta_1)^N e^{-z^2\Delta_1}(du)||_{L^2(F)}\lesssim e^{-c\frac{d(F,B)^2}{|z|^2}}||du||_1,$$
where the various constants in the inequality are independent of the ball $B$ and the function $u$. 
\end{lemma}

\begin{proof}

Denote $x_B$ the center of $B$. We start with the case $z=t>0$ positive real number, for which there are two cases: either $t\leq r(x_B)+1$, or $t>r(x_B)+1$. For $t\leq r(x_B)+1$, we proceed by duality:  let $h\in L^2$ with support in $F$, then

$$
\begin{array}{lll}
\dsp \left\vert \int_F V(x,t)^{1/2}td^{\ast} (t^2\Delta_1)^N e^{-t^2\Delta_1}(du)(x)\cdot h(x)d\mu(x) \right\vert   &
\end{array}$$

$$
\begin{array}{lll}
 = & \dsp\left\vert \int_F V(x,t)^{1/2}t (t^2\Delta_0)^N e^{-t^2\Delta_0}(d^*du)(x)\cdot h(x)d\mu(x) \right\vert\\
= & \dsp \left\vert \int_M d^*du(x)\cdot t (t^2\Delta_0)^N e^{-t^2\Delta_0}V(x,t)^{1/2}h(x)d\mu(x)\right\vert \\
 = & \dsp \left\vert \int_B du(x)\cdot \left(t d (t^2\Delta_0)^N e^{-t^2\Delta_0}V(\cdot ,t)^{1/2}h\right)(x)d\mu(x)\right\vert.
\end{array}
$$
However, by Lemma \ref{lem:grad},

$$\left|\left(td(t^2\Delta_0)^N e^{-t^2\Delta_0}V(\cdot,t)^{1/2}h\right)(x)\right|\lesssim \dsp \int_F  \left(1+\frac{t}{r(x)+1}\right)\frac{e^{-c\frac{d^2(x,y)}{t^2}}}{V(x,t)}\,V(y,t)^{1/2}\left\vert h(y)\right\vert \,d\mu(y)$$
Since $t\leq r(x_B)+1$ and $B$ is admissible, it follows that $t\lesssim r(x)+1$ for every $x\in B$. Hence,

$$\left|\left(td (t^2\Delta_0)^N e^{-t^2\Delta_0}V(\cdot,t)^{1/2}h\right)(x)\right|\lesssim \int K_t(x,y)\,V(y,t)^{1/2}|h(y)|\,d\mu(y),\quad x\in B,$$
where $K_t(x,y)$ is a Gaussian kernel. According to Lemma \ref{lem:Gauss}, one obtains

$$\sup_{t>0}e^{C\frac{d^2(F,B)}{t^2}}\left|\left|td  (t^2\Delta_0)^N e^{-t^2\Delta_0}V(\cdot,t)^{1/2}\right|\right|_{L^2(F)\to L^\infty(B)}<+\infty.$$
This implies

$$\dsp \left\vert \int_F V(x,t)^{1/2}td^{\ast} (t^2\Delta_1)^N e^{-t^2\Delta_1}(du)(x)\cdot h(x)d\mu(x) \right\vert\lesssim e^{-C\frac{d^2(F,B)}{t^2}} ||du||_{L^1(B)}\cdot ||h||_{L^2(F)},$$
hence

$$||V(\cdot,t)^{1/2}td^* (t^2\Delta_1)^N e^{-t^2\Delta_1}(du)||_{L^2(F)}\lesssim e^{-c\frac{d(F,B)^2}{t^2}}||du||_1.$$
This proves the result for $z=t\leq r(x_B)+1$. Now, we treat the case $t>r(x_B)+1$: we write
$$
 td^* (t^2\Delta_1)^N e^{-t^2\Delta_1}(du)=\frac 1t (t^2\Delta_0)^{N+1} e^{-t^2\Delta_0}u. 
$$
According to \cite[Theorem 4]{D}, the kernel $s^{N+1}\partial_s^{N+1}p_s(x,y)$ has pointwise Gaussian estimates. Applying this with $s=t^2$ and using Lemma \ref{lem:Gauss},
$$||V(\cdot,t)^{1/2}td^* (t^2\Delta_1)^N e^{-t^2\Delta_1}(du)||_{L^2(F)}\lesssim \frac{1}{t}e^{-c\frac{d^2(F,B)}{t^2}}||u||_1.$$
As in the proof of Lemma \ref{off-diag1}, \eqref{eq:domin_P} yields

$$||V(\cdot,t)^{1/2}td^* (t^2\Delta_1)^N e^{-t^2\Delta_1}(du)||_{L^2(F)}\lesssim \frac{r(B)}{t}e^{-c\frac{d^2(F,B)}{t^2}}||du||_1.$$
Since $B$ is admissible, $r(B)\lesssim r(x_B)+1$, and because $t>r(x_B)+1$, one gets that

$$||V(\cdot,t)^{1/2}td^* (t^2\Delta_1)^N e^{-t^2\Delta_1}(du)||_{L^2(F)}\lesssim e^{-c\frac{d^2(F,B)}{t^2}}||du||_1.$$
This concludes the proof for $z=t>0$ real; it remains to prove the lemma for complex $z$. We write $z^2=(z')^2+t^2$, where $z'\in \Sigma_\mu$ with $\mu>\theta$, $t>0$, and

\begin{equation} \label{zzt}
|z|\simeq |z'|\simeq t.
\end{equation}
Then, one has

$$V(\cdot,|z|)^{1/2}zd^*\Bk (z^2\Delta_1)^N\Bk e^{-z^2\Delta_1}$$
$$\begin{array}{rcl}
=\dsp \left(\frac{z}{t}\right)^{2N+1}\left(V(\cdot,|z|)^{1/2}e^{-(z')^2\Delta_0}V(\cdot,|z|)^{-1/2}\right)\left(V(\cdot,|z|)^{1/2}td^* \Bk (t^2\Delta_1)^N\Bk e^{-t^2\Delta_1}\right).
\end{array}$$
Since \eqref{zzt} holds, Lemma \ref{lem:complex} and the above entail that the $L^2\to L^2$ off-diagonal estimates for $\left(V(\cdot,|z|)^{1/2}e^{-(z')^2\Delta_0}V(\cdot,|z|)^{-1/2}\right)$ and for $\left(V(\cdot,|z^{\prime}|)^{1/2}e^{-(z')^2\Delta_0}V(\cdot,|z'|)^{-1/2}\right)$ are comparable. The same is true for the $L^1\to L^2$ off-diagonal estimates for $\left(V(\cdot,\left\vert z\right\vert)^{1/2}td^* (t^2\Delta_1)^N e^{-t^2\Delta_1}\right)$ and $\left(V(\cdot,t)^{1/2}td^*(t^2\Delta_1)^N e^{-t^2\Delta_1}\right)$. Since the term $\frac{z}{t}$ is bounded, the composition lemma (see \cite[Proposition 3.1]{Auscher2007} for the Euclidean case) yields the $L^1\to L^2$ off-diagonal estimates for the composed operator 

$$\left(V(\cdot,|z'|)^{1/2}e^{-(z')^2\Delta_0}V(\cdot,|z'|)^{-1/2}\right)\left(V(\cdot,t)^{1/2}td^*\Bk (t^2\Delta_1)^N\Bk  e^{-t^2\Delta_1}\right),$$
hence the result.

\end{proof}

\section{Proof of Proposition \ref{pro:tp3}} \label{prooftp3}

\subsection{Splitting into three regimes}

\noindent 
\noindent Recall that, for all $x\in M$,
$$
{\mathcal A}\omega(x) := \left(\iint_{\Gamma(x)} \left\vert  td^{\ast}(t^2\Delta_1)^Ne^{-t^2\Delta_1}\omega(z)\right\vert^2\frac{d\mu(z)}{V(z,t)}\frac{dt}t\right)^{\frac 12}.
$$
The conclusion of Proposition \ref{pro:tp3} means that
\begin{equation} \label{alp}
\left\Vert {\mathcal A}(du)\right\Vert_p\lesssim \left\Vert  du\right\Vert_p.
\end{equation} 
For the proof of \eqref{alp}, following \cite{carron}, we fix a constant $ \kappa\geq 2^{10}$ and, as explained in the introduction, decompose the integration domain in the definition of ${\mathcal A}\omega$ into three pieces or ``regimes'', namely:
\begin{equation}
{\mathcal A}\omega\leq {\mathcal A}_l\omega+{\mathcal A}_s\omega+{\mathcal A}_d\omega,
\end{equation}
where ${\mathcal A}_l\omega$ stands for the ``long-to-short'' regime, that is
\begin{equation} \label{defbl}
{\mathcal A}_l\omega(x):=\left(\iint_{(z,t)\in \Gamma(x)} \left(\int_{r(y)\geq \kappa r(z)} k_t(z,y)\cdot \omega(y) d\mu(y)\right)^2\frac{d\mu(z)}{V(z,t)}\frac{dt}t\right)^{\frac 12},
\end{equation}
${\mathcal A}_s\omega$ stands for the ``short-to-long'' regime, that is
\begin{equation} \label{defbs}
{\mathcal A}_s\omega(x):=\left(\iint_{(z,t)\in \Gamma(x)} \left(\int_{\kappa r(z)>r(y),\ d(z,y)\geq \kappa^{-1}r(z)} k_t(z,y)\cdot \omega(y) d\mu(y)\right)^2\frac{d\mu(z)}{V(z,t)}\frac{dt}t\right)^{\frac 12},
\end{equation}
and ${\mathcal A}_d\omega$ stands for the ``diagonal'' regime, that is
\begin{equation} \label{defbd}
{\mathcal A}_d\omega(x):=\left(\iint_{(z,t)\in \Gamma(x)} \left(\int_{d(z,y)<\kappa^{-1}r(z)} k_t(z,y)\cdot \omega(y) d\mu(y)\right)^2\frac{d\mu(z)}{V(z,t)}\frac{dt}t\right)^{\frac 12}.
\end{equation}
Recall that $k_t$ is the kernel of $td^{\ast}e^{-t^2\Delta_1}$. Notice that, whenever $r(y)\geq \kappa r(z)$, one has
$$
d(z,y)\geq r(y)-r(z)\geq (\kappa-1)r(z)\geq \kappa^{-1}r(z),
$$
which shows that the long-to-short and the short-to-long regimes cover the case where $d(z,y)\geq \kappa^{-1}r(z)$. 
\subsection{The ``long-to-short'' regime} \label{lts}
In this section, we establish that, for all $\lambda>0$,
\begin{equation} \label{Blweak}
\mu\left(\left\{x\in M;\ {\mathcal A}_l\omega(x)>\lambda\right\}\right)\lesssim  \frac{\left\Vert \omega\right\Vert_1}{\lambda}.
\end{equation}
To this purpose, we split ${\mathcal A}_l$ into two parts, whether $t\geq r(y)$ or $t<r(y)$, that is
$$
\begin{array}{lll}
\dsp{\mathcal A}_l\omega(x) & \leq & \dsp \left(\iint_{(z,t)\in \Gamma(x)} \left(\int_{t\geq r(y)\geq \kappa r(z)} \left\vert k_t(z,y)\right\vert\left\vert \omega(y)\right\vert d\mu(y)\right)^2\frac{d\mu(z)}{V(z,t)}\frac{dt}t\right)^{\frac 12}\\
& + & \dsp \left(\iint_{(z,t)\in \Gamma(x)} \left(\int_{r(y)\geq \max(t,\kappa r(z))} \left\vert k_t(z,y)\right\vert\left\vert \omega(y)\right\vert d\mu(y)\right)^2\frac{d\mu(z)}{V(z,t)}\frac{dt}t\right)^{\frac 12}\\
& =: & \dsp {\mathcal A}_{l,1}\omega(x)+{\mathcal A}_{l,2}\omega(x).
\end{array}
$$
\subsubsection{ The case $t\geq r(y)$: } \label{LSlarge_t}

For this part, by Lemma \ref{estimkernel},
$$
\begin{array}{lll}
\dsp {\mathcal A}_{l,1}\omega(x) & = & \dsp \left(\iint_{d(z,x)\leq t} \left(\int_{\kappa r(z)\leq r(y)\leq t} \left\vert k_t(z,y)\right\vert \left\vert \omega(y)\right\vert d\mu(y)\right)^2\frac{d\mu(z)}{V(z,t)}\frac{dt}t\right)^{\frac 12}\\
& \lesssim & \dsp \left(\iint_{d(z,x)\leq t} \left(\int_{\kappa r(z)\leq r(y)\leq t} \frac 1{V(y,t)}\frac{t}{r(y)} \left\vert \omega(y)\right\vert d\mu(y)\right)^2\frac{d\mu(z)}{V(z,t)}\frac{dt}t\right)^{\frac 12}.
\end{array}
$$
We will estimate the latter quantity by a duality argument. Pick up a function $h\in L^2\left(\Gamma(x),\ \frac{d\mu(z)}{V(z,t)}\frac{dt}t\right)$ such that 
\begin{equation} \label{normh}
\int_{\Gamma(x)} \left\vert h(z,t)\right\vert^2\frac{d\mu(z)}{V(z,t)}\frac{dt}t=1.
\end{equation}
Then, by Fubini,
$$
\begin{array}{l}
\dsp \iint_{d(z,x)\leq t} \left\vert h(z,t)\right\vert \left(\int_{\kappa r(z)\leq r(y)\leq t} \frac 1{V(y,t)}\frac{t}{r(y)} \left\vert \omega(y)\right\vert d\mu(y)\right)\frac{d\mu(z)}{V(z,t)}\frac{dt}t   \\
=  \dsp  \int_M  \left\vert \omega(y)\right\vert \left(\iint_{d(z,x)\leq t,\ \kappa r(z)\leq r(y)\leq t} \left\vert h(z,t)\right\vert \frac 1{V(y,t)}\frac{t}{r(y)} \frac{d\mu(z)}{V(z,t)}\frac{dt}t\right)d\mu(y).
\end{array}
$$
The Cauchy-Schwarz inequality and condition \eqref{normh} yield
$$
\begin{array}{l}
\dsp \iint_{d(z,x)\leq t,\ \kappa r(z)\leq r(y)\leq t} \left\vert h(z,t)\right\vert \frac 1{V(y,t)}\frac{t}{r(y)} \frac{d\mu(z)}{V(z,t)}\frac{dt}t  \\
\dsp \leq \left(\iint_{d(z,x)\leq t,\ \kappa r(z)\leq r(y)\leq t} \frac 1{V(y,t)^2}\frac{t^2}{r^2(y)} \frac{d\mu(z)}{V(z,t)}\frac{dt}t\right)^{1/2},
\end{array}
$$
so that
\begin{equation} \label{I1I2}
\begin{array}{l}
\dsp \iint_{d(z,x)\leq t} \left\vert h(z,t)\right\vert \left(\int_{\kappa r(z)\leq r(y)\leq t} \frac 1{V(y,t)}\frac{t}{r(y)} \left\vert \omega(y)\right\vert d\mu(y)\right)\frac{d\mu(z)}{V(z,t)}\frac{dt}t\\
\dsp \leq \int_M  \left\vert \omega(y)\right\vert \left(\left(\iint_{d(z,x)\leq t,\ \kappa r(z)\leq r(y)\leq t}\frac 1{V(y,t)^2}\frac{t^2}{r^2(y)} \frac{d\mu(z)}{V(z,t)}\frac{dt}t\right)^{1/2}\right)d\mu(y)\\
\dsp = \int_{r(y)\geq r(x)}  \left\vert \omega(y)\right\vert \left(\left(\iint_{d(z,x)\leq t,\ \kappa r(z)\leq r(y)\leq t} \frac 1{V(y,t)^2}\frac{t^2}{r^2(y)} \frac{d\mu(z)}{V(z,t)}\frac{dt}t\right)^{1/2}\right)d\mu(y)\\
\dsp +\int_{r(y)<r(x)}  \left\vert \omega(y)\right\vert \left(\left(\iint_{d(z,x)\leq t,\ \kappa r(z)\leq r(y)\leq t} \frac 1{V(y,t)^2}\frac{t^2}{r^2(y)} \frac{d\mu(z)}{V(z,t)}\frac{dt}t\right)^{1/2}\right)d\mu(y)\\
\dsp =: I_1+I_2.
\end{array}
\end{equation}	
Notice that
\begin{equation} \label{vot}
\begin{array}{lll}
\dsp \int_{\kappa r(z)\leq r(y)\leq t} \frac 1{V(z,t)}d\mu(z) & = & \dsp\frac 1{V(o,t)}\int_{\kappa r(z)\leq r(y)\leq t} \frac{V(o,t)}{V(z,t)}d\mu(z)\\
& \leq & \dsp\frac 1{V(o,t)}\int_{\kappa r(z)\leq r(y)\leq t} \frac{V(z,t+r(z))}{V(z,t)}d\mu(z)\\
& \lesssim & \dsp \frac{V(o,r(y))}{V(o,t)}.
\end{array}
\end{equation}
We therefore estimate the innermost integral in $I_1$ as follows:
$$
\begin{array}{l}
\dsp\iint_{d(z,x)\leq t,\ \kappa r(z)\leq r(y)\leq t} \frac 1{V(y,t)^2} \frac{t^2}{r^2(y)} \frac{d\mu(z)}{V(z,t)}\frac{dt}t\\
\dsp = \int_{r(y)}^{+\infty} \frac 1{V(y,t)^2} \frac{t^2}{r^2(y)} \left(\int_{d(z,x)\leq t,\ \kappa r(z)\leq r(y)} \frac{d\mu(z)}{V(z,t)}\right)\frac{dt}t\\
\dsp \lesssim \int_{r(y)}^{+\infty} \frac 1{V(y,t)^2} \frac{t^2}{r^2(y)} \frac{V(o,r(y))}{V(o,t)}\frac{dt}t\\
\dsp  =  \int_{r(y)}^{+\infty}  \frac 1{V(o,t)^2}\frac{V(o,t)^2}{V(y,t)^2}\frac{t^2}{r^2(y)} \frac{V(o,r(y))}{V(o,t)}\frac{dt}t\\
\dsp\lesssim \int_{r(y)}^{+\infty} \frac 1{V(o,t)^2} \frac{t^2}{r^2(y)} \frac{V(o,r(y))}{V(o,t)}\frac{dt}t\\
\dsp = \frac 1{V(o,r(y))^2} \int_{r(y)}^{+\infty} \frac {V(o,r(y))^3}{V(o,t)^3} \frac{t^2}{r^2(y)} \frac{dt}t\\
\dsp\lesssim \frac 1{V(o,r(y))^2} \int_{r(y)}^{+\infty} \left(\frac{r(y)}t\right)^{3\nu-2}\frac{dt}t\lesssim \frac 1{V(o,r(y))^2}.
\end{array}
$$
where the third line holds since $r(y)\le t$ and the fifth line follows from $V(o,t)\leq V(y,t+r(y))\lesssim V(y,t)$ since $r(y)\leq t$. As a consequence,
\begin{equation} \label{estimI1}
I_1\lesssim  \int_{r(y)\geq r(x)}  \frac{\left\vert \omega(y)\right\vert}{V(o,r(y))}d\mu(y) \leq \frac 1{V(o,r(x))} \left\Vert \omega\right\Vert_1.
\end{equation}
For $I_2$, notice first that, when $d(z,x)\leq t$ \Bk and \Bk $\kappa r(z)\leq r(y)\leq t$,
\begin{equation} \label{rxt}
r(x)\leq r(z)+t\leq \frac 1{\kappa}r(y)+t\leq 2t.
\end{equation}
On the other hand, 
\begin{equation} \label{dxy}
d(x,y)\leq d(x,z)+d(z,y)\leq t+r(z)+r(y)\leq 3t.
\end{equation}
Gathering \eqref{rxt} and \eqref{dxy} and using \eqref{vot} again, we obtain
$$
\begin{array}{l}
\dsp \iint_{d(z,x)\leq t,\ \kappa r(z)\leq r(y)\leq t} \frac 1{V(y,t)^2}\frac{t^2}{r^2(y)} \frac{d\mu(z)}{V(z,t)}\frac{dt}t\\
\dsp \leq \int_{\frac{r(x)}2}^{+\infty} \frac 1{V(y,t)^2}\frac{t^2}{r^2(y)} \left(\int_{d(z,x)\leq t,\ \kappa r(z)\leq r(y)} \frac{d\mu(z)}{V(z,t)}\right) \frac{dt}t\\
\dsp \lesssim \int_{\frac{r(x)}2}^{+\infty} \frac 1{V(y,t)^2}\frac{t^2}{r^2(y)}\frac{V(o,r(y))}{V(o,t)}\frac{dt}t\\
\dsp\lesssim \int_{\frac{r(x)}2}^{+\infty} \frac 1{V(x,t)^2}\frac{t^2}{r^2(y)}\frac{V(o,r(y))}{V(o,t)}\frac{dt}t\\
\dsp  \lesssim \frac 1{V(x,r(x))^2} \int_{\frac{r(x)}2}^{+\infty} \frac{t^2}{r^2(y)}\frac{V(o,r(y))}{V(o,t)}\frac{dt}t\\
\dsp\lesssim \frac 1{V(o,r(x))^2}\int_{\frac{r(y)}2}^{+\infty} \left(\frac{r(y)}t\right)^{\nu-2}\frac{dt}t\\
\dsp\lesssim \frac 1{V(o,r(x))^2},
\end{array}
$$
where the fourth line uses \eqref{dxy}. As a consequence,
\begin{equation} \label{estimI2}
I_2\lesssim  \frac 1{V(o,r(x))}\int_{r(y)<r(x)} \left\vert \omega(y)\right\vert d\mu(y) \leq \frac 1{V(o,r(x))} \left\Vert \omega\right\Vert_1.
\end{equation}
Using \eqref{I1I2}, \eqref{estimI1} and \eqref{estimI2}, we conclude that
$$
\begin{array}{l}
\dsp\iint_{d(z,x)\leq t} \left\vert h(z,t)\right\vert \left(\int_{\kappa r(z)\leq r(y)\leq t} \frac 1{V(y,t)}\frac{t}{r(y)} \left\vert \omega(y)\right\vert d\mu(y)\right)\frac{d\mu(z)}{V(z,t)}\frac{dt}t\\
\dsp\lesssim \frac 1{V(o,r(x))} \left\Vert \omega\right\Vert_1,
\end{array}
$$
and finally, taking the supremum over all functions $h\in L^2\left(\Gamma(x),\ \frac{d\mu(z)}{V(z,t)}\frac{dt}t\right)$ satisfying \eqref{normh},
$$
{\mathcal A}_{l,1}\omega(x)\lesssim \frac 1{V(o,r(x))} \left\Vert \omega\right\Vert_1.
$$
Thus, if $\lambda>0$ and ${\mathcal A}_{l,1}\omega(x)>\lambda$, then $V(o,r(x))\leq \frac{\left\Vert \omega\right\Vert_1}{\lambda}$. Lemma \ref{weak} in the Appendix therefore yields 
\begin{equation} \label{Bl1weak}
\mu\left(\left\{x\in M;\ {\mathcal A}_{l,1}\omega(x)>\lambda\right\}\right)\lesssim  \frac{\left\Vert \omega\right\Vert_1}{\lambda}.
\end{equation}

\subsubsection{The case $t\leq r(y)$: } \label{tlry}

Here,
\begin{equation} \label{AGlongshortbis}
\begin{array}{lll}
\dsp {\mathcal A}_{l,2}\omega(x) & = & \dsp \left(\iint_{d(z,x)\leq t} \left(\int_{\max(\kappa r(z),t)\leq r(y)} \left\vert k_t(z,y)\right\vert \left\vert \omega(y)\right\vert d\mu(y)\right)^2\frac{d\mu(z)}{V(z,t)}\frac{dt}t\right)^{\frac 12}\\
& \lesssim &\dsp \left(\iint_{d(z,x)\leq t} \left(\int_{\max(\kappa r(z),t)\leq r(y)} \frac 1{V(y,t)}e^{-c\frac{r^2(y)}{t^2}} \left\vert \omega(y)\right\vert d\mu(y)\right)^2\frac{d\mu(z)}{V(z,t)}\frac{dt}t\right)^{\frac 12},
\end{array}
\end{equation}
where the last line holds since
$$
r(y)\le r(z)+d(z,y)\le \frac 1{\kappa}r(y)+d(z,y)
$$
so that $r(y)\lesssim d(z,y)$. As in the previous case, we therefore have to estimate
$$
\begin{array}{l}
\dsp\int_{r(y)\geq r(x)}  \left\vert \omega(y)\right\vert \left(\left(\iint_{d(z,x)\leq t,\ \max(\kappa r(z),t)\leq r(y)}\frac 1{V(y,t)^2} e^{-c\frac{r^2(y)}{t^2}} \frac{d\mu(z)}{V(z,t)}\frac{dt}t\right)^{1/2}\right)d\mu(y)\\
\dsp+\int_{r(y)<r(x)} \left\vert \omega(y)\right\vert \left(\left(\iint_{d(z,x)\leq t,\ \max(\kappa r(z),t)\leq r(y)} \frac 1{V(y,t)^2} e^{-c\frac{r^2(y)}{t^2}} \frac{d\mu(z)}{V(z,t)}\frac{dt}t\right)^{1/2}\right)d\mu(y) \\
\dsp =:I_1+I_2.
\end{array}
$$	
Since $V(o,t)\leq V(z,t+r(z))\leq V(z,t+r(y))\lesssim V(z,t)\left(1+\frac{r(y)}t\right)^D$ whenever $\kappa r(z)\leq r(y)$, we estimate the innermost integral in $I_1$ by
$$
\begin{array}{l}
\dsp\iint_{d(z,x)\leq t,\ \max(\kappa r(z),t)\leq r(y)}\frac 1{V(y,t)^2} e^{-c\frac{r^2(y)}{t^2}} \frac{d\mu(z)}{V(z,t)}\frac{dt}t\\
\dsp = \int_0^{r(y)} \frac 1{V(y,t)^2} e^{-c\frac{r^2(y)}{t^2}} \left(\int_{\kappa r(z)\leq r(y)} \frac{d\mu(z)}{V(z,t)}\right) \frac{dt}t\\
\dsp \lesssim \int_0^{r(y)} \frac {V(o,r(y))}{V(o,t)}\frac 1{V(y,t)^2}\left(1+\frac{r(y)}t\right)^De^{-c\frac{r^2(y)}{t^2}}\frac{dt}t\\
\dsp \le \int_0^{r(y)} \left(\frac{r(y)}t\right)^D\frac 1{V(y,r(y))^2}\frac{V(y,r(y))^2}{V(y,t)^2}\left(1+\frac{r(y)}t\right)^{D} e^{-c\frac{r^2(y)}{t^2}}\frac{dt}t\\
\dsp\lesssim \frac 1{V(o,r(y))^2}.
\end{array}
$$
It follows that
\begin{equation} \label{estimI1bis}
I_1\lesssim \int_{r(y)\geq r(x)}  \frac{\left\vert \omega(y)\right\vert}{V(o,r(y))}d\mu(y)
\le \frac{\left\Vert \omega\right\Vert_1}{V(o,r(x))}.
\end{equation}
For the innermost integral in $I_2$, since $\kappa r(z)\leq r(y)\leq r(x)$,  
$$
r(x)\leq r(z)+t\leq \frac 1{\kappa}r(y)+t\leq \frac 1{\kappa}r(x)+t,
$$
one has $r(x)\lesssim t\leq r(y)\leq r(x)$, which entails
$$
\begin{array}{l}
\dsp\iint_{d(z,x)\leq t,\ \max(\kappa r(z),t)\leq r(y)}\frac 1{V(y,t)^2} e^{-c\frac{r^2(y)}{t^2}} \frac{d\mu(z)}{V(z,t)}\frac{dt}t\\
\dsp = \int_{cr(x)}^{r(y)} \frac 1{V(y,t)^2}  e^{-c\frac{r^2(y)}{t^2}} \left(\int_{B(x,t)} \frac{d\mu(z)}{V(z,t)}\right)\frac{dt}t\\
\dsp\lesssim \int_{cr(x)}^{r(y)}  \frac 1{V(y,t)^2}  e^{-c\frac{r^2(y)}{t^2}} \frac{dt}t\\
\dsp\leq  \frac 1{V(y,r(y))^2} \int_{cr(x)}^{r(y)} \left(\frac{r(y)}t\right)^{2D}e^{-c\frac{r^2(y)}{t^2}} \frac{dt}t\\
\dsp\lesssim  \frac 1{V(y,r(y))^2}\lesssim \frac 1{V(x,r(x))^2}, 
\end{array}
$$
therefore
\begin{equation} \label{estimI2bis}
I_2\lesssim \frac 1{V(x,r(x))}\left\Vert \omega\right\Vert_1\lesssim \frac 1{V(o,r(x))}\left\Vert \omega\right\Vert_1.
\end{equation}
Gathering \eqref{AGlongshortbis}, \eqref{estimI1bis} and \eqref{estimI2bis}, we obtain
$$
{\mathcal A}_{l,2}\omega(x)\lesssim \frac 1{V(o,r(x))}\left\Vert \omega\right\Vert_1,
$$
and, using Lemma \ref{weak} again, we conclude that 
\begin{equation} \label{Bl2weak}
\mu\left(\left\{x\in M;\ {\mathcal A}_{l,2}\omega(x)>\lambda\right\}\right)\lesssim  \frac{\left\Vert \omega\right\Vert_1}{\lambda}.
\end{equation}
The conjunction of \eqref{Bl1weak} and \eqref{Bl2weak} yields \eqref{Blweak}.

\subsection{The ``short-to-long'' regime}

This section is devoted to the analysis of ${\mathcal A}_s\omega$. Again, we split this term into two parts: we bound ${\mathcal A}_s\omega(x)$ by the sum ${\mathcal A}_{s,1}\omega(x)+{\mathcal A}_{s,2}\omega(x)$, where

$${\mathcal A}_{s,1}\omega(x)=\left(\iint_{(z,t)\in \Gamma(x)} \left(\int_{\kappa r(z)>r(y),\ d(z,y)\geq \kappa^{-1}r(z),\ t\geq r(y)} \left\vert k_t(z,y)\right\vert\left\vert \omega(y)\right\vert d\mu(y)\right)^2\frac{d\mu(z)}{V(z,t)}\frac{dt}t\right)^{\frac 12}$$
and

$${\mathcal A}_{s,2}\omega(x)=\left(\iint_{(z,t)\in \Gamma(x)} \left(\int_{\kappa r(z)>r(y),\ d(z,y)\geq \kappa^{-1}r(z),\ t<r(y)} \left\vert k_t(z,y)\right\vert\left\vert \omega(y)\right\vert d\mu(y)\right)^2\frac{d\mu(z)}{V(z,t)}\frac{dt}t\right)^{\frac 12}$$
In this regime, we will assume and use the fact that $p>\frac{\nu}{\nu-1}$. We then intend to prove that, for all $\lambda>0$,
\begin{equation} \label{Bsweak1}
\mu\left(\left\{x\in M;\ {\mathcal A}_{s,1}\omega(x)>\lambda\right\}\right)\lesssim  \frac{\left\Vert \omega\right\Vert_p^p}{\lambda^p}
\end{equation}
and
\begin{equation} \label{Bsweak2}
\mu\left(\left\{x\in M;\ {\mathcal A}_{s,2}\omega(x)>\lambda\right\}\right)\lesssim  \frac{\left\Vert \omega\right\Vert_1}{\lambda}.
\end{equation}
Note that, in this short-to-long regime, since $d(y,z)\geq \kappa^{-1}r(z)$,
\begin{equation} \label{estimksl}
\left\vert k_t(z,y)\right\vert\lesssim \frac 1{V(z,t)}\left(1+\frac{t}{r(y)}\right)e^{-c\frac{r^2(z)}{t^2}}.
\end{equation}
Indeed, using $r(z)\lesssim d(y,z)$ and doubling, one has
$$
\begin{array}{rcl}
\dsp \left\vert k_t(z,y)\right\vert & \lesssim & \dsp \frac 1{V(y,t)}\left(1+\frac t{r(y)}\right)e^{-c\frac{d^2(y,z)}{t^2}}\\
& \lesssim & \dsp \frac 1{V(z,t)}\frac{V(y,t+d(y,z))}{V(y,t)}\left(1+\frac t{r(y)}\right)e^{-c\frac{d^2(y,z)}{t^2}}\\
& \lesssim & \dsp \frac 1{V(z,t)} \left(1+\frac t{r(y)}\right)e^{-c\frac{r^2(z)}{t^2}}.
\end{array}
$$

\subsubsection{ The case $t\geq r(y)$}

As in the corresponding case of Section \ref{lts}, using \eqref{estimksl}, one has
$$
\begin{array}{lll}
\dsp {\mathcal A}_{s,1}\omega(x)   \le   \dsp \left(\iint_{d(z,x)\leq t} \left(\int_{\kappa r(z)>r(y),\ d(z,y)\geq \kappa^{-1}r(z),\ r(y)\leq t} \left\vert k_t(z,y)\right\vert \left\vert \omega(y)\right\vert d\mu(y)\right)^2\frac{d\mu(z)}{V(z,t)}\frac{dt}t\right)^{\frac 12}\\
\lesssim \dsp\left(\iint_{d(z,x)\leq t} \left(\int_{\kappa r(z)>r(y),\ d(z,y)\geq \kappa^{-1}r(z),\ r(y)\leq t} \frac 1{V(z,t)}\frac{t}{r(y)}e^{-c\frac{r^2(z)}{t^2}} \left\vert \omega(y)\right\vert d\mu(y)\right)^2\frac{d\mu(z)}{V(z,t)}\frac{dt}t\right)^{\frac 12}\\
\le    \dsp\left(\iint_{d(z,x)\leq t}  \frac 1{V(z,t)^2}  e^{-c\frac{r^2(z)}{t^2}}\left(\int_{\kappa r(z)>r(y),\ d(z,y)\geq \kappa^{-1}r(z),\ r(y)\leq t} \frac{t}{r(y)}\left\vert \omega(y)\right\vert d\mu(y)\right)^2\frac{d\mu(z)}{V(z,t)}\frac{dt}t\right)^{\frac 12}
\end{array}
$$
By the H\"older inequality and Lemma \ref{integration} from the Appendix, since $p^{\prime}<\nu$,
$$
\begin{array}{l}
\dsp\int_{\kappa r(z)>r(y),\ d(z,y)\geq \kappa^{-1}r(z),\ r(y)\leq t} \frac{t}{r(y)}\left\vert \omega(y)\right\vert d\mu(y)\leq t\left\Vert \omega\right\Vert_p \left(\int_{r(y)\leq \min\left(\kappa r(z),t\right)} \frac 1{r(y)^{p^{\prime}}}d\mu(y)\right)^{\frac 1{p^{\prime}}}\\
\dsp\lesssim t\left\Vert \omega\right\Vert_p \min\left(\kappa r(z),t\right)^{-1} V\left(o,\min\left(\kappa r(z),t\right)\right)^{\frac 1{p^{\prime}}}.
\end{array}
$$
As a consequence,
\begin{equation} \label{Bs1}
\begin{array}{lll}
\dsp{\mathcal A}_{s,1}\omega(x) \lesssim  \dsp\left\Vert \omega\right\Vert_p \left(\iint_{d(z,x)\leq t} \frac {t^2}{V(z,t)^3} e^{-c\frac{r^2(z)}{t^2}}\left(\min\left(\kappa r(z),t\right)\right)^{-2}V\left(o,\min\left(\kappa r(z),t\right)\right)^{\frac 2{p^{\prime}}}d\mu(z)\frac{dt}t\right)^{\frac 12}\\
 \le  \dsp\left\Vert \omega\right\Vert_p \left(\iint_{d(z,x)\leq t,\ t\leq \frac{r(x)}2} \frac {t^2}{V(z,t)^3} e^{-c\frac{r^2(z)}{t^2}}\left(\min\left(\kappa r(z),t\right)\right)^{-2}V\left(o,\min\left(\kappa r(z),t\right)\right)^{\frac 2{p^{\prime}}}d\mu(z)\frac{dt}t\right)^{\frac 12}\\
 +   \dsp\left\Vert \omega\right\Vert_p \left(\iint_{d(z,x)\leq t,\ t>\frac{r(x)}2} \frac {t^2}{V(z,t)^3} e^{-c\frac{r^2(z)}{t^2}}\left(\min\left(\kappa r(z),t\right)\right)^{-2}V\left(o,\min\left(\kappa r(z),t\right)\right)^{\frac 2{p^{\prime}}}d\mu(z)\frac{dt}t\right)^{\frac 12}\\
=:  I_1+I_2.
\end{array}
\end{equation}
When $d(z,x)\leq t$ and $t\leq \frac{r(x)}2$, then $r(x)\leq r(z)+t\leq r(z)+\frac{r(x)}2$, and $r(z)\leq r(x)+t\leq \frac{3r(x)}2$, so that $r(x)\simeq r(z)$ and $\min\left(\kappa r(z),t\right)\simeq t$, which shows that
\begin{equation} \label{estimI1sl1}
\begin{array}{lll}
\dsp I_1& \lesssim & \dsp\left\Vert \omega\right\Vert_p \left(\iint_{d(z,x)\leq t,\ t\leq \frac{r(x)}2} \frac 1{V(z,t)^3}  e^{-c\frac{r^2(z)}{t^2}}V(o,t)^{\frac 2{p^{\prime}}}d\mu(z)\frac{dt}t\right)^{\frac 12}\\
& \lesssim & \dsp\left\Vert \omega\right\Vert_p \left(\int_0^{\frac{r(x)}2} \frac 1{V(x,t)^2}  e^{-c\frac{r^2(x)}{t^2}} V(o,t)^{\frac 2{p^{\prime}}}\frac{dt}t\right)^{\frac 12}\\
&\lesssim & \dsp\frac{\left\Vert \omega\right\Vert_p}{V(x,r(x))}V(o,r(x))^{\frac 1{p^{\prime}}} \left(\int_0^{\frac{r(x)}2} \left(\frac{r(x)}t\right)^D e^{-c\frac{r^2(x)}{t^2}} \frac{dt}t\right)^{\frac 12}\\
& \lesssim & \dsp\frac{\left\Vert \omega\right\Vert_p}{V(o,r(x))^{\frac 1p}}. 
\end{array}
\end{equation}
When $d(z,x)\leq t$ and $t>\frac{r(x)}2$, then $r(z)\leq r(x)+t\leq 3t$, so that $\min\left(\kappa r(z),t\right)\simeq r(z)$. Therefore,
\begin{equation} \label{estimI2sl1}
\begin{array}{lll}
\dsp I_2& \lesssim &\dsp \left\Vert \omega\right\Vert_p \left(\iint_{d(z,x)\leq t,\ t>\frac{r(x)}2} \frac {t^2}{V(z,t)^3}  e^{-c\frac{r^2(z)}{t^2}}r(z)^{-2}V(o,r(z))^{\frac 2{p^{\prime}}}d\mu(z)\frac{dt}t\right)^{\frac 12}\\
& \lesssim & \dsp\left\Vert \omega\right\Vert_p\left(\int_{\frac{r(x)}2}^{+\infty} \frac 1{V(x,t)^3} \left(\int_{d(z,x)\le t,\ r(z)\le 3t} \frac{t^2}{r(z)^2}e^{-c\frac{r^2(z)}{t^2}} V(o,r(z))^{\frac 2{p^{\prime}}}d\mu(z)\right)\frac{dt}t\right)^{\frac 12}\\
& \lesssim & \dsp\left\Vert \omega\right\Vert_p\left(\int_{\frac{r(x)}2}^{+\infty} \frac 1{V(x,t)^3} \left(\int_{d(z,x)\le t,\ r(z)\leq 3t} \frac{t^2}{r(z)^2}\left(\frac{r(z)}t\right)^{\frac {2\nu}{p^{\prime}}}e^{-c\frac{r^2(z)}{t^2}}d\mu(z)\right)V(o,t)^{\frac 2{p^{\prime}}}\frac{dt}t\right)^{\frac 12}\\
& \lesssim & \dsp\left\Vert \omega\right\Vert_p \left(\int_{\frac{r(x)}2}^{+\infty}  \frac 1{V(x,t)^2} V(o,t)^{\frac 2{p^{\prime}}}\frac{dt}t\right)^{\frac 12}\\
& \lesssim & \dsp\left\Vert \omega\right\Vert_p \left(\int_{\frac{r(x)}2}^{+\infty} V(x,t)^{-\frac 2p}\frac{dt}t\right)^{\frac 12}\\
& \lesssim & \dsp\left\Vert \omega\right\Vert_p V(x,r(x))^{-\frac 1p}  \left(\int_{\frac{r(x)}2}^{+\infty} \left(\frac{r(x)}t\right)^{2\frac{\nu}p}\frac{dt}t\right)^{\frac 12}\\
& \lesssim & \dsp\left\Vert \omega\right\Vert_p V(x,r(x))^{-\frac 1p},
\end{array}
\end{equation}
where the third line follows from the second one since $\frac{2\nu}{p^{\prime}}>2$ (this, in turn, is due to the fact that $p>\frac{\nu}{\nu-1}$). Thus, \eqref{Bs1}, \eqref{estimI1sl1} and \eqref{estimI2sl1} yield
$$
{\mathcal A}_{s,1}\omega(x)\lesssim \frac{\left\Vert \omega\right\Vert_p}{V(o,r(x))^{\frac 1p}}.
$$
Lemma \ref{weak} therefore ensures that \eqref{Bsweak1} holds.

\subsubsection{The case $t\leq r(y)$:}

In this case, following the argument in Section \ref{tlry} and using \eqref{estimksl} again, one obtains
\begin{equation} \label{estimBG2}
\begin{array}{lll}
\dsp{\mathcal A}_{s,2}\omega(x)   \le \dsp \left(\iint_{d(z,x)\leq t} \left(\int_{\kappa r(z)>r(y),\ d(z,y)\geq \kappa^{-1}r(z),\ t\leq r(y)} \left\vert k_t(z,y)\right\vert \left\vert \omega(y)\right\vert d\mu(y)\right)^2\frac{d\mu(z)}{V(z,t)}\frac{dt}t\right)^{\frac 12}\\
\lesssim  \dsp\left(\iint_{d(z,x)\leq t} \left(\int_{\kappa r(z)>r(y),\ d(z,y)\geq \kappa^{-1}r(z),\ t\leq r(y)}  \frac 1{V(z,t)}e^{-c\frac{r^2(z)}{t^2}}\left\vert \omega(y)\right\vert d\mu(y)\right)^2\frac{d\mu(z)}{V(z,t)}\frac{dt}t\right)^{\frac 12}\\
  \lesssim   \dsp\int_M \left\vert \omega(y)\right\vert\left(\int_0^{r(y)} \frac 1{V(z,t)^3} \left(\int_{d(z,x)\leq t,\ \kappa r(z)>r(y),\ d(z,y)\geq \kappa^{-1}r(z)} e^{-c\frac{r^2(z)}{t^2}}d\mu(z)\right)\frac{dt}t\right)^{1/2}d\mu(y)\\
 \lesssim  \dsp\int_M \left\vert \omega(y)\right\vert\left(\int_0^{r(y)} \frac 1{V(x,t)^3} \left(\int_{d(z,x)\leq t,\ \kappa r(z)>r(y),\ d(z,y)\geq \kappa^{-1}r(z)} e^{-c\frac{r^2(z)}{t^2}}d\mu(z)\right)\frac{dt}t\right)^{1/2}d\mu(y)\\
 =:  \dsp\int_M \left\vert \omega(y)\right\vert I(y)d\mu(y).
\end{array}
\end{equation}
When $\frac 12r(x)\leq r(y)$, then, using the doubling property, we simply estimate
\begin{equation} \label{estimIy1}
\begin{array}{lll}
\dsp I(y)& \lesssim & \dsp\left(\int_0^{r(y)} \frac 1{V(x,t)^2} e^{-c\frac{r^2(y)}{t^2}}\frac{dt}t\right)^{1/2}\\
& \lesssim & \dsp \frac 1{V(x,r(y))} \lesssim \frac 1{V(x,r(x))}\lesssim \frac 1{V(o,r(x))}.
\end{array}
\end{equation}
If $r(y)<\frac 12 r(x)$, then 
$$
r(x)\leq r(z)+d(x,z)\leq r(z)+t\leq r(z)+r(y)<r(z)+\frac 12 r(x),
$$
so that $r(x)\leq 2r(z)$. As a consequence,
\begin{equation} \label{estimIy2}
\begin{array}{lll}
\dsp I(y)& \lesssim & \dsp\left(\int_0^{r(y)} \frac 1{V(x,t)^2}  e^{-c\frac{r^2(x)}{t^2}}\frac{dt}t\right)^{1/2}\\
&\leq & \dsp\left(\int_0^{r(x)} \frac 1{V(x,t)^2}  e^{-c\frac{r^2(x)}{t^2}}\frac{dt}t\right)^{1/2}\\
&\lesssim &\dsp \frac 1{V(x,r(x))} \lesssim \frac 1{V(o,r(x))}.
\end{array}
\end{equation}
Gathering \eqref{estimBG2}, \eqref{estimIy1} and \eqref{estimIy2}, we obtain
$$
{\mathcal A}_{s,2}\omega(x)\lesssim \frac 1{V(o,r(x))}\left\Vert \omega\right\Vert_1,
$$
and we conclude as before that \eqref{Bsweak2} holds.

\subsection{The ``diagonal'' regime}

We now turn to the range $d(z,y)<\kappa^{-1}r(z)$.  As in \cite[Sections 4.3 and 5.3]{carron}, we will use the covering $(B_{\alpha})_{\alpha\in \N}$ of $M$ by admissible balls introduced in Section \ref{goodcover}, as well as the associated partition of unity. Let $\alpha\in \N$ and $y\in B_{\alpha}$. If $d(z,y)\leq \kappa^{-1}r(z)$, then, by \eqref{ralpha}, 
$$
\begin{array}{lll}
\dsp d(z,x_{\alpha})& \leq & \dsp d(z,y)+d(y,x_{\alpha})\\
& \leq & \dsp \kappa^{-1}r(z)+r_{\alpha}\\
& \leq & \dsp \kappa^{-1}d(z,x_{\alpha})+\kappa^{-1}r(x_{\alpha})+r_{\alpha}\\
& \leq & \dsp \kappa^{-1}d(z,x_{\alpha})+ (1+2^{10}\kappa^{-1}) r_{\alpha},
\end{array}
$$
and a short computation shows that since $\kappa\geq 2^{10}$ by assumption, one has $d(z,x_\alpha)\leq 4r_\alpha$, that is $z\in 4B_{\alpha}$. Therefore, 
$$
\begin{array}{lll}
\dsp {\mathcal A}_d\omega(x)  & \leq  & \dsp\sum_{\alpha} \left(\iint_{(z,t)\in \Gamma(x),\ z\in 4B_{\alpha}} \left\vert \int_{d(z,y)< \kappa^{-1}r(z)} k_t(z,y)\cdot (\chi_{\alpha}\omega)(y) d\mu(y)\right\vert^2 \frac{d\mu(z)}{V(z,t)}\frac{dt}t\right)^{\frac 12}\\
& =: & \dsp\sum_{\alpha} {\mathcal A}_{d,\alpha}(\omega)(x).
\end{array}
$$
Fix $\alpha\in \N$ and split
$$
\begin{array}{lll}
\dsp{\mathcal A}_{d,\alpha}(\omega)(x) \leq  \dsp \left(\iint_{(z,t)\in \Gamma(x),\ z\in 4B_{\alpha},\ t\geq r_{\alpha}} \left\vert\int_{d(z,y)<\kappa^{-1}r(z)} k_t(z,y)\cdot (\chi_{\alpha}\omega)(y)d\mu(y)\right\vert^2 \frac{d\mu(z)}{V(z,t)}\frac{dt}t\right)^{\frac 12}\\
\quad +  \dsp \left(\iint_{(z,t)\in \Gamma(x),\ z\in 4B_{\alpha},\ t<r_{\alpha}} \left(\int_{d(z,y)<\kappa^{-1}r(z)} k_t(z,y)\cdot (\chi_{\alpha}\omega)(y) d\mu(y)\right)^2\frac{d\mu(z)}{V(z,t)}\frac{dt}t\right)^{\frac 12}\\
 =:  {\mathcal A}_{d,\alpha,1}\omega(x)+{\mathcal A}_{d,\alpha,2}\omega(x).
\end{array}
$$

\subsubsection{The case $t\geq r_{\alpha}$: } 

We intend to prove that
\begin{equation} \label{bdweak1}
\mu\left(\left\{x\in M;\ \sum_{\alpha}{\mathcal A}_{d,\alpha,1}\omega(x)>\lambda\right\}\right)\lesssim \frac{\left\Vert \omega\right\Vert_1}{\lambda}.
\end{equation}
We use the upper bound
$$
\left\vert k_t(z,y)\right\vert\lesssim \frac 1{V(y,t)}\frac{t}{r(y)+1},
$$
which follows from Lemma \ref{estimkernel}. Indeed, note that, for all $\alpha$ and all $y\in B_{\alpha}$, $r(y)\lesssim r_{\alpha}$. As a consequence, $r(y)\lesssim t$ in the range under consideration. \par

As in section \ref{LSlarge_t}, one has
$$
{\mathcal A}_{d,\alpha,1}(\omega)(x)\leq \int_{y\in B_{\alpha}} \left\vert \omega(y)\right\vert I(y)d\mu(y),
$$
where 
$$
I(y)\lesssim \left(\int_{d(z,x)\leq t,\ t\geq r_{\alpha},\ z\in 4B_{\alpha}} \frac 1{V(y,t)^3}\left(\frac{t^2}{r_{\alpha}^2}\right)d\mu(z)\frac{dt}t\right)^{1/2}.
$$
Noticing that, for all $y\in B_{\alpha}$, $4B_{\alpha}\subset B(y,5r_{\alpha})$, we write
$$
\begin{array}{lll}
I(y)& \lesssim & \dsp\left(\int_{t\geq r_{\alpha}} \frac {V(y,5r_{\alpha})}{V(y,t)^3}\left(\frac{t^2}{r_{\alpha}^2}\right)\frac{dt}t\right)^{1/2}\\
& \lesssim & \dsp \frac 1{V(y,r_{\alpha})} \left(\int_{r_{\alpha}}^{+\infty} \frac {V(y,r_{\alpha})^3}{V(y,t)^3}\left(\frac{t^2}{r_{\alpha}^2}\right)\frac{dt}t\right)^{1/2}\\
& \lesssim & \dsp \frac 1{V(y,r_{\alpha})} \left(\int_{r_{\alpha}}^{+\infty} \left(\frac {r_{\alpha}}t\right)^{3\nu-2}\frac{dt}t\right)^{1/2}\\
& \lesssim & \dsp \frac 1{V(y,r_{\alpha})}.
\end{array}
$$
As a consequence,
$$
\begin{array}{lll}
\dsp {\mathcal A}_{d,\alpha,1}(\omega)(x) & \lesssim &\dsp \int_{y\in B_{\alpha}} \frac{\left\vert \omega(y)\right\vert}{V(y,r_{\alpha})}d\mu(y)\\
& \lesssim & \dsp \frac 1{V(B_{\alpha})}\left\Vert \omega\right\Vert_{L^1(B_\alpha)},
\end{array}
$$
so that
\begin{equation} \label{bdalpha1case1}
{\mathcal A}_{d,\alpha,1}(\omega)(x) \lesssim \frac 1{V(o,r(x)))}\left\Vert \omega\right\Vert_{L^1(B_\alpha)},
\end{equation}
whenever $r(x)\leq 2(4+2^{10})r_{\alpha}$.\par

Consider now the case where $r(x)>2(4+2^{10})r_{\alpha}$. Then, for all $z\in 4B_{\alpha}$, $r(z)\leq 4r_{\alpha}+r(x_{\alpha})\leq (4+2^{10})r_{\alpha}<\frac {r(x)}2$, so that $d(x,z)\geq \frac{r(x)}2$. As a consequence,
$$
\begin{array}{lll}
I(y) & \lesssim & \dsp \left(\int_{t\geq\frac 12r(x)} \frac {V(y,5r_{\alpha})}{V(y,t)^3}\left(\frac{t^2}{r_{\alpha}^2}\right)\frac{dt}t\right)^{1/2}\\
& \lesssim & \dsp \left(\frac{V(y,r_{\alpha})}{V(y,r(x))}\right)^{1/2}\left(\int_{t\geq\frac 12r(x)} \frac {V(y,r(x))}{V(y,t)^3}\left(\frac{t^2}{r_{\alpha}^2}\right)\frac{dt}t\right)^{1/2}\\
& \leq & \dsp \left(\frac{V(y,r_{\alpha})}{V(y,r(x))}\right)^{1/2}\frac{r(x)}{r_{\alpha}} \frac 1{V(y,r(x))} \left(\int_{t\geq\frac 12r(x)} \frac {V(y,r(x))^3}{V(y,t)^3}\left(\frac{t^2}{r(x)^2}\right)\frac{dt}t\right)^{1/2}\\
& \lesssim & \dsp \left(\frac{r_{\alpha}}{r(x)}\right)^{\frac{\nu}2-1} \frac 1{V(y,r(x))} \left(\int_{t\geq\frac 12r(x)} \left(\frac{r(x)}t\right)^{3\nu-2}\frac{dt}t\right)^{1/2}\\
& \lesssim & \dsp \frac 1{V(y,r(x))}\lesssim \frac 1{V(x,r(x))},
\end{array}
$$
where the last inequality follows from the fact that $d(x,y)\leq r(x)+r(y)\lesssim r(x)+r_{\alpha}\lesssim r(x)$, which entails $V(x,r(x))\lesssim V(y,r(x))$. As a consequence,
\begin{equation} \label{bdalpha1case2}
{\mathcal A}_{d,\alpha,1}(\omega)(x)  \lesssim  \frac 1{V(o,r(x))}\left\Vert \omega\right\Vert_{L^1(B_\alpha)}.
\end{equation}
Gathering \eqref{bdalpha1case1} and \eqref{bdalpha1case2}, and using the fact that the balls $(B_\alpha)_\alpha$ have the finite intersection property, we obtain

$$\sum_\alpha \mathcal{A}_{d,\alpha,1}(\omega)(x)\lesssim \frac{1}{V(o,r(x))}\Vert \omega\Vert_1.$$ 
By Lemma \ref{weak} from Appendix A, we obtain, for all $\lambda>0$,
$$
\mu\left(\left\{x;\ \sum_\alpha {\mathcal A}_{d,\alpha,1}(\omega)(x)>\lambda\right\}\right)\lesssim \frac 1{\lambda} \left\Vert \omega\right\Vert_1.
$$

\subsubsection{The case $t<r_{\alpha}$: } 

We now turn to the case of ${\mathcal A}_{d,\alpha,2}\omega$; we wish to prove that

\begin{equation}\label{eq:A2a}
\mu\left(\left\{x;\ \sum_\alpha {\mathcal A}_{d,\alpha,2}(\omega)(x)>\lambda\right\}\right)\lesssim \frac 1{\lambda} \left\Vert \omega\right\Vert_1.
\end{equation}
Decompose
$$
\begin{array}{l}
\dsp\omega=df=\sum_{\alpha\in A} \chi_{\alpha}df\\
\dsp=\sum_{\alpha\in A} \chi_{\alpha}d(f-f_{B_{\alpha}})\\
\dsp= \sum_{\alpha\in A} d(\chi_{\alpha}(f-f_{B_{\alpha}}))-\sum_{\alpha\in A} (f-f_{B_{\alpha}})d\chi_{\alpha}\\
\dsp=:\sum_{\alpha\in A} df_{\alpha}-\sum_{\alpha\in A}\eta_{\alpha}.
\end{array}
$$
Note that $f_{\alpha}$ and $\eta_{\alpha}$ are supported in $B_{\alpha}$. Moreover, we claim:
\begin{equation} \label{L1alpha}
\left\Vert d f_{\alpha}\right\Vert_{L^1}+\left\Vert \eta_{\alpha}\right\Vert_{L^1}+r_{\alpha}\left\Vert d^{\ast}\eta_{\alpha}\right\Vert_{L^1}\lesssim \left\Vert \omega\right\Vert_{L^1(B_{\alpha})}.
\end{equation}
Indeed, by Lemma \ref{poincadm},
$$
\left\Vert \eta_{\alpha}\right\Vert_{L^1(B_{\alpha})}\lesssim \frac 1{r_{\alpha}} \left\Vert f-f_{B_{\alpha}}\right\Vert_{L^1(B_{\alpha})}\lesssim \left\Vert df\right\Vert_{L^1(B_{\alpha})}=||\omega||_{L^1(B_\alpha)}.
$$
Thanks to the fact that $\chi_\alpha \omega=df_\alpha-\eta_\alpha$ this also entails that $\left\Vert df_{\alpha}\right\Vert_{L^1}\lesssim \left\Vert \omega\right\Vert_{L^1(B_{\alpha})}$. Finally,
$$
\begin{array}{lll}
\dsp d^{\ast}\eta_{\alpha} & = & \dsp d^{\ast}\left[\left(f-f_{B_{\alpha}}\right)d\chi_{\alpha}\right]\\
& = & \dsp \langle df,d\chi_{\alpha}\rangle+\left(f-f_{B_{\alpha}}\right)\Delta\chi_{\alpha},
\end{array}
$$
which, in view of \eqref{conditionchi} and Lemma \ref{poincadm}, entails that $r_{\alpha}\left\Vert d^{\ast}\eta_{\alpha}\right\Vert_{L^1}\lesssim \left\Vert \omega\right\Vert_{L^1(B_{\alpha})}$, completing the proof of \eqref{L1alpha}. \par

Let $T_\alpha$ be the operator defined by

$$ T_\alpha \omega(x)=\left(\int_{d(x,z)\leq t,\, t<r_\alpha}\chi_{4B_\alpha}(z)|td^*(t^2\Delta_1)^Ne^{-t^2\Delta_1}\omega(z)|^2\,\frac{d\mu(z)}{V(z,t)}\frac{dt}{t}\right)^{1/2}.$$
Note that if $d(x,z)\leq t$, $z\in 4B_\alpha$ and $t\leq r_\alpha$, then $x\in 5B_\alpha$; consequently, the support of $T_\alpha\omega $ is included in $5B_\alpha$, which is a remote ball for $\alpha\ne 0$ by \eqref{ralpha}. 

Clearly, one has

$$\Vert T_\alpha\Vert_{2\to 2}\leq \Vert \mathcal{A} \Vert_{2\to 2}\leq C,$$
where $C$ is independent of $\alpha$. Also, by \eqref{L1alpha}, 

\begin{equation}\label{eq:Ta}
\mathcal{A}_{2,d,\alpha}(df)(x)\leq T_\alpha(df_\alpha)(x)+T_\alpha(\eta_\alpha)(x).
\end{equation}
Since the support of $T_\alpha$ is included in $5B_\alpha$, the covering $(5B_\alpha)_{\alpha\ge 0}$ has the finite intersection property and in view of \eqref{eq:Ta}, in order to show \eqref{eq:A2a}, it is enough to prove the following pair of inequalities:
\begin{equation} \label{talpha1}
\mu\left(\left\{x\in 5B_\alpha;\ \left\vert T_{\alpha}(df_{\alpha})(x)\right\vert>\lambda\right\}\right)\lesssim \frac 1{\lambda}\left\Vert df_{\alpha}\right\Vert_{L^1}
\end{equation}
and
\begin{equation} \label{talpha2}
\mu\left(\left\{x\in 5B_\alpha;\ \left\vert T_{\alpha}(\eta_{\alpha})(x)\right\vert>\lambda\right\}\right)\lesssim \frac 1{\lambda}\left\Vert \omega\right\Vert_{L^1(B_{\alpha})},
\end{equation}

\subsubsection{The exact diagonal part}
According to Proposition \ref{pro:CZ} in Appendix B applied to $T_\alpha$, the inequality \eqref{talpha1} will follow from the following pair of inequalities: for every (admissible) sub-ball $B\subset 2B_\alpha$, and every function $u$ supported in $B$,

\begin{equation}\label{eq:1.1}
\left(\frac{1}{V(2^{j+1}B)}\int_{C_j(B)\cap 5B_\alpha}|T_\alpha(I-A_{r(B)})(du)|^2\right)^{1/2}\leq g(j)\frac{1}{V(B)}\int_B |du|,\quad j\geq 2,
\end{equation}
as well as

\begin{equation}\label{eq:1.2}
\left(\frac{1}{V(2^{j+1}B)}\int_{C_j(B)\cap 5B_\alpha}|A_{r(B)}(du)|^2\right)^{1/2}\leq g(j)\frac{1}{V(B)}\int_B |du|,\quad j\geq 1,
\end{equation}
where  $C_1(B):=4B$, $C_j(B):=2^{j+1}B\setminus 2^jB$ for all $j\ge 2$ and $A_{r(B)}$ is a smoothing operator to be defined and $\sum_{j=1}^\infty g(j)2^{jD}<\infty$.

\medskip


Fix $\alpha\in \N$,  let $B\subset 2 B_\alpha$ be a sub-ball, and let $r:=r(B)$. Given $t>0$, define the operator 

$$\psi(\Delta_1):=e^{-t\Delta_1}(I-e^{-r^2\Delta_1})^m,$$
where $m$ will be chosen big enough later. According to \cite[Equations (2.6) and (4.3)]{Auscher2007}, one has

$$\psi(\Delta_1)=\int_{\Gamma_\pm}e^{-z\Delta_1}\eta_\pm(z)\,dz,$$
where $\Gamma_{\pm}$ is the half-ray $\R_+e^{\pm i\left(\frac{\pi}2-\theta\right)}$ for a suitable $\theta\in \left(0,\frac{\pi}2\right)$ and $\eta_\pm(z)$ is a complex function satisfying the estimate

$$|\eta_\pm(z)|\lesssim \frac{1}{|z|+t}\inf\left(1,\frac{r^{2m}}{(|z|+t)^m}\right).$$
Using Lemma \ref{off-diag2} and following the argument in \cite[p.27-28]{Auscher2007}, one obtains

\begin{equation}\label{eq:est_psi}
\begin{array}{rcl}
||V(\cdot,\sqrt{t})^{1/2}\sqrt{t}d^* (t^2\Delta_1)^N \psi(\Delta_1)(du)||_{L^2(C_j(B))}  &\lesssim &\frac{1}{4^{jm}}\left(\frac{t}{4^jr^2}\right)^{D/2}\times \\\\
&\inf&\left(\left(\frac{t}{4^jr^2}\right)^{1/2},\left(\frac{4^j r^2}{t}\right)^{m-1/2}\right)||du||_{L^1(B)},
\end{array}
\end{equation}
for every $u$ supported in the admissible ball $B$. If one now lets

$$\varphi(\Delta_1):=e^{-t^2\Delta_1}(I-e^{-r^2\Delta_1})^m,$$
then by \eqref{eq:est_psi}, one gets

\begin{equation}\label{eq:est_phi}
||V(\cdot,t)^{1/2}td^* (t^2\Delta_1)^N\varphi(\Delta_1)(du)||_{L^2(C_j(B))} \lesssim 4^{-jm}\left(\frac{t}{2^jr}\right)^{D+1}\inf\left(1,\left(\frac{2^j r}{t}\right)^{2m}\right)||du||_{L^1(B)}.
\end{equation}

We now define the smoothing operator $A_{r(B)}$ by

$$A_{r(B)}=I-(I-e^{-r(B)^2\Delta_1})^m.$$
We need to check that \eqref{eq:1.1} and \eqref{eq:1.2} hold. In what follows, for simplicity we will simply write $r$ instead of $r(B)$. 

\medskip

{\bf Proof of \eqref{eq:1.2}:} this uses the estimate in Lemma \ref{off-diag1}. Indeed, we first notice that $A_{r(B)}$ is a linear combination of terms $e^{-kr^2\Delta_1}$, $k=1,\cdots,m$, and it suffices to treat independently each of these terms. In what follows, we will thus fix an integer $k$ between $1$ and $m$. For every $j\geq1$, letting $F=C_j(B)\cap 5B_\alpha$, one has

$$r(B)\lesssim r(x)+1,\quad \forall x\in F.$$
Indeed, the inequality is trivial for $\alpha=0$, and for $\alpha\neq 0$, $r(x)\geq r(x_{\alpha})-5r_{\alpha}\ge (2^9-5)r_{\alpha}\ge \frac{(2^9-5)}{2}r(B)$.  Consequently, by Lemma \ref{off-diag1}, for every $u$ with support in $B$ and $du\in L^1$,

$$||V(\cdot,r\sqrt{k})^{1/2}e^{-kr^2\Delta_1}(du)||_{L^2(C_j(B)\cap 5B_\alpha)}\lesssim e^{-c4^j}||du||_{L^1(B)}.$$
By doubling, if $x\in C_j(B)$, then

$$\begin{array}{rcl}
\dsp\frac{V(x_B,r)}{V(x,r\sqrt{k})}&\leq& \dsp\frac{ V(x,r(1+2^{j+1}))}{V(x,r\sqrt{k})}\\\\
&\lesssim& (1+2^j)^D\\
&\lesssim& 2^{D j}.
\end{array}$$
Hence,

$$V(B)^{1/2}||e^{-kr^2\Delta_1}(du)||_{L^2(C_j(B)\cap 5B_\alpha)}\lesssim 2^{jD/2} e^{-c4^j}||du||_{L^1(B)}.$$
Therefore,

$$V(B)^{1/2}\left(\frac{ V(B)}{ V(2^{j+1}B)}\right)^{1/2}||e^{-kr^2\Delta_1}(du)||_{L^2(C_j(B)\cap 5B_\alpha)}\lesssim 2^{jD/2} e^{-c4^j}||du||_{L^1(B)},$$
which implies that \eqref{eq:1.2} holds with $g(j)\simeq 2^{jD/2}e^{-c4^j}$.





\cqfd

\medskip

{\bf Proof of \eqref{eq:1.1}:} this uses \eqref{eq:est_phi}. Given $j\geq 2$, we write $ T_\alpha(1-A_{r(B)})(du)\le S_\alpha(du)+L_\alpha(du)$, where

$$S_\alpha \omega(x)=\left(\int_{d(x,z)\leq t,\,\,t\leq r_\alpha\wedge 2^{j-1}r}\chi_{4B_\alpha}(z)|td^*(t^2\Delta_1)^Ne^{-t^2\Delta_1}(I-e^{-r(B)^2\Delta_1})^m\omega(z)|^2\,\frac{d\mu(z)}{V(z,t)}\frac{dt}{t}\right)^{1/2}$$
and

$$L_\alpha \omega(x)=\left(\int_{d(x,z)\leq t,\,\,2^{j-1}r^\leq t\leq r_\alpha}\chi_{4B_\alpha}(z)|td^*(t^2\Delta_1)^Ne^{-t^2\Delta_1}(I-e^{-r(B)^2\Delta_1})^m\omega(z)|^2\,\frac{d\mu(z)}{V(z,t)}\frac{dt}{t}\right)^{1/2}$$
(of course, $L_\alpha$ is non-zero only if $2^{j-1}r\leq r_\alpha$). Let 

$$F_t(z):=|td^*(t^2\Delta_1)^Ne^{-t^2\Delta_1}(I-e^{-r(B)^2\Delta_1})^m(du)(z)|. $$
We need to estimate

$$\begin{array}{rcl}
I&:=&\dsp\left(\frac{1}{V(2^{j+1}B)}\int_{C_j(B)}|S_\alpha(du)(x)|^2\,d\mu(x)\right)^{1/2}\\\\
&\leq &\dsp \left(\frac{1}{V(2^{j+1}B)}\int_{C_j(B)}\int_{d(x,z)\leq t,\,\,t\leq 2^{j-1}r}|F_t(z)|^2\,\frac{d\mu(z)dt}{tV(z,t)}d\mu(x)\right)^{1/2}
\end{array}$$
as well as

$$\begin{array}{rcl}
II&:=&\dsp\left(\frac{1}{V(2^{j+1}B)}\int_{C_j(B)}|L_\alpha(du)(x)|^2\,d\mu(x)\right)^{1/2}\\\\
&\leq &\dsp \left(\frac{1}{V(2^{j+1}B)}\int_{C_j(B)}\int_{d(x,z)\leq t,\,\,2^{j-1}r\leq t}|F_t(z)|^2\,\frac{d\mu(z)dt}{tV(z,t)}d\mu(x)\right)^{1/2}.
\end{array}$$
According to \eqref{eq:est_phi}, one has

$$
||V(\cdot,t)^{1/2}F_t||_{L^2(C_j(B))}\lesssim \left\{\begin{array}{lll}
\dsp 4^{-jm}\left(\frac{t}{2^jr}\right)^{D+1}||du||_{L^1(B)},&& t\leq 2^jr.\\\\
\dsp 4^{-jm}\left(\frac{2^jr}{t}\right)^{2m-D-1}||du||_{L^1(B)},&&t\geq 2^jr.
\end{array}
\right.
$$
Furthermore, doubling and reverse doubling imply that if $z\in C_j(B)$ and $t\ge 2^jr$, then

$$\frac{V(z,t)}{V(x_B,r)}\geq \frac{V(z,t)}{V(z,2^{j+2}r)}\geq C\inf \left(\left(\frac{t}{2^jr}\right)^\nu,\left(\frac{t}{2^jr}\right)^D\right).$$
Hence,

$$
V(B)^{1/2}||F_t||_{L^2(C_j(B))}\lesssim \left\{\begin{array}{lll}
4^{-jm}\left(\frac{t}{2^j r}\right)^{\frac D2+1}||du||_{L^1(B)},&& t\leq 2^jr.\\\\
4^{-jm}\left(\frac{2^jr}{t}\right)^{2m-D+\frac{\nu}2-1}||du||_{L^1(B)},&&t\geq 2^jr.
\end{array}
\right.
$$
One deduces that

\begin{equation}\label{eq:F_t}
V(B)^{1/2}||F_t||_{L^2(C_j(B))}\lesssim \left\{\begin{array}{lll}
4^{-jm}\left(\frac{t}{2^j r}\right)||du||_{L^1(B)},&& t\leq 2^jr.\\\\
4^{-jm}\left(\frac{2^jr}{t}\right)^{2m-D+\frac{\nu}2-1}||du||_{L^1(B)},&&t\geq 2^jr.
\end{array}
\right.
\end{equation}
\medskip

\noindent {\bf Estimate of $I$:} for $t\leq 2^{j-1}r$ and $x\in C_j(B)$, $d(x,z)\leq t$ implies that $z$ belongs to $C_{j-1}(B)\cup C_j(B)\cup C_{j+1}(B)$. Furthermore, if $z$ is fixed, then the measure of the set

$$\{x\in C_j(B)\,;\,d(x,z)\leq t\}$$
is by definition at most $V(z,t)$. Therefore, by Fubini and \eqref{eq:F_t}, we obtain

$$\begin{array}{rcl}
V(B) \cdot I&\leq& \dsp\left(\frac{V(B)}{V(2^{j+1}B)}\int_0^{2^{j-1}r} V(B)||F_t||^2_{L^2(C_{j-1}(B)\cup C_{j}(B)\cup C_{j+1}(B))}\,\frac{dt}{t}\right)^{1/2}||du||_{L^1(B)}\\\\
&\lesssim & \dsp \left(2^{-j\nu}\int_0^{2^{j-1}r} 16^{-jm}\left(\frac{t}{2^{j}r}\right)^2\,\frac{dt}{t}\right)^{1/2}||du||_{L^1(B)}\\\\
&=& C2^{-j(2m+\nu)/2}||du||_{L^1(B)}.
\end{array}$$
Therefore,

$$I\lesssim 2^{-j(2m+\nu)/2}\cdot \frac{1}{V(B)}||du||_{L^1(B)}.$$

\medskip

\noindent {\bf Estimate of $II$:} let $t\geq 2^{j-1}r$,  and $i=i(t)$ be the  lowest integer such that 

$$2^{i}r\geq t.$$
Then, for $x\in C_j(B)$, $d(x,z)\leq t$ implies that $z$ belongs to $2^{i+3}B$. 
We bound the integral 

$$\int_{C_j(B)}\int_{d(x,z)\leq t}|F_t(z)|^2\,d\mu(z)d\mu(x)$$
by

$$V(2^{j+1}B)\cdot \sum_{k\leq i(t)+2} ||F_t||_{L^2(C_k(B))}^2.$$
By \eqref{eq:F_t}, we get

$$\begin{array}{rcl}
\dsp\int_{C_j(B)}\int_{d(x,z)\leq t}|F_t(z)|^2\,d\mu(z)d\mu(x)&\leq& \dsp \frac{V(2^{j+1}B)}{V(B)} \sum_{ k\leq i+2} V(B) \cdot ||F_t||^2_{L^2(C_k(B))}\\\\
&\lesssim & \dsp \frac{V(2^{j+1}B)}{V(B)} \left(\sum_{k\leq i+2}4^{-2km}\left(\frac{2^kr}{t}\right)^{4m-2D+\nu-2}\right)||du||^2_{L^1(B)}\\\\
&\lesssim&\dsp \frac{V(2^{j+1}B)}{V(B)}\left(\frac{r}{t}\right)^{4m-2D+\nu-2}\left(\sum_{k\leq i+2}4^{-k(D-\frac{\nu}2+1)}\right)||du||^2_{L^1(B)}\\\\
&\lesssim& \dsp V(2^{j+1}B) \left(\frac{r}{t}\right)^{4m-2D+\nu-2}\,\frac{1}{V(B)}\cdot ||du||^2_{L^1(B)},
\end{array}$$
where the last line follows from the fact that $D-\frac{\nu}2+1>0$ (since $\nu\leq D$). Also, if $z\in C_j(B)$ and $t\geq 2^{j-1}r$, by doubling

$$\frac{V(B)}{V(z,t)}\leq \frac{V(z,8t)}{V(z,t)}\lesssim 1.$$
Consequently,

$$\frac{1}{V(2^{j+1}B)}\int_{C_j(B)}\int_{d(x,z)\leq t}|F_t(z)|^2\,\frac{d\mu(z)d\mu(x)}{V(z,t)}\lesssim \left(\frac{r}{t}\right)^{4m-2D+\nu-2}\,\frac{1}{V(B)^2}\cdot ||du||^2_{L^1(B)}.$$
Therefore,

$$\begin{array}{rcl}
II&\lesssim & \dsp\left(\int_{2^{j-1}r}^\infty \left(\frac{r}{t}\right)^{4m-2D+\nu-2}\frac{dt}{t}\right)^{1/2}\frac{1}{V(B)}\cdot ||du||_{L^1(B)}\\\\
&\lesssim &\dsp 2^{-j(2m-D+\frac{\nu}2-1)}\frac{1}{V(B)}\cdot ||du||_{L^1(B)}.
\end{array}$$

\medskip

Collecting the estimates, and choosing $m>2D+1$, one gets \eqref{eq:1.1} with $g(j)\simeq 2^{-j(2m+\nu)/2}+2^{-j(2m-D+\frac{\nu}2-1)}\lesssim 2^{-jm}$. Since $m>2D+1$, one has

$$\sum_jg(j)2^{jD}<\infty,$$
which concludes the proof of \eqref{eq:1.1}.

\cqfd

\subsubsection{The non-exact diagonal part}
We now prove \eqref{talpha2}. One has

$$T_\alpha \eta_\alpha(x)=\left(\int_{d(x,z)\leq t<r_\alpha}\chi_{4B_\alpha}(z)|t(t^2\Delta_0)^Ne^{-t^2\Delta_0}d^*\eta_\alpha(z)|^2\,\frac{d\mu(z)}{V(z,t)}\frac{dt}{t}\right)^{1/2}.$$
Define the function $g_\alpha$ by

$$g_\alpha:=r_\alpha d^*\eta_\alpha,$$
so that $g_\alpha$ is supported in $B_\alpha$, and according to \eqref{L1alpha},

\begin{equation}\label{eq:g_alpha}
||g_\alpha||_1\lesssim ||\omega||_{L^1(B_\alpha)}.
\end{equation}
Then,

$$\begin{array}{rcl}
T_\alpha \eta_\alpha(x)&=&\dsp \left(\int_{ d(x,z)\leq t<r_\alpha}\chi_{4B_\alpha}(z)|(t^2\Delta_0)^Ne^{-t^2\Delta_0}g_\alpha(z)|^2\,\left(\frac{t}{r_\alpha}\right)^2 \frac{d\mu(z)}{V(z,t)}\frac{dt}{t}\right)^{1/2}\\\\
&\leq &\dsp \left(\int_{ d(x,z)\leq t<r_\alpha}|(t^2\Delta_0)^Ne^{-t^2\Delta_0}g_\alpha(z)|^2\,\left(\frac{t}{r_\alpha}\right)^2  \frac{d\mu(z)}{V(z,t)}\frac{dt}{t}\right)^{1/2}.
\end{array}$$
Thus, for $r>0$, we are led to consider the following non-tangential functional:

$$R_rg(x):=\left(\int_{d(x,z)\leq t\leq r}|(t^2\Delta_0)^Ne^{-t^2\Delta_0}g(z)|^2\, \left(\frac{t}{r}\right)^2 \frac{d\mu(z)}{V(z,t)}\frac{dt}{t}\right)^{1/2}.$$
We claim that there exists a constant $C$ independent of $r>0$ such that, for every $g\in L^1(M)$,
\begin{equation}\label{eq:R_r}
\mu(\{x\,;\,|R_rg(x)|>\lambda\})\leq C\frac{1}{\lambda}||g||_1,\quad\forall \lambda>0.
\end{equation}
This claim, together with \eqref{eq:g_alpha}, readily implies \eqref{talpha2}. 

\medskip

Let us first check that $R_r$ is bounded on $L^2$, with 

$$||R_r||_{2\to 2}\leq C,$$
independent of $r>0$. One has

$$\begin{array}{rcl}
\dsp\int_M|R_rg(x)|^2\,d\mu(x)&=&\dsp \int_M\int_{d(x,z)\leq t\leq r}|(t^2\Delta_0)^Ne^{-t^2\Delta_0}g(z)|^2\, \left(\frac{t}{r}\right)^2 \frac{d\mu(z)}{V(z,t)}\frac{dt}{t}\\\\
&=&\dsp \int_M\int_0^{r} |(t^2\Delta_0)^Ne^{-t^2\Delta_0}g(z)|^2\,d\mu(z) \left(\frac{t}{r}\right)^2 \frac{dt}{t}d\mu(x)\\\\
&=&\dsp \int_0^{r} ||(t^2\Delta_0)^Ne^{-t^2\Delta_0}g||^2_2\, \left(\frac{t}{r}\right)^2\frac{dt}{t}\\\\
&\leq& \dsp ||g||_2^2\left(\int_0^{r} \left(\frac{t}{r}\right)^2 \frac{dt}{t}\right)\\\\
&=&||g||_2^2.
\end{array}$$
On the other hand, one notices that for every $r>0$,

$$R_rg(x)\leq Rg(x):=\left(\int_{d(x,z)\leq t}|(t^2\Delta_0)^Ne^{-t^2\Delta_0}g(z)|^2\,\frac{d\mu(z)}{V(z,t)}\frac{dt}{t}\right)^{1/2}.$$
This inequality, as well as the Gaussian estimates satisfied by \Bk $(t^2\Delta_0)^Ne^{-t^2\Delta_0}$ \Bk allow one to show that the hypotheses of Theorem 1.1 in \cite{Auscher2007} for $p_0=1$ are satisfied for $R_r$ with the choice $A_{r(B)}=I-(I-e^{-r(B)^2\Delta_0})^m$, $m\gg1$, and with constants that are bounded independently of $r>0$. As a consequence of this theorem, \eqref{eq:R_r} holds.

\section*{Appendix A: two lemmata on volume}
The following two lemmata are of frequent use in the present work:
\begin{lemma} \label{integration}
Let $A>0$, $1\leq p<\nu$. Then
$$
\int_{r(y)\leq A} \frac 1{r(y)^p}d\mu(y)\lesssim A^{-p}V(o,A).
$$
\end{lemma}
\bp Splitting the integration domain into dyadic annuli, we get
$$
\begin{array}{lll}
\dsp\int_{r(y)\leq A} \frac 1{r(y)^p}d\mu(y) & = & \dsp\sum_{j\geq 0} \int_{2^{-j-1}A<r(y)\leq 2^{-j}A} \frac 1{r(y)^p}d\mu(y)\\
& \leq & \dsp A^{-p} \sum_{j\geq 0} 2^{(j+1)p}V(o,2^{-j}A)\\
& \lesssim & \dsp A^{-p} \sum_{j\geq 0} 2^{(j+1)p} 2^{-j\nu}V(o,A),
\end{array}
$$
which yields the result since $p<\nu$. 
\ep
\begin{lemma} \label{weak}
Let $m>0$ and $A\subset M$ be a measurable set such that, for all $x\in A$, $V(o,r(x))\leq m$. Then $\mu(A)\lesssim m$.
\end{lemma}
\bp Define $t:=\sup\left\{r>0;\ V(o,r)\leq m\right\}$ (note that $t$ is well-defined and $t>0$, since $\lim_{s\rightarrow 0} V(o,s)=0$ and $\lim_{s\rightarrow +\infty} V(o,s)=+\infty$). Since $B(o,t)=\displaystyle \bigcup_{k\geq 1} B\left(o,t-\frac 1k\right)$, $V(o,t)\leq m$. The assumption on $A$ means that, for all $x\in A$, $r(x)\leq t$, so that $A\subset B(o,2t)$. Therefore, $\mu(A)\leq V(o,2t)\lesssim m$. \ep
\section*{Appendix B: A Calder\'on-Zygmund decomposition localized in balls} \label{czloc}
Recall that ${\mathcal M}$ denotes the uncentered Hardy-Littlewood maximal function, given by
$$
{\mathcal M}u(x):=\sup_{B\ni x} \frac 1{V(B)}\int_B \left\vert u(y)\right\vert d\mu(y),
$$
where the supremum is taken over all open balls containing $x$.

\begin{lemma}\label{CZ}
Let $B$ be a ball in $M$, and $u\in C_0^\infty(B)$. Let $1\leq q<\infty$, and assume that the Poincar\'e inequality with exponent $q$ holds for any ball $7\tilde{B}$, where $\widetilde{B}\subset 2B$. Then, there exists a constant $C>0$ depending only on the doubling constant, with the following property: for all $\lambda>\displaystyle\left(\frac{C\left\Vert \nabla u\right\Vert_q^q}{V(B)}\right)^{\frac 1q}$, there exists a denumerable collection of balls $(B_i)_{i\ge 1}\subset 2B$, a denumerable collection of $C^1$ functions $(b_i)_{i\ge 1}$ and a Lipschitz function $g$ such that:
\begin{enumerate}
\item $\displaystyle  u=g+\sum_{i\ge 1}b_i$,
\item The support of $g$ is included in $B$, and $|\nabla g(x)|\lesssim \lambda$, for a.e. $x\in B$.
\item The support of $b_i$ is included in $B_i$, and 
$$\displaystyle \int_{B_i} |\nabla b_i|^q\lesssim \lambda^q V(B_i).$$
\item $\displaystyle \sum_{i\ge 1} V(B_i)\lesssim \lambda^{-q}\int |\nabla u|^q$.
\item There is a finite upper bound $N$ for the number of balls $B_i$ that have a non-empty intersection.
\end{enumerate}
\end{lemma}
\begin{proof}
Define
$$
\Omega:=\left\{x\in M;\ {\mathcal M}\left(\left\vert \nabla u\right\vert^q\right)(x)>\lambda^q\right\},
$$
which is an open subset of $M$, and set $F:=M\setminus \Omega$. We first claim that $\Omega \subset 2B$. Indeed, if $x\notin 2B$, and $\tilde{B}$ is a ball containing $x$ and intersecting the support of $\nabla u$ (hence intersecting $B$), then $B\subset 3\tilde{B}$, hence $V(B)\leq cV(\tilde{B})$ by doubling. Consequently,

$$\frac{1}{V(\tilde{B})}\int_{\tilde{B}}|\nabla u|^q\leq c \frac{1}{V(B)}\int_{B}|\nabla u|^q \le cC^{-q}\lambda^q,$$
hence, if $C\geq c^{1/q}$, one obtains

$$\frac{1}{V(\tilde{B})}\int_{\tilde{B}}|\nabla u|^q\leq \lambda^q.$$
Taking the supremum over all balls $\tilde{B}$ containing $x$, one gets

$$\mathcal{M}\left(\left\vert \nabla u\right\vert^q\right)(x)\le\lambda^q,$$
and consequently $x\notin \Omega$. Therefore, we have proved that $\Omega \subset 2B$.

For all $x\in \Omega$, let $r_x:=\frac{1}{10}d(x,M\setminus \Omega)$ and $B_x:=B(x,r_x)$, so that $B_x\subset \Omega$, and $\Omega=\bigcup_{x\in \Omega}B_x$. Since the radii of the balls $B_x$ are uniformly bounded, there exists a denumerable collection of points $(x_i)_{i\ge 1}\in \Omega$ such that the balls $B_{x_i}$ are pairwise disjoint and $\Omega= \bigcup_{i\ge 1} 5B_{x_i}$. For all $i$, write $s_i:=5r_{x_i}$ and let $B_i=B(x_i,s_i)$. Notice that $B_i\subset 2B$ for all $i$. Furthermore, the balls $\frac{1}{5}B_i$ being disjoint together with doubling entail that the covering by balls $B_i$ has the finite intersection property. And by construction also, $3B_i\cap F\neq \emptyset$ for every $i$. Therefore, if one lets $\underline{B}_i:=\frac{1}{5}B_i$ and $\bar{B}_i:=3B_i$, then the families of balls $(\underline{B}_i,B_i,\bar{B}_i)_i$ form a Whitney-type covering of $\Omega$ in the sense of Coifman and Weiss (\cite{CW}). Note that, for all $i,j\ge 1$, if $B_i\cap B_j\neq \emptyset$, then $\delta^{-1}s_i\leq  s_j\leq \delta s_i$ with $\delta=3$. Indeed, let $x\in B_i\cap B_j$. Then
$$
s_i=\frac 12d(x_i,M\setminus \Omega)\le \frac 12d(x_i,x)+\frac 12 d(x,M\setminus \Omega)\le \frac 12s_i+\frac 12d(x,M\setminus \Omega),
$$
so that 
$$
s_i\le d(x,M\setminus \Omega)\le d(x,x_j)+d(x_j,M\setminus \Omega)\le 3s_j,
$$
and exchanging the roles of $i$ and $j$ proves the claim. \par
\noindent Let $(\chi_i)_{i\geq 1}$ be a partition of unity of $\Omega$, subordinated to the covering $(B_i)_{i\ge 1}$, and such that $|\nabla \chi_i|\lesssim s_i^{-1}$. Then, define

$$b_i=(u-u_{B_i})\chi_i,$$
so that $b_i$ has support in $B_i$. We also let

$$g=u-\sum_{i\ge 1}b_i.$$
According to the proof of Prop. 1.1 in \cite{AC}, $g$ is a well-defined, locally integrable function on $M$. The Lebesgue differentiation theorem implies that $|\nabla u|\leq \lambda$ a.e. on $F$. Following the proof of Prop. 1.1 in \cite{AC}, and using the $L^q$ Poincar\'e inequality for the balls $B_i$ as well as $(2\delta+1)B_i=7B_i$, the points $4.$ as well as $5.$ are easily proved.
\end{proof}

\begin{proposition}\label{pro:CZ}
Let $B\subset M$ be an admissible ball. 
Let $T$ be a real-valued sublinear operator of strong type $(2,2)$. Assume that, for all balls $\widetilde{B}\subset B$, all $u\in C_0^\infty(B)$, and for all $j\ge 2$,

\begin{equation}\label{eq:Ap1.1}
\left(\frac{1}{V(2^{j+1}\widetilde{B})}\int_{C_j(\widetilde{B})\cap 5B}|T(I-A_{r(\widetilde{B})})(du)|^2\right)^{1/2}\leq g(j)\frac{1}{V(\widetilde{B})}\int_{\widetilde{B}} |du|,\quad j\geq 2,
\end{equation}
and, for all $j\ge 1$,
\begin{equation}\label{eq:Ap1.2}
\left(\frac{1}{V(2^{j+1}\widetilde{B})}\int_{C_j(\widetilde{B})\cap 5B}|A_{r(\widetilde{B})}(du)|^2\right)^{1/2}\leq g(j)\frac{1}{V(\widetilde{B})}\int_{\widetilde{B}} |du|,\quad j\geq 1,
\end{equation}
where $(A_r)_{r>0}$ is a collection of operators acting on $1$-differential forms and $\sum_{j=1}^\infty g(j)2^{Dj}<\infty$, where $D>0$ is the doubling exponent from \eqref{VD}. Then,

$$\mu\left(\left\{x\in 5B\,; \left\vert T(du)(x)\right\vert>\lambda\right\}\right)\lesssim \frac{\left\Vert \nabla u\right\Vert_1}{\lambda},\quad u\in C_0^\infty(B).$$
\end{proposition}
\begin{proof}
Let $\lambda>0$. If $\lambda\le \displaystyle\frac{C\left\Vert \nabla u\right\Vert_1}{V(B)}$, where $C$ is given by Lemma \ref{CZ}, then by doubling
$$
\mu( \left\{x\in 5B; \left\vert T(du)(x)\right\vert >\lambda\right\})\leq V(5B)\lesssim V(B)\le \frac{C\left\Vert \nabla u\right\Vert_1}{\lambda}.
$$
Assume now that $\lambda>\displaystyle \frac{C\left\Vert \nabla u\right\Vert_1}{V(B)}$. Decompose $u=g+b$ with $g$ and $b$ given by Lemma \ref{CZ} applied with $q=1$. One has
\begin{eqnarray*} \label{tdu}
\mu( \left\{x\in 5B;\ \left\vert T(du)(x)\right\vert>\lambda\right\})
&\le & \mu\left( \left\{x\in 5B;\ \left\vert T(dg)(x)\right\vert>\frac{\lambda}2\right\}\right)\\
&&+\mu\left( \left\{x\in 5B;\ \left\vert T(db)(x)\right\vert>\frac{\lambda}2\right\}\right).
\end{eqnarray*}
For the first term of the right-hand side of \eqref{tdu}, write
$$
\begin{array}{lll}
\displaystyle \mu\left( \left\{x\in 5B;\ \left\vert T(dg)(x)\right\vert>\frac{\lambda}2\right\}\right) & \le & \displaystyle \frac 4{\lambda^2} \left\Vert T(dg)\right\Vert_2^2\\
& \le & \displaystyle \frac C{\lambda^2} \left\Vert dg\right\Vert_2^2\\
& \le & \displaystyle \frac C{\lambda^2}\left\Vert dg\right\Vert_{\infty} \left\Vert dg\right\Vert_1\\
& \le & \displaystyle \frac C{\lambda} \left\Vert du\right\Vert_1.
\end{array}
$$
The second line follows from the $L^2$-boundedness of $T$ and the last line is due to property $2$ of Lemma \ref{CZ} and the fact (due in turn to items $3$ and $4$ of Lemma \ref{CZ}) that
$$
\left\Vert dg\right\Vert_1\le \left\Vert du\right\Vert_1+\left\Vert db\right\Vert_1\le \left\Vert du\right\Vert_1+\sum_i \left\Vert db_i\right\Vert_1\lesssim \left\Vert du\right\Vert_1+\lambda\sum_i V(B_i)\lesssim\left\Vert du\right\Vert_1.
$$
As far as the second term in the right-hand side of \eqref{tdu} is concerned, write
$$
\left\vert T\left(\sum_i db_i\right)\right\vert \le \sum_i \left\vert T(I-A_{r_i})db_i\right\vert+\left\vert T\left(\sum_i A_{r_i}db_i\right)\right\vert.
$$
This entails that it is enough to check that
\begin{equation} \label{estimI}
I:=\mu\left( \left\{x\in 5B;\ \sum_i \left\vert T(I-A_{r_i})db_i(x)\right\vert>\frac{\lambda}4\right\}\right)\le \frac{C}{\lambda} \left\Vert du\right\Vert_1
\end{equation}
and
\begin{equation} \label{estimJ}
J:=\mu\left(\left\{x\in 5B;\ \left\vert T\left(\sum_i A_{r_i}db_i\right)(x)\right\vert>\frac{\lambda}4\right\}\right)\le \frac{C}{\lambda} \left\Vert du\right\Vert_1.
\end{equation}
For $I$,
$$
I\le \mu\left( \bigcup_{i\ge 1} 4B_i\right)+\mu\left( \left\{x\in 5B\setminus \bigcup_{i} 4B_i;\ \sum_{i\ge 1}\left\vert T(I-A_{r_i})db_i(x)\right\vert>\frac{\lambda}4\right\}\right)=:I_1+I_2.
$$
On the one hand, by doubling and property $4.$ of Lemma \ref{CZ},
\begin{equation} \label{estimI1ter}
I_1\le C\sum_i V(B_i)\le \frac C{\lambda} \left\Vert \nabla u\right\Vert_1.
\end{equation}
On the other hand, the Chebyshev inequality entails
\begin{equation} \label{I2}
I_2\le \frac{16}{\lambda^2}\left\Vert \sum_i{\bf 1}_{5B\setminus \bigcup_{i} 4B_i}T(I-A_{r_i})db_i\right\Vert_{L^2(5B)}^2.
\end{equation}
Pick up a function $h\in L^2(5B)$ with $\left\Vert h\right\Vert_2=1$. One has
$$
\left\vert \int_{5B} \sum_i{\bf 1}_{5B\setminus \bigcup_{i} 4B_i}T(I-A_{r_i})db_i(x)h(x)d\mu(x)\right\vert \le \sum_{i} \sum_{j\ge 2} A_{ij},
$$
where 
$$
A_{ij}:=\int_{5B\cap C_j(B_i)} \left\vert T(I-A_{r_i})db_i(x)\right\vert \left\vert h(x)\right\vert d\mu(x).
$$
By assumption \eqref{eq:Ap1.1} and property $3$ of Lemma \ref{CZ},
$$
\left\Vert T(I-A_{r_i})db_i\right\Vert_{L^2(5B\cap C_j(B_i))}\le V(2^{j+1}B_i)^{\frac 12} g(j)\frac 1{V(B_i)}\left\Vert db_i\right\Vert_1\le V(2^{j+1}B_i)^{\frac 12} g(j) \lambda.
$$
Moreover, for all $y\in B_i$,
$$
\left\Vert h\right\Vert_{L^2(5B\cap C_j(B_i))}\le \left\Vert h\right\Vert_{L^2(2^{j+1}B_i)}\le V(2^{j+1}B_i)^{\frac 12} \left({\mathcal M}\left\vert h\right\vert^2(y)\right)^{\frac 12}.
$$
As a consequence,
$$
A_{ij}\le V(2^{j+1}B_i)g(j)\lambda \left({\mathcal M}\left\vert h\right\vert^2(y)\right)^{\frac 12}\le C2^{jD}V(B_i)g(j) \lambda \left({\mathcal M}\left\vert h\right\vert^2(y)\right)^{\frac 12},
$$
and since this holds for all $y\in B_i$,
$$
A_{ij}\le C2^{jD} g(j) \lambda \int_{B_i} \left({\mathcal M}\left\vert h\right\vert^2(y)\right)^{\frac 12}d\mu(y).
$$
It follows that $\left\vert \int_{5B} \sum_i{\bf 1}_{5B\setminus \bigcup_{i} 4B_i}T(I-A_{r_i})db_i(x)h(x)d\mu(x)\right\vert$ is bounded by 
$$
\begin{array}{lll}
\displaystyle C\lambda\sum_{j\ge 2}2^{jD} g(j)\left(\sum_i \int_{B_i} \left({\mathcal M}\left\vert h\right\vert^2(y)\right)^{\frac 12}d\mu(y)\right) & \le & \displaystyle  CN\lambda\int_{\bigcup_i B_i} \left({\mathcal M}\left\vert h\right\vert^2(y)\right)^{\frac 12}d\mu(y)\\
& \le & \displaystyle CN\lambda\left\vert  \bigcup_i B_i\right\vert^{\frac 12} \left\Vert \left\vert h\right\vert^2\right\Vert_1^{\frac 12}\\
& \le & \displaystyle C\lambda^{1/2}\left\Vert \nabla u\right\Vert_1^{1/2},
\end{array}
$$
and \eqref{I2} shows that
\begin{equation} \label{estimI2ter}
I_2\le \frac C{\lambda} \left\Vert \nabla u\right\Vert_1.
\end{equation}
Gathering \eqref{estimI1ter} and \eqref{estimI2ter} yields \eqref{estimI}.\par
\noindent For $J$, the $L^2$-boundedness of $T$ gives
$$
J\le \frac{16}{\lambda^2}\left\Vert T\left(\sum_i A_{r_i}db_i\right)\right\Vert_{L^2(5B)}^2\le \frac{C}{\lambda^2} \left\Vert \sum_i A_{r_i}db_i\right\Vert_{L^2(M)}^2.
$$
Pick up again a function $h\in L^2(M)$ with $\left\Vert h\right\Vert_2=1$. Then
$$
\begin{array}{lll} 
\displaystyle \left\vert \int_M \left(\sum_i A_{r_i}db_i(x)\right)h(x)d\mu(x)\right\vert & \le & \displaystyle \sum_i \left(\sum_{j\ge 1} \int_{C_j(B_i)} \left\vert A_{r_i}db_i(x)\right\vert \left\vert h(x)\right\vert d\mu(x)\right):=\sum_i \sum_{j\ge 1} B_{ij}.
\end{array}
$$
For all $i,j$, using \eqref{eq:Ap1.2}, one obtains
$$
\begin{array}{lll}
\displaystyle B_{ij} & \le & \displaystyle V(2^{j+1}B_i)^{1/2}g(j)\frac 1{V(B_i)}\left\Vert db_i\right\Vert_{1}\\
& \le & \displaystyle CV(2^{j+1}B_i)^{1/2}g(j)\lambda,
\end{array}
$$
and arguing as before, we conclude that \eqref{estimJ} holds.
\end{proof}

\bibliographystyle{plain}               
\bibliography{hardy-quadratic}  

\begin{thebibliography}{10}

\bibitem{AO}
J.~Assaad and E.~M. Ouhabaz.
\newblock Riesz transforms of {S}chr\"{o}dinger operators on manifolds.
\newblock {\em J. Geom. Anal.}, 22(4):1108--1136, 2012.

\bibitem{Auscher2007}
P.~Auscher.
\newblock On necessary and sufficient conditions for ${L}^p$-estimates of
  {R}iesz transforms associated to elliptic operators on {$\mathbb{R}^n$} and
  related estimates.
\newblock {\em Mem. Amer. Math. Soc.}, 186(871):75 pp, 2007.

\bibitem{AC}
P.~Auscher and T.~Coulhon.
\newblock Riesz transform on manifolds and {P}oincar\'e inequalities.
\newblock {\em Ann. Scuola Norm. Sup. Pisa Cl. Sci. (5)}, IV(871), 2005.

\bibitem{AMM}
P.~Auscher, A.~McIntosh, and A~J. Morris.
\newblock {Calder\'on reproducing formulas and applications to Hardy spaces.}
\newblock {\em {Rev. Mat. Iberoam.}}, 31(3):865--900, 2015.

\bibitem{amr}
P.~Auscher, A.~McIntosh, and E.~Russ.
\newblock {Hardy spaces of differential forms on Riemannian manifolds}.
\newblock {\em J. Geom. Anal.}, 18(1):192--248, 2008.

\bibitem{BB}
I.~{Bailleul} and F.~{Bernicot}.
\newblock {Heat semigroup and singular PDEs. With an appendix by F. Bernicot
  and D. Frey.}
\newblock {\em {J. Funct. Anal.}}, 270(9):3344--3452, 2016.

\bibitem{bakry1}
D.~Bakry.
\newblock {Transformations de Riesz pour les semi-groupes sym\'etriques,
  Seconde partie: \'etude sous la condition $\Gamma_2\geq 0$}.
\newblock {\em S\'eminaire de Probabilit\'es XIX, Lecture Notes in Math.},
  1123:145--174, 1985.

\bibitem{bakry2}
D.~Bakry.
\newblock {Etude des transformations de Riesz dans les vari{\'e}t{\'e}s
  riemanniennes {\`a} courbure de Ricci minor{\'e}e}.
\newblock {\em S\'eminaire de Probabilit\'es XXI, Lecture Notes in Math.},
  1247:137--172, 1987.

\bibitem{BGL}
D.~Bakry, I.~Gentil, and M.~Ledoux.
\newblock {\em Analysis and geometry of {M}arkov diffusion operators}, volume
  348 of {\em Grundlehren der Mathematischen Wissenschaften [Fundamental
  Principles of Mathematical Sciences]}.
\newblock Springer, Cham, 2014.

\bibitem{Buser}
P.~{Buser}.
\newblock {A note on the isoperimetric constant.}
\newblock {\em {Ann. Sci. \'Ec. Norm. Sup\'er. (4)}}, 15:213--230, 1982.

\bibitem{carron2}
G.~{Carron}.
\newblock {Formes harmoniques \(L^ 2\) sur les vari\'et\'es non-compactes}.
\newblock {\em {Rend. Mat. Appl., VII. Ser.}}, 21(1-4):87--119, 2001.

\bibitem{carron}
G.~Carron.
\newblock Riesz transform on manifolds with quadratic curvature decay.
\newblock {\em Rev. Mat. Iberoam.}, 33(3):749--788, 2017.

\bibitem{CCH}
G.~Carron, T.~Coulhon, and A.~Hassell.
\newblock {Riesz transform and $L^p$-cohomology for manifolds with Euclidean
  ends}.
\newblock {\em Duke Math. J.}, 133(1):59--94, 2006.

\bibitem{CC}
J.~Cheeger and T.~H. Colding.
\newblock Lower bounds on {R}icci curvature and the almost rigidity of warped
  products.
\newblock {\em Ann. of Math.}, 144(1):189–237, 1996.

\bibitem{CW}
Ronald~R. {Coifman} and Guido {Weiss}.
\newblock {Extensions of Hardy spaces and their use in analysis.}
\newblock {\em {Bull. Am. Math. Soc.}}, 83:569--645, 1977.

\bibitem{cms}
R.R. Coifman, Y.~Meyer, and E.M. Stein.
\newblock {Some new function spaces and their applications to harmonic
  analysis}.
\newblock {\em J. Funct. Anal.}, 62(2):304--335, 1985.

\bibitem{cdgeq2}
T.~Coulhon and X.T. Duong.
\newblock Riesz transform and related inequalities on noncompact {R}iemannian
  manifolds.
\newblock {\em Comm. Pure Appl. Math.}, 56(12):1728--1751, 2003.

\bibitem{davies}
E.~B. {Davies}.
\newblock {Gaussian upper bounds for the heat kernels of some second-order
  operators on Riemannian manifolds}.
\newblock {\em {J. Funct. Anal.}}, 80(1):16--32, 1988.

\bibitem{D}
E.~B. Davies.
\newblock {Non-Gaussian aspects of heat kernel behavior}.
\newblock {\em J. London Math. Soc.}, 97(2):105--125, 1997.

\bibitem{Dev}
B.~Devyver.
\newblock {Hardy spaces and heat kernel regularity.}
\newblock {\em {Potential Anal.}}, 48(1):1--33, 2018.

\bibitem{Dev4}
Baptiste Devyver.
\newblock On the finiteness of the {M}orse index for {S}chr\"{o}dinger
  operators.
\newblock {\em Manuscripta Math.}, 139(1-2):249--271, 2012.

\bibitem{Dev2}
Baptiste {Devyver}.
\newblock {Noyau de la chaleur et transform\'ee de Riesz des op\'erateurs de
  Schr\"odinger.}
\newblock {\em {Ann. Inst. Fourier}}, 69(2):457--513, 2019.

\bibitem{grigobook}
A.~{Grigor'yan}.
\newblock {\em {Heat kernel and analysis on manifolds.}}, volume~47.
\newblock Providence, RI: American Mathematical Society (AMS); Somerville, MA:
  International Press, 2009.

\bibitem{GSC}
A.~Grigor'yan and L.~{Saloff-Coste}.
\newblock {Stability results for Harnack inequalities.}
\newblock {\em {Ann. Inst. Fourier}}, 55(3):825--890, 2005.

\bibitem{GH}
Colin {Guillarmou} and Andrew {Hassell}.
\newblock {Resolvent at low energy and Riesz transform for Schr\"odinger
  operators on asymptotically conic manifolds. I.}
\newblock {\em {Math. Ann.}}, 341(4):859--896, 2008.

\bibitem{H}
Juha Heinonen.
\newblock {\em Lectures on analysis on metric spaces}.
\newblock Universitext. Springer-Verlag, New York, 2001.

\bibitem{KP}
K.~Kroencke and O.~L. Petersen.
\newblock Long-time estimates for heat flows on ale manifolds.
\newblock {\em arXiv preprint arXiv:2006.06662}, 2020.

\bibitem{LY}
Peter {Li} and Shing~Tung {Yau}.
\newblock {On the parabolic kernel of the Schr\"odinger operator.}
\newblock {\em {Acta Math.}}, 156:154--201, 1986.

\bibitem{s}
A.~Sikora.
\newblock {Riesz transform, Gaussian bounds and the method of wave equation}.
\newblock {\em Math. Z.}, 247(3):643--662, 2004.

\end{thebibliography}
\end{document}